\documentclass{bmcart}

\usepackage{xr}
\usepackage{amsthm,amsmath,amsfonts}
\usepackage{graphicx}
\usepackage{epstopdf}
\usepackage{algorithmic}
\usepackage{enumitem}
\usepackage{rotating}
\usepackage{multirow}
\usepackage{xcolor}
\usepackage{xr}
\usepackage{enumitem}
\usepackage{array}
\newcolumntype{P}[1]{>{\centering\arraybackslash}p{#1}}
\newcolumntype{M}[1]{>{\centering\arraybackslash}m{#1}}
\usepackage[normalem]{ulem}

\usepackage[utf8]{inputenc} 

\startlocaldefs
\newcommand{\subscript}[2]{$#1 _ #2$}

\renewcommand*\env@matrix[1][*\c@MaxMatrixCols c]{%
  \hskip -\arraycolsep
  \let\@ifnextchar\new@ifnextchar
  \array{#1}}
\makeatother

\newcommand\fref[1]{Fig~\ref{#1}}

\newcommand\Fref[1]{Figure~\ref{#1}}

\newcommand\eqdef{\stackrel{{\normalfont\mbox{\footnotesize{def}}}}{=}}
\newtheorem{theorem}{Theorem}
\newtheorem{lemma}[theorem]{Lemma}
\theoremstyle{definition}
\newtheorem{definition}{Definition}[section]
\newtheorem{remark}{Remark}[section]
\endlocaldefs

\begin{document}

\begin{frontmatter}

\begin{fmbox}
\dochead{Research}


\title{Auditory streaming emerges from fast excitation and slow delayed inhibition}


\author[
  addressref={aff1},                   
  corref={aff1},                       
  email={A.A.Ferrario@exeter.co.uk}   
]{\inits{A.}\fnm{Andrea} \snm{Ferrario}}
\author[
  addressref={aff1},
  email={J.A.Rankin@exeter.ac.uk}
]{\inits{J.}\fnm{James} \snm{Rankin}}


\address[id=aff1]{
  \orgdiv{Department of Mathematics, College of Engineering, Mathematics \& Physical Sciences},             
  \orgname{University of Exeter},          
  \city{Exeter},                              
  \cny{UK}                                    
}



\end{fmbox}


\begin{abstractbox}
\begin{abstract} 
In the auditory streaming paradigm alternating sequences of pure tones can be perceived as a single galloping rhythm (integration) or as two sequences with separated low and high tones (segregation). Although studied for decades, the neural mechanisms underlining this perceptual grouping of sound remains a mystery. With the aim of identifying a plausible minimal neural circuit that captures this phenomenon, we propose a firing rate model with two periodically forced neural populations coupled by fast direct excitation and slow delayed inhibition. By analyzing the model in a non-smooth, slow-fast regime we analytically prove the existence of a rich repertoire of dynamical states and of their parameter dependent transitions. We impose plausible parameter restrictions and link all states with perceptual interpretations. Regions of stimulus parameters occupied by states linked with each percept matches those found in behavioral experiments. Our model suggests that slow inhibition masks the perception of subsequent tones during segregation (forward masking), while fast excitation enables integration for large pitch differences between the two tones. 
\end{abstract}




\end{abstractbox}
%

\end{frontmatter}



\section{Introduction}
Understanding how our perceptual system encodes multiple objects
simultaneously is an open challenge in sensory neuroscience. In a
busy room we can separate out a  voice of interest from other
voices and ambient sound (\emph{cocktail party problem})
\cite{cherry1953some,bizley2013and}. Theories of feature discrimination developed with mathematical models are based on evidence that different neurons respond to different stimulus features (e.g. visual orientation \cite{hubel1962receptive,Ben1995,Bressloff2001,rankin2017neural}). Primary auditory cortex (ACx) has a topographic map of sound frequency (tonotopy):  a gradient of locations  preferentially responding to frequencies from low to high \cite{romani1982tonotopic,da2011human}. However, feature separation alone cannot account for the auditory system segregating objects  overlapping or interleaved in time (e.g. melodies, voices). Understanding the role of temporal neural mechanisms in perceptual segregation presents an interesting modelling challenge where the same neural populations represent different percepts through temporal encoding.

\subsection{Auditory streaming and auditory cortex} In the auditory system sequences of sounds (streams) that are close in feature space (e.g. frequency) and interleaved in time lead to multiple perceptual interpretations. The so-called auditory streaming paradigm \cite{van1975temporal,bizley2013and} consists of interleaved sequences of tones A and B, separated by a difference in tone frequency (called $df$) and repeating in an ABABAB\ldots pattern (\Fref{fig:figure1}A).  This can be perceived as one integrated stream with an alternating rhythm (Integrated in \Fref{fig:figure1}B) or as two segregated streams (Segregated in \Fref{fig:figure2}B). When $df$ is small we hear integrated and when df is large we hear segregated, but at an intermediate range, which also depends on presentation rate $P\!R$, both percepts are possible (\Fref{fig:figure1}C). In this region of $(\text{df},{P\!R})$ parameter space bistability occurs, where perception switches between integrated and segregated every 2--15\,s~\cite{pressnitzer2008perceptual}. The curve separating integration and bistability is called the fission boundary, while the curve separating bistability and segregation is called coherence boundary \cite{van1975temporal} (\Fref{fig:figure1}C).

\begin{figure}[htbp]
  \centering
  \includegraphics[width=0.9\linewidth]{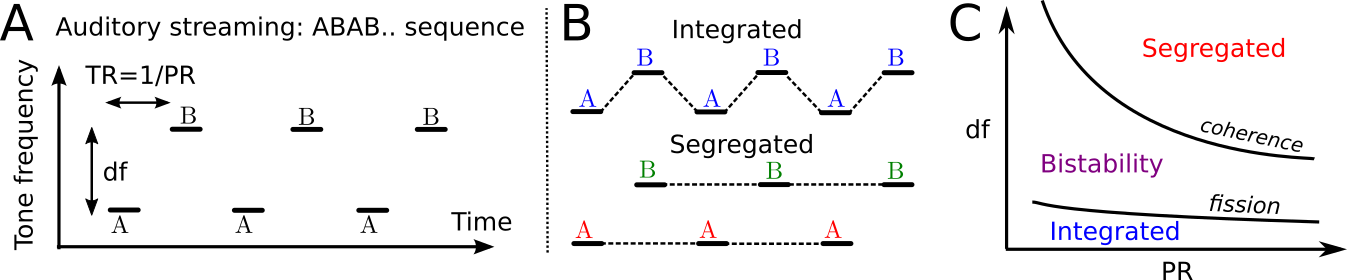}
  \caption{The auditory streaming paradigm (A) Auditory stimuli consist of sequences of interleaved higher pitch A and lower pitch B pure tones with duration $T\!D$, pitch difference $df$ and time difference between tone onsets $T\!R$ (the repetition time; $P\!R\!=\!1/T\!R$ is the repetition rate). (B) The stimulus may be perceived as either an integrated ABAB stream or as two separate streams A-A- and -B-B. (C) Sketch of the perceptual regions when varying $P\!R$ and $d\!f$ (van Noorden diagram), redrawn after \cite{van1975temporal}. Bistability corresponds to the perception of temporal switches between integration and segregation. The curves in the $(P\!R,df)$ space separating integration from bistability and bistability from segregation are called fission and coherence boundaries.}
  \label{fig:figure1}
\end{figure}

\Fref{fig:figure2}A shows our proposal for the encoding of auditory streaming. We follow the hypothesis proposed by \cite{musacchia2014thalamocortical}, where primary and secondary ACx encode respectively perception of the rhythm and the pitch. In our proposed framework the processing of auditory stimuli occurs firstly in primary ACx, which projects to secondary ACx. The various rhythms occurring in the auditory streaming paradigm arise via threshold-crossing detection in the activity of neural populations in secondary ACx. The process underlying bistability is likely resolved downstream of early auditory cortex \cite{rankin2019computational} and will not be addressed in this study.

\begin{figure}[htbp]
  \centering
  \includegraphics[width=0.9\linewidth]{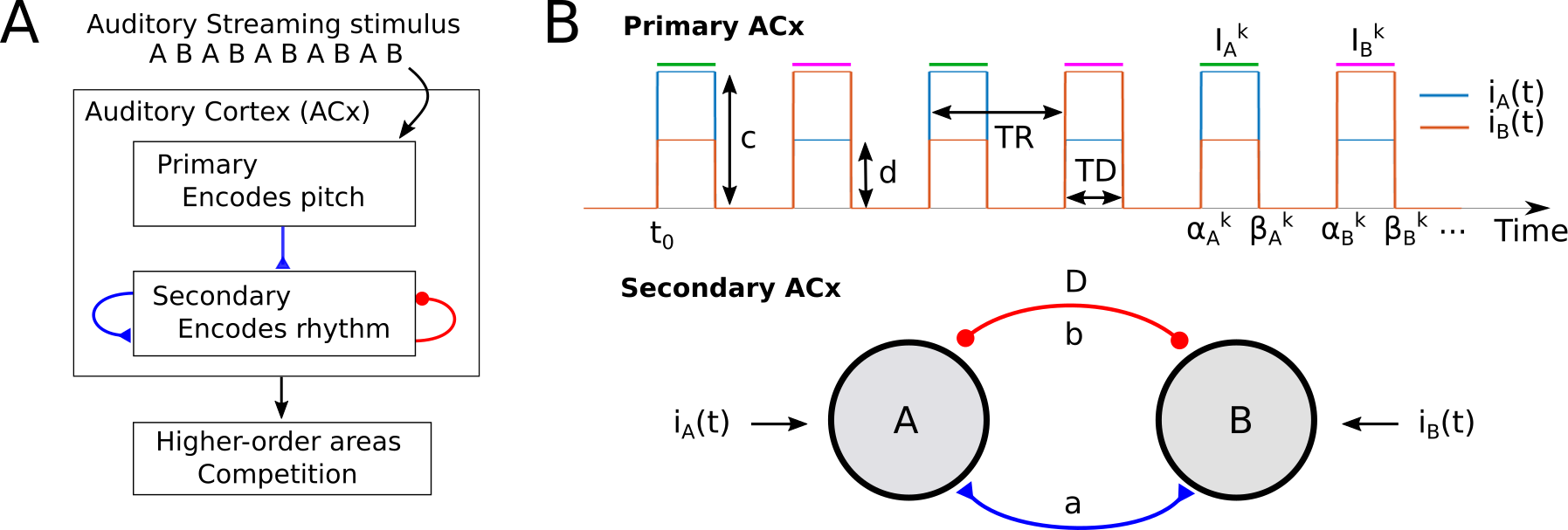}
  \caption{(A) Proposed modelling framework of the auditory streaming paradigm. Two-tone streams are processed in primary ACx. Seconday ACs receives inputs from primary areas and has recurrent excitatory and inhibitory connections. Primary and secondary areas encode respectively pitch and rhythm \cite{musacchia2014thalamocortical}, while high-order cortical areas encode the perceptual switches via competition (bistability).  (B) ACx circuit model. Primary ACx tonotopic responses consist of square-wave A and B tone inputs $i_A$ and $i_B$ with duration $T\!D$ and with the time between tone onsets $T\!R$ (called repetition time - the inverse of the presentation rate (PR)). Parameters $c$ and $d$ respectively represent the connection strength from $i_A$ ($i_B$) to the A(B) and B(A) units. Bottom: sketch of the model circuit consisting of two mutually excitatory and inhibitory populations with strengths $a$ and $b$ respectively, receiving inputs $i_A$ and $i_B$. Inhibition is delayed of the amount $D$.}
  \label{fig:figure2}
\end{figure}

\subsection{Existing models of auditory streaming} Inspired by evidence of feature separation shown in neural recordings in primary auditory cortex (A1) \cite{fishman2004auditory}, many existing models have sidestepped the issue of the temporal encoding of the perceptual interpetations by focusing on a feature representation (reviews: \cite{snyder2017recent,szabo2016computational,rankin2019computational}). Neurons responding primarily to the A or to the B tones are in adjacent locations, spatially separated along A1's tonotopic axis. The so-called neuromechanistic model \cite{rankin2015neuromechanistic} proposed the encoding of percepts based on discrete, tonotopically organised units interacting through plausible neural mechanisms. Models proposed in a neural oscillator framework feature significant redunancy in their structure or work only at specific presentation rate (PR) values \cite{wang2008oscillatory,perez2019uncoupling}. Temporal forward masking results in weaker responses to similar sounds that are close in time (at high PR), but this ubiquitous feature of the auditory system~\cite{moore2012introduction} has been overlooked in previous models.
  
\subsection{Theoretical framework.} The cortical encoding of sensory information involves large neural populations suitably represented by coarse-grained variables like the mean firing rate. The Wilson-Cowan equations \cite{Wilson1972} considered here describe neural populations with excitatory and delayed inhibitory coupling. Variants of these equations include networks with excitatory and inhibitory coupling, intrinsic synaptic dynamics that include neural adaptation, nonlinear gain functions \cite{Laing2002,Shpiro2007,curtu2008mechanisms} and symmetries \cite{diekman2012reduction,diekman2014network}. This framework (and related voltage- or conductance-based formulations) are widely used to study e.g. decision making~\cite{wang2002probabilistic}, perceptual competition in the visual \cite{Wilson2003,diekman2012reduction,vattikuti2016canonical} and in the auditory system \cite{rankin2015neuromechanistic}. Mathematical studies of these models often use a discontinuous (Heaviside) gain function due to its analytical tractability \cite{farcot2010limit}.
  
A range of neural and synaptic activation times often leads to timescale separation~\cite{rinzel1998analysis,izhikevich2007dynamical,ermentrout2010mathematical} as considered here. Singular perturbation theory has been instrumental in revealing the dynamic mechanisms behind neural behaviors involving a slow-fast decomposition, e.g. the generation of spiking and bursting \cite{izhikevich2007dynamical,desroches2016canards}, neural competition~\cite{curtu2008mechanisms,curtu2010singular} and rhythmic behaviors \cite{marder1996principles,rubin2000geometric}. In this work we use these techniques to determine the existence conditions of various dynamical states.
  
We consider the role of delayed inhibition in generating oscillatory activity compatible with auditory percepts. Delayed inhibition produces similar patterns of in- and anti-phase oscillations in spiking neural models \cite{wang1992alternating, ferrario2018bifurcations}. Delays in small neural circuits \cite{ashwin2016mathematical} lead to many interesting phenomena including inhibition-induced Hopf oscillations, oscillator death, multistability and switching between oscillatory solutions \cite{campbell2007time,dhamala2004enhancement}. Two novel features of our study are that the units are not intrinsically oscillating and that periodic forcing drives oscillations. Periodically forced, timescale separated  models of perceptual competition~\cite{jayasuriya2012effects,vattikuti2016canonical,perez2019uncoupling} typically do not feature delays.
 
\subsection{Outline} With the aim of clarifying a plausible model for the processing of ambiguous sounds we present a biologically-inspired neural circuit in ACx with mixed feature and temporal encoding (Section \ref{model_equations}). Section \ref{motivation} describes numerical simulations of the model states linked to percepts in the auditory streaming paradigm. Later sections focus on the analytical derivation of existence conditions for these states in a non-smooth, slow-fast regime. Proofs of most of these results are given in the Supplementary Material for the interested reader. In Sections \ref{fast_dyn} we dissect the model into slow and fast subsystems then analyze quasi-equilibria of the fast subsystem. We use this analysis in Sections \ref{dyn_noinputs} and \ref{dyn_inputs} to classify dynamical states with binary matrix representations (\emph{matrix form}). This tool determines all periodic states, their existence conditions and which states are impossible.  Sections \ref{sec_2TR_states} and \ref{biologically_relevant} classify periodic states for long and short inhibitory delays, respectively. Lastly, in Section \ref{comp_analysis} we show with numerics how these results extend to a smooth setting with reduced timescale separation. Overall, we propose a new method for analytically determining the solutions of a periodically driven networks in a slow-fast setting with delays. When applied to study the auditory streaming paradigm, these methods suggest how competing perceptual interpretations emerge as a result of mutual excitation and slow delayed inhibition in tonotopically localized units in a non-primary part of auditory cortex.
 
 \section{The mathematical model} \label{model_equations}
 We present a model for the encoding of different perceptual interpretations of the auditory streaming paradigm. Following our proposal of rhythm and pitch perception (\Fref{fig:figure2}A) we consider a periodically-driven competition network of two localised Wilson-Cowan units (\Fref{fig:figure2}B) with lumped excitation and inhibition generalised to include dynamics via inhibitory synaptic variables. The units A and B are driven by a stereotyped input signals $i_A$ and $i_B$ representative of neural responses in primary auditory cortex~\cite{fishman2004auditory} at tonotopic locations that preferentially respond to A and to B tones (\Fref{fig:figure2}B). The model is described by the following system of DDEs:
\begin{equation}
    \begin{array}{lcl} 
        \tau \dot{u}_A(t) & = & -u_A(t) + H(a u_B(t)- b s_B(t-D)+ c i_A(t)), \\ 
        \tau \dot{u}_B(t) & = & -u_B(t) + H(a u_A(t)- b s_A(t-D)+ c i_B(t)),  \\
        \dot{s}_A(t) & = & H(u_A(t))(1-s_A(t))/\tau-s_A(t)/\tau_i, \\
        \dot{s}_B(t) & = & H(u_B(t))(1-s_B(t))/\tau-s_B(t)/\tau_i,
    \end{array}\label{model}
\end{equation}
where units $u_A$ and $u_B$ represent the average firing rate of two neural populations encoding sequences of tone (sound) inputs with timescale $\tau$. The Heaviside gain function with activity-threshold $\theta \in (0,1)$: $\{H(x)\!=\!1$ if $x \geq \theta$ and 0 otherwise$\}$ is widely used in firing rate and neuronal field models \cite{curtu2008mechanisms,bressloff2011spatiotemporal} (we later relax this assumption to consider a smooth gain function). Mutual coupling through direct fast excitation has strength $a \geq 0$. The delayed, slowly-decaying inhibition has timescale $\tau_i$, strength $b \geq 0$ and delay $D$ (\Fref{fig:figure2}A). The synaptic variables $s_A$ and $s_B$ describe the time-evolution of the inhibitory dynamics. Typically we will assume $\tau_i$ to be large and $\tau$ to be small. This slow-fast regime and the choice of a Heaviside gain function allows for the derivation of analytical conditions for the existence of biologically relevant network states.

\subsection{Model Inputs}
We model primary ACx inputs to secondary areas as time-dependent, periodic square wave functions $i_A(t)$ and $i_B(t)$ representing the averaged excitatory synaptic currents from primary ACx at A and B tonotopic locations during the repetition of interleaved A and B tone sequences (\fref{fig:figure2}B top). These functions characterize responses to tones in primary ACx (from experiments \cite{fishman2004auditory}) rather than the sound waveform of the tone sequences (motivated in Section \ref{motivation}) and are defined by:
\begin{equation}
    \begin{array}{lcl} 
        i_A(t) & = c\sum_{k=0}^{\infty} \chi_{I_A^k} (t) + d \sum_{k=0}^{\infty} \chi_{I_B^k} (t) &  \\
        i_B(t) & = d \sum_{k=0}^{\infty} \chi_{I_A^k} (t) + c \sum_{k=0}^{\infty} \chi_{I_B^k} (t) & 
    \end{array}
    \label{inputs}
\end{equation}
Where $c\!\geq\!0$ and $d\!\geq\!0$ represent the input strengths from A (B) tonotopic location respectively to the A (B) unit and to the B (A) unit; $\chi_{I}$ is the standard indicator function over the set $I$, defined as $\chi_{I}(t)\!=\!1$ for $t \in I$ and 0 otherwise. The intervals when A and B tones are on (active tone intervals) are respectively $I_A^k\!=\![\alpha_k^A,\beta_k^A]$ and $I_B^k\!=\![\alpha_k^B,\beta_k^B]$ (see Figure \fref{fig:figure2}B top), where
$$ \alpha_A^k\!=\!2kT\!R, \quad \beta_A^k\!=\!2kT\!R\!+\!T\!D, \quad \alpha_B^k\!=\!(2k\!+\!1)T\!R, \quad \beta_B^k\!=\!
(2k\!+\!1)T\!R\!+\!T\!D. $$
Where the parameters $T\!D$ represents the duration of each tone's presentation (see the Discussion for another interpretation of $T\!D$) and $T\!R$ the time between tone onsets (called repetition time; $P\!R\!=\!1/T\!R$ is the presentation rate). Let us name the set of active tone intervals $R$ and its union $I$ as 
$$\Phi=\{ R \subset \mathbb{R} : R=I_k^A \: \text{or} \: R=I_k^B, \: \forall k \in \mathbb{N} \} \quad \mbox{and} \quad I=\bigcup_{R \in \Phi}R.$$ 

 As shown in Figure \ref{fig:figure1}, parameters $T\!D$ and $P\!R\!$ play an important influence on auditory streaming \cite{fishman2004auditory}. We consider $P\!R \in [1,40]$Hz, $T\!R\!\geq\!T\!D$ and $T\!R\!\geq\!D$, where $D$ is the inhibitory delay. These restrictions are typical conditions tested in psychoacoustic experiments. In particular, $T\!R\!\geq\!T\!D$ guarantees no overlaps between tone inputs, i.e.  $I_A^i \cap I_B^j=\emptyset$, $\forall i,j \in \mathbb{N}$.

\begin{remark} [Constraining model parameters]
Assuming $\tau$ sufficiently small and a Heaviside gain function $H$, system \ref{model} with no inputs ($i_A\!=\!i_B\!=\!0$) has two possible equilibrium points: a quiescent  state $P\!=\!(0,0,0,0)$ and an active state $Q\!=\!(1,1,1,1)$. If the difference between excitatory and inhibitory strengths $a\!-\!b\!\geq\!\theta$, then both $P$ and $Q$ exist, and any trajectory of the non-autonomous system is trivially determined by input strength $c$: 
\begin{itemize}
 \item If $c\!<\!\theta$: any trajectory starting from the basin of attraction of $P$ (or $Q$) quickly converges to $P$ ($Q$) and remains at this equilibrium.  
 \item If $c\!\geq\!\theta$: any trajectory converges to $Q$ and remains at this equilibrium. Indeed, if an orbit is in the basin of $P$, the synaptic variables monotonically decrease until one unit turns ON. This turns ON the other unit (since $a\!-\!b\!\geq\!\theta$) and both units remain ON.  
\end{itemize}
To avoid these unrealistic scenarios we assume the following conditions:
\begin{enumerate}[label=(\subscript{U}{{\arabic*}})]
 \centering
 \item $a-b < \theta$ \label{U1}
 \item $c \geq \theta$ \label{U2}
\end{enumerate}
Condition \ref{U1} guarantees that the point $P=(0,0,0,0)$, representing a quiescent state, is the only equilibrium of system \ref{model} with no inputs ($i_A=i_B=0$). Condition \ref{U2} guarantees inputs to be ``strong enough'' to turn ON the A (B) unit at the onset time of the A (B) tone in the absence of inhibition ($b=0$). 
\end{remark}

\section{A motivating example} \label{motivation}
 In this section we present examples of the type of responses studied throughout this work using a smooth version of model \ref{model} and by proposing a link between these responses and the different percepts in the auditory streaming experiments. We use a sigmoid gain function $S(x)=[1+\exp(-\lambda x)]^{-1}$ with fixed slope $\lambda\!=\!30$. Inputs in \ref{inputs} are made continuous using function $S$ by redefining them as:
 \begin{align}
  \begin{split}
   I_A(t) = c \cdot p(t)p(T\!D\!-\!t) + d \cdot  q(t)q(T\!D\!-\!t)\\
   I_B(t) = d \cdot p(t)p(T\!D\!-\!t) + c \cdot q(t)q(T\!D\!-\!t)
  \end{split}
  \label{continuous_inputs}
 \end{align}
 Where $p(t)\!=\!S(\sin(\pi P\!R\!\cdot\!t))$and $q(t)\!=\!S(-\sin(\pi P\!R\!\cdot\!t))$, so that the component $p(t)p(T\!D\!-\!t)$ ($q(t)q(T\!D\!-\!t)$) represents the responses to A (B) tone inputs with duration $T\!D$. These inputs are similar to the discontinuous input shown in Figure \ref{fig:figure2}B but with smooth ramps at the discontinuous jump up and jump down points. 
 
 Psychoacoustic experiments systematically analysed the changes in perceptual outcomes when varying input parameters $P\!R$ and $df$ (Figure \ref{fig:figure1}C). Parameter $P\!R$ is encoded in the model inputs' repetition rates. To model parameter $df$ we take into account the experimental recordings of the average spiking activity from the primary ACx of various animals (macaque \cite{fishman2004auditory,micheyl2005perceptual}, guinea pigs \cite{scholes2015stream}). These recordings show that the average spiking activity at A tonotopic locations decreases non-linearly with $df$ during B tone presentations. We thus assume that the input strength $d$ can be scaled by $df$ according to $d\!=\!c \cdot (1\!-\!df^{1/m})$, where $m$ is a positive integer and $df$ is a unitless parameter in $[0,1]$ which may be converted to semitone units using the formula $12\log(1\!+\!df)$. 

 Figure \ref{fig:figure3new}A shows simulated time histories of all the $2T\!R$-periodic states for different values of parameters $(P\!R,df)$, where all the other parameters are fixed. Blue and red bars indicate the A and B active tone intervals $[0,T\!D]$ and $[T\!R,T\!R\!+\!T\!D]$, respectively, to show when the inputs are on. The system exhibits one of three possible behaviors: (1) a state in which both units cross threshold (total of 4 crossings), (2) a state in which the A unit crosses threshold twice and the B unit once (total of 3 crossings) and (3) a state in which both units cross threshold once (total of 2 crossings). We then summarize the effect of parameters $(P\!R,df)$ on the convergence to the different attractors by running massive simulations at varying parameters $(P\!R,df)$ and counting the number of threshold crossings (Figure \ref{fig:figure3new}B). States (1-3) belong to one of the grey regions in Figure \ref{fig:figure3new}B. We note that state (2) coexist with its complex conjugate state for which the B unit crosses threshold twice and the A unit once (not shown). 

 \begin{figure}[htbp]
  \centering
  \includegraphics[width=0.9\linewidth]{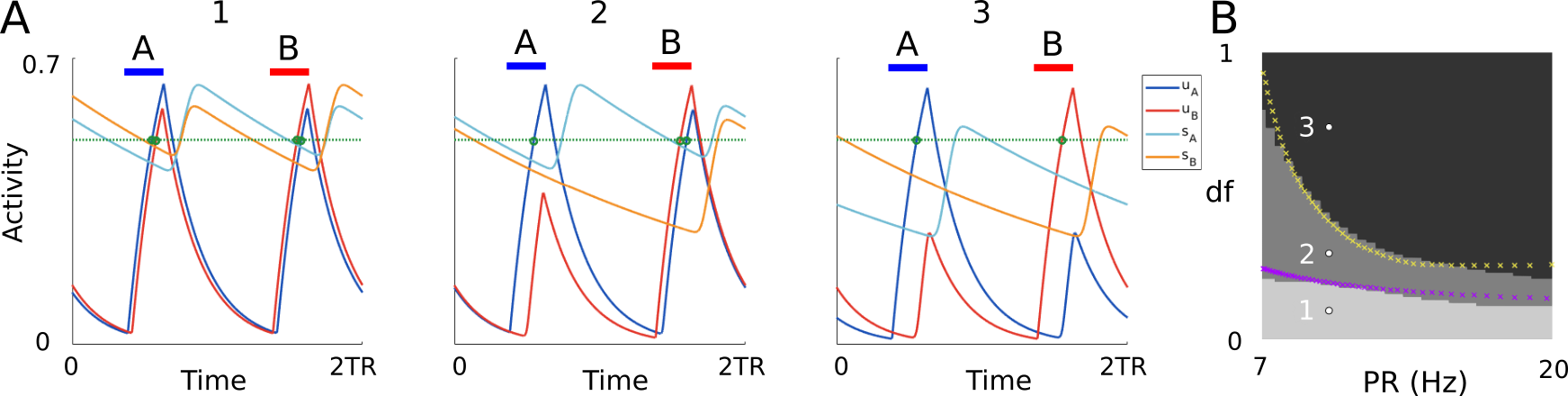}
  \caption{A. Time histories of the $2T\!R$-periodic states in system \ref{model}. Active tone A and B intervals are shown by blue and red bars, respectively. Units' threshold crossings are shown by green dots. B. The total number of threshold crossings for both units is shown in greyscale for simulated trajectories at varying $P\!R$ and $df$ (black = $2$, lightest gray = $4$ crossings). Parameters $P\!R$ and $df$ in panel A are shown by white dots in panel B. Yellow and purple crosses are the experimentally detected coherence and fission boundaries respectively (data replotted from Figure 2 in \cite{almonte2005integration}, digitalized using the software WebPlotDigitizer \cite{rohatgi2017webplotdigitizer}). The remaining parameters are $a\!=\!2$, $b\!=\!2.8$, $c\!=\!5.5$, $D\!=\!0.015$, $\theta\!=\!0.5$, $T\!D\!=\!0.022$, $\tau_i\!=\!0.25$, $\tau\!=\!0.025$, $\theta\!=\!0.5$ and $m\!=\!6$. Simulations are performed using dde23 in Matlab with absolute and relative tolerances set to $10^{-7}$. Initial conditions on the interval $[-D,0]$ are specified as a constant vector function equal to $[1,0,1,0]$. }
  \label{fig:figure3new}
\end{figure}
 
 We propose the following link between these states and the different percepts emerging in auditory streaming (integration, segregation and bistability). In our proposed framework rhythms are tracked by responding (threshold crossing) in the A and B units' activities of $2T\!R$-periodic states. More precisely: 
\begin{itemize}
 \item Integration corresponds to state (1): both units respond to both tones.
 \item Bistability corresponds to state (2): one unit responds to both tones the other unit responds to only one tone.
 \item Segregation corresponds to state (3): no unit responds to both tones.
\end{itemize} 

Following this proposal the states (1-3) match the regions of existence of their equivalent percepts. The transition boundaries between these states fit with the fission and coherence boundaries found experimentally (Figure \ref{fig:figure3new}B). In the next sections we take an analytical approach to study the model's states and their existence conditions. This approach allows us to derive expressions for the fission and coherence boundaries (equations (\ref{boundary_eqn}) in Section \ref{boundaries}) in a mathematically tractable version of the model (\ref{model_equations}). Quantitative comparisons between the analytical and computational approaches are discussed in Section \ref{comp_analysis}.
 
\section{Fast dynamics} \label{fast_dyn}
 In this and the next sections (until Section \ref{comp_analysis}) we present analytical results of the fast subsystem \ref{fast-model} with Heaviside gain. System \ref{model} can decoupled into slow and fast subsystems. The fast subsystem is given by:
\begin{equation}
    \begin{array}{lcl} 
        u_A(r)' & = & -u_A(r) + H(a u_B(r)- b s_B(r-D)+i_A(r)) \\ 
        u_B(r)' & = & -u_B(r) + H(a u_A(r)- b s_A(r-D)+i_B(r))  \\
        s_A(r)' & = & H(u_A(r))(1-s_A(r)) \\
        s_B(r)' & = & H(u_B(r))(1-s_B(r))
    \end{array}
    \label{fast-model}
\end{equation}
Where $'=d/dr$ is the derivative with respect to the fast scale $r=t/\tau$. Activities $u_A$ and $u_B$ take a value of $0$ or $1$, or move rapidly (on the fast time scale) between these two values. We call A(B) ON if $u_A \sim 1$ and OFF if $u_A \sim 0$. The activity of the A (B) unit is determined by the sign of $a u_B(t)\!-\! b s_B(t\!-\!D)\!+\!i_A(t)$ ($a u_A(t)\!-\! b s_A(t-D)\!+\!i_B(t)$). Positive sign changes make $u_A$ ($u_B$) jump up from $0$ to $1$ (turn ON), while negative sign changes in make $u_A$ ($u_B$) jump down from $1$ to $0$ (turn OFF). The synaptic variables can act on either the fast or the slow time scales. If A (B) is ON the variable $s_A$ ($s_B$) jumps to 1 on the fast time scale. Instead, if A (B) is OFF the dynamics of $s=s_A$ (or $s=s_B$) slowly decay according to:
\begin{equation}
 \dot{s}=-s/\tau_i
 \label{slow-model}
\end{equation}

\begin{remark}
\label{remark1} 
The previous considerations demonstrate that $s_A(t)$ ($s_B(t)$) is a monotonically decreasing in time, except for when the A (B) unit makes an OFF to ON transition. 
\end{remark}  

 We proceed by analyzing the system at times $t \in I$, i.e. in one of the active tone intervals. WLOG from the definition of $I$ we assume that $t \in I_k^A$, a generic A tone interval. The analysis below can easily be extended for B tone intervals $I_k^B$ by swapping parameters $c$ and $d$. On the fast time scale the A and B unit satisfy the subsystem:
\begin{equation}
    \begin{array}{lcl} 
        u_A' & = & -u_A + H(au_B-b \tilde{s}_B+c) \\ 
        u_B' & = & -u_B + H(au_A-b \tilde{s}_A+d)  
    \end{array}
    \label{fast-subsystem-I}
\end{equation}
Where $\tilde{s}_{A}\!=\!s_{A}(t\!-\!D)$ and $\tilde{s}_{B}\!=\!s_{B}(t\!-\!D)$. System \ref{fast-subsystem-I} has four equilibrium points: (0,0), (1,0), (0,1) and (1,1), and their existence conditions are reported in Table \ref{equilibria-fast}.
\begin {table}[h]
\begin{center}
\begin{tabular}{ |M{2cm}|M{2cm}|M{2cm}|M{2cm}|M{2cm}| } 
 \hline
  Equilibrium & \centering (0,0) & \centering (1,0) & \centering (0,1) & (1,1) \\ \hline
  Conditions 
  & $\begin{aligned}[t] 
  c\!&<\!b \tilde{s}_B\!+\!\theta \\ 
  d\!&<\!b \tilde{s}_A\!+\!\theta \\ 
 \end{aligned}$
  & $\begin{aligned}[t] 
  c \!&\geq\! b \tilde{s}_B\!+\!\theta \\ 
  a+d \!&<\! b \tilde{s}_A\!+\!\theta \\ 
 \end{aligned}$  
  & $\begin{aligned}[t] 
  a+c \!&<\! b \tilde{s}_B\!+\!\theta \\ 
  d \!&\geq\! b \tilde{s}_A\!+\!\theta \\ 
 \end{aligned}$  
  & $\begin{aligned}[t] 
  a+c \!&\geq\! b \tilde{s}_B\!+\!\theta \\ 
  a+d \!&\geq\! b \tilde{s}_A\!+\!\theta \\ 
 \end{aligned}$  
 \\ \hline 
\end{tabular}
\end{center}
 \label{equilibria-fast}
\caption{Equilibria and existence conditions for the fast subsystem \ref{fast-subsystem-I}}
\end {table}

The full system \ref{model} may jump between these equilibria due to the slow decay of the synaptic variables or when $s_A(t\!-\!D)$ and $s_B(t\!-\!D)$ jumps up to 1. 

\subsection{Basins of attraction} \label{basin_01}
From the inequalities given in Table \ref{equilibria-fast} we note that points $(1,0)$ and $(0,1)$ cannot coexist with any other equilibrium and thus have trivial basins of attraction. However, $(0,0)$ and $(1,1)$ may coexist under the following conditions: 
\begin{equation}
    \begin{array}{lcccl} 
        b \tilde{s}_B+\theta-a & \leq & c & < & b \tilde{s}_B+\theta \\ 
        b \tilde{s}_A+\theta-a & \leq & d & < & b \tilde{s}_A+\theta 
    \end{array}
    \label{bistability-fast}
\end{equation}
Thus we must have $a>0$, i.e. when the excitation is not absent in the model. To study the basin of attraction for these two equilibria, we consider the vector field of system \ref{fast-subsystem-I}. For convenience we introduce the following quantities: $s_1\!=\!(b \tilde{s}_A\!-\!c\!+\!\theta)/a$ and $s_2\!=\!(b \tilde{s}_B\!-\!c\!+\!\theta)/a$. 
Conditions \ref{bistability-fast} hold if and only if
$0 < s_k \leq 1$, for $k=1,2$. Thus we can rewrite system \ref{fast-subsystem-I} as:
\begin{equation}
    \begin{array}{lcl} 
        u_A' & = & -u_A+H(a(u_B-s_2)) \\ 
        u_B' & = & -u_B+H(a(u_A-s_1))
    \end{array}
    \label{fast-subsystem-simple}
\end{equation} 
Since $H$ is the Heaviside function $a$ can be removed. Figure \ref{fig:figure3} shows an example basins of attraction for parameter values for which $(0,0)$ and $(1,1)$ coexist (black circles). The $u_A$- and $u_B$-nullclines are shown in blue and red, respectively. We simulated model \ref{fast-subsystem-simple} starting from several initial conditions, covering the phase space. Simulated trajectories converge either to $(0,0)$ (green) and $(1,1)$ (purple) and show the subdivision in the basin of attraction. 

\begin{figure}[htbp]
  \centering
  \includegraphics[width=0.4\linewidth]{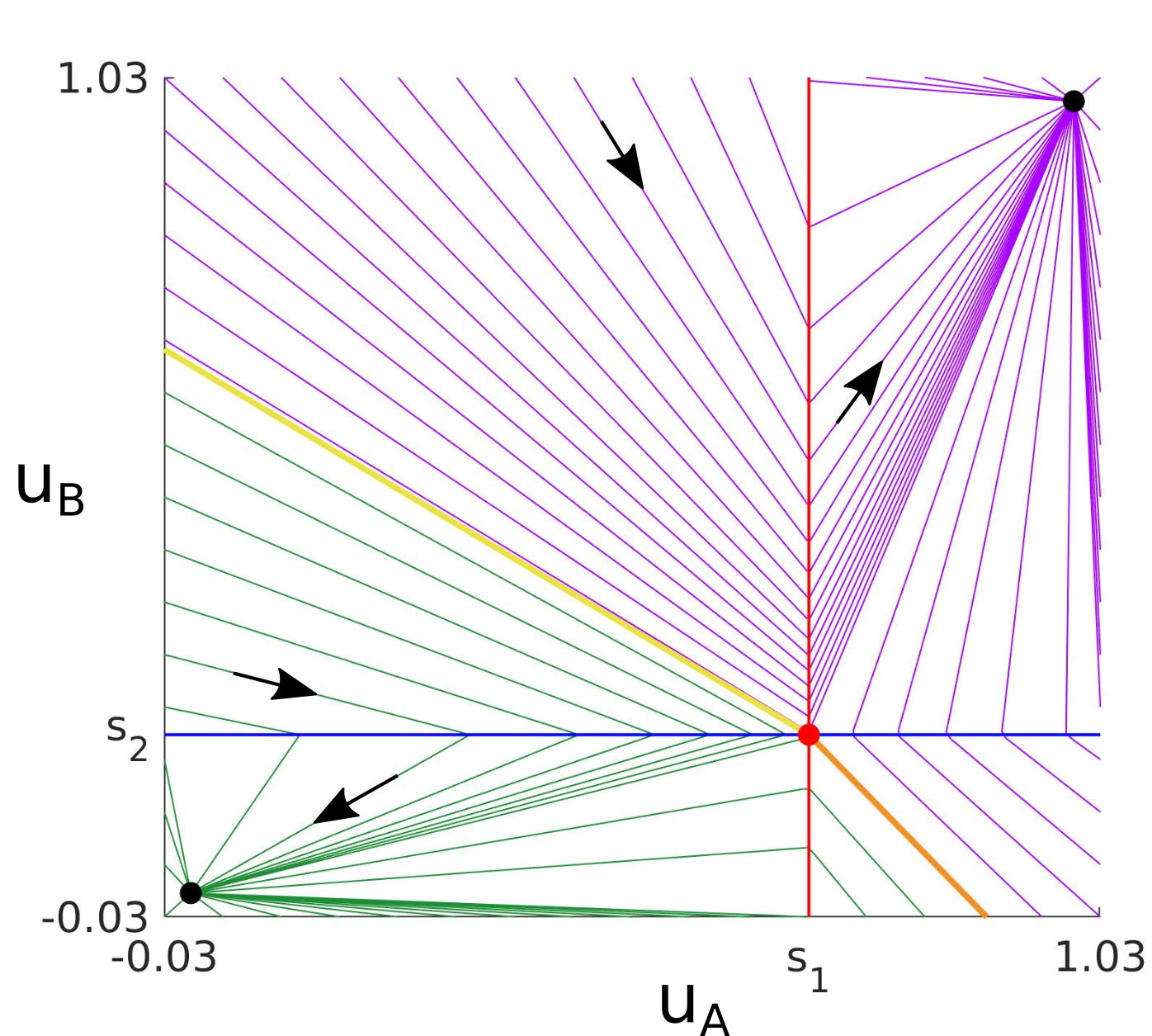}
  \caption{Phase portrait for system \ref{fast-subsystem-simple} with $s_1\!=\!0.7$ and $s_2\!=\!0.4$. Purple and green lines show orbits converging to $(1,1)$ and $(0,0)$, respectively (black circles). Black arrows indicate the direction of convergence. The $u_A$- and $u_B$-nullclines are shown in blue and red, respectively. Yellow and orange lines show the separatrices of the degenerate saddle $(s_1,s_2)$ (red circle). }
  \label{fig:figure3}
\end{figure}

There is a degenerate fixed point $(s_1,s_2)$ (red dot), where separatrices (yellow and orange lines) originate, dividing the phase plane into the regions of attraction shown in the figure. In the Supplementary Material \ref{separatrices} we prove that these curves are give by:
\begin{equation}
\begin{cases} 
(u_A-1)s_2/(s_1-1) & \mbox{if } u_A \leq s_1 \\ 
u_A(s_2-1)/s_1+1  & \mbox{otherwise}
\end{cases}
\nonumber
\end{equation}
The computational analysis of the basin of attractions (including equilibria and separatrices) with steep sigmoidal gains is presented in the Supplementary Material \ref{basin_sigmoid} and leads to qualitatively similar results.

\subsection{Differential convergence to $(1,1)$} \label{differential} 
We now study the differential rate of convergence of the variables $u_A$ and $u_B$ for parameter values where $(1,1)$ is the only stable equilibrium, for an orbit starting from  $(0,0)$. We will use the results below to classify of states of system \ref{model}. For simplicity we consider the case $t \in I_A^k$, as in system \ref{fast-subsystem-I}. Similar considerations hold in the case $t \in I_B^k$. Obviously, $(0,0)$ cannot be an equilibria, thus at least one of the two conditions in Table \ref{equilibria-fast} must not be met. There are three cases to consider: 
\begin{enumerate}
 \item If $c\!\!-b\tilde{s}_B\!\geq\!\theta$ and $d\!-\!b\tilde{s}_A\!\geq\!\theta$ both units turn ON simultaneously following each following the same dynamics $u'\!=\!1\!-\!u$. An orbit starting from $(0,0)$ must therefore reach $(1,1)$ under the same exponential rate of convergence. 
 \item If $c\!-\!b\tilde{s}_B\!\geq\!\theta$, $d\!-\!b\tilde{s}_B\!<\!\theta$ and $a\!+\!d\!-\!b\tilde{s}_A\!\geq\!\theta$ unit B turns ON after A by some small delay $\delta$ ($\sim \tau$). Indeed from $d\!-\!b\tilde{s}_B\!<\!\theta$ and $a\!+\!d\!-\!b\tilde{s}_A\!\geq\!\theta$ there $\exists u^* \in (0,1]$: $au_*\!+\!d\!-\!b\tilde{s}_A\!=\!\theta$. Since $c\!-\!b\tilde{s}_B\!\geq\!\theta$ the fast subsystem reduces to:
 \[\begin{array}{lcl} 
        u_A' & = & 1-u_A \\ 
        u_B' & = & -u_B+H(au_A-b \tilde{s}_A+d) \eqdef -u_B+\eta(u_A)
      \end{array}\]
    Thus, the dynamics of $u_A$ is independent of $u_B$. Consider an orbit starting  $(0,0)$ at $r=0$. From the first equation $u_A(r)$ tends to $1$ exponentially as $r \rightarrow \infty$, reaching a point $u^*$ at time $r^*\!=\!\log[(1\!-\!u^*)^{-1}]$. For $r\!<\!r^*$ we have $u_A(r)\!<\!u^*$, which yields $\eta(u_A(r))\!=\!0$. Since the orbit starts from $u_B=0$, it must remain constant and equal to zero $\forall r\!<\!r^*$. For $r\!\geq\!r^*$, $\eta(u_A(r))\!=\!1$ and $u_A(r) \rightarrow 1$ following the same dynamics as $u_A$ at time $r=0$. On the time scale $t\!=\!\tau r$ of system \ref{model}, the A unit precedes the B unit in converging to $1$ precisely after an infinitesimal delay 
 \begin{equation}
  \delta=\tau \log[(1\!-\!u^*)^{-1}].
  \label{delta}
 \end{equation} 
 \item The case $d\!-\!b\tilde{s}_A\!\geq\!\theta$, $c\!-\!b\tilde{s}_A\!<\!\theta$ and $a\!+\!c\!-\!b\tilde{s}_B\!\geq\!\theta$ is analogous to the previous after replacing $u_A$ with $u_B$. In this case A turns ON a delay $\delta$ after B. 
\end{enumerate}

\subsection{Fast dynamics for $t \in \mathbb{R}\!-\!I$} \label{fast_dynamics_I}
 The analysis for times when inputs are OFF ($t \in \mathbb{R}\!-\!I$) follows analogously by posing $c\!=\!d\!=\!0$ in system \ref{fast-subsystem-I}. Thus $(0,0)$ is an equilibrium for any values of parameters and delayed synaptic quantities $\tilde{s}_{A}$ and $\tilde{s}_{B}$. Instead $(1,1)$ is an equilibrium when
$$a\!-\!b\tilde{s}_{A} \!\geq\! \theta \quad \mbox{and} \quad a\!-\!b\tilde{s}_{B} \!\geq\! \theta.$$

\section{Dynamics in the intervals with no inputs ($\mathbb{R}\!-\!I$)} \label{dyn_noinputs}
The study of equilibria for the fast subsystem described so far constraints the dynamics of the full system in the intervals with no inputs, i.e. in $\mathbb{R}\!-\!I$. The first constraint is that the units can either be both ON, both OFF, or both turning OFF at any time in $\mathbb{R}\!-\!I$ (Theorem below). 

\begin{theorem} [Dunamics ifn $\mathbb{R}\!-\!I$]  \label{thm:uON}
For any $t \in \mathbb{R}\!-\!I$: 
\begin{enumerate}
 \item If A or B is OFF at time $t$, both units are OFF in $(t,t^*]$, where $$t^*=\min_{s \in I} \{s > t \}$$
 \item If A or B is ON at time $t$, both units are ON in $[t_*,t)$, where $$t_*=\max_{s \in I} \{s < t \}$$
\end{enumerate}
\end{theorem}
\begin{figure}[htbp]
  \centering
  \includegraphics[width=0.5\linewidth]{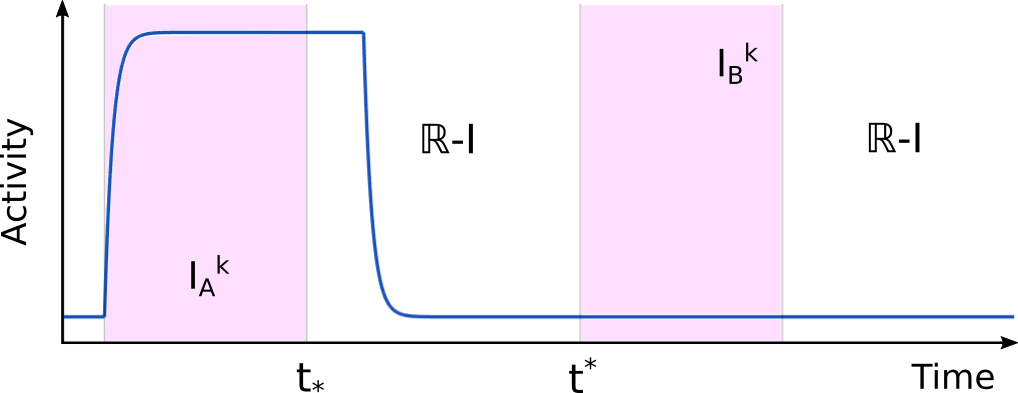}
  \caption{Illustration of Theorem~\ref{thm:uON} showing one unit's dynamics (blue) during one  $2T\!R$ period. Active tone intervals $I_A^k$ and $I_B^k$ are shown in purple. Note: the unit turns OFF at some time in $[t_*,t^*]$ due to the delayed inhibition from the the other unit, whose activity is omitted.}
  \label{cartoon1}
\end{figure}

This theorem is proved in the Supplementary Material \ref{thm_ON_appendix} and illustrated with an example in Figure \ref{cartoon1}. Due to this theorem we can classify network states as follows.

\begin{definition} [LONG and SHORT states] \label{long_short}
We define any state of system \ref{model}:
\begin{itemize}
 \item LONG if $\exists t \in \mathbb{R}-I$ when both units are ON
 \item SHORT if both units are OFF  $\forall t \in \mathbb{R}-I$ 
 \end{itemize}
\end{definition}

The choice of the names LONG and SHORT derives from the following considerations. Since both units are ON at some time $t \in \mathbb{R}-I$ of a LONG state, Theorem \ref{thm:uON} implies they must be ON at time the end of the active tone interval preceding $t$ and prolong their activation after the active tone interval up to time $t$. SHORT states by definition are OFF between each pair of successive tone intervals.

Theorem \ref{thm:uON} guarantees either unit can turn ON only during an active tone interval. This guarantees that the delayed synaptic variables are monotonically decreasing in the intervals $[\alpha_k^A,\alpha_k^A\!+\!D]$ and $[\alpha_k^B,\alpha_k^B\!+\!D]$ if the condition $T\!D\!+\!D\!<\!T\!R$ is guaranteed. The latter theorem is proven in the Supplementary Material \ref{syn_decay_appendix} and it is illustrated in Figure \ref{cartoon2}A. 

\begin{lemma}[synaptic decay] \label{lem:syn_decay}
If $T\!D\!+\!D<T\!R$ the delayed synaptic variables $s_A(t\!-\!D)$ and $s_B(t\!-\!D)$ are monotonically decreasing in $[\alpha_k^A,\alpha_k^A\!+\!D]$ or $[\alpha_k^B,\alpha_k^B\!+\!D]$, $\forall k \in \mathbb{N}$
\end{lemma}

A second important implication of Theorem \ref{thm:uON} under $T\!D\!+\!D\!<\!T\!R$ is that both units must turn OFF once between successive tone intervals (see next lemma). This guarantees that at the start of each active tone interval any state of the fast subsystem start from point $(0,0)$. The following lemma is proven in the Supplementary Material \ref{no_saturation_appendix} and it is illustrated in Figure \ref{cartoon2}A.

\begin{lemma}[no saturated states]\label{lem:no_saturation}
If $T\!D\!+\!D\!<\!T\!R$ both units are OFF in the intervals $(\alpha_k^A\!+\!T\!D+\!D,\alpha_{k}^B]$ and $(\alpha_k^B\!+\!T\!D+\!D,\alpha_{k+1}^A]$, $\forall k \in \mathbb{N}$.
\vspace{-0.5mm}
\end{lemma}


\begin{figure}[htbp]
  \centering
  \includegraphics[width=0.8\linewidth]{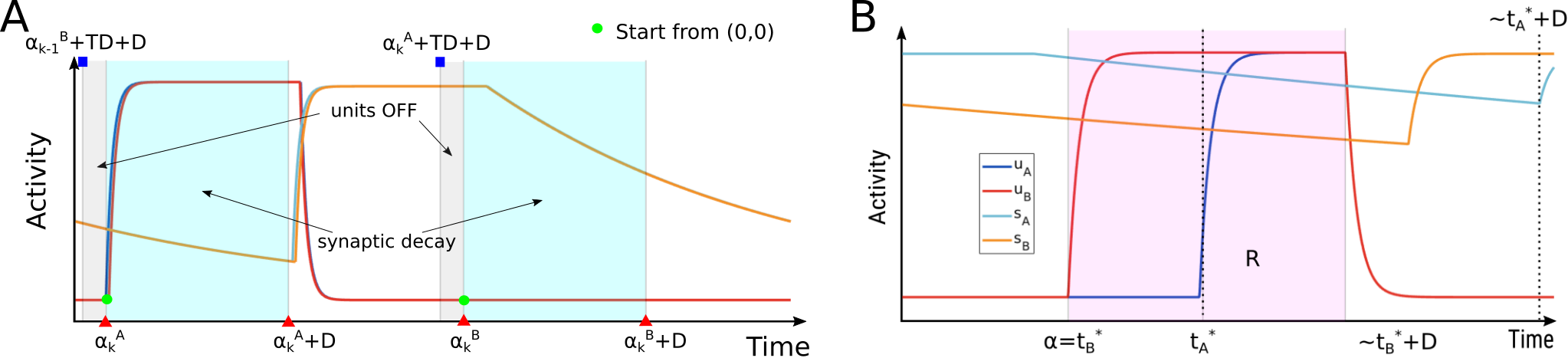}
  \caption{A. Example dynamics of the A and B units in each interval $L \subset \Gamma$ and $J$ illustrating lemmas \ref{lem:syn_decay} and \ref{lem:one_transition} during one period $2T\!R$. B. Dynamics in an active tone interval $R\!=\![\alpha,\beta] \in \Phi$ illustrating the quantities in (1--3) of Lemma~\ref{lem:one_transition}, where  $t_A^*$ ($t_B^*$) is the turning ON time for A (B), respectively. }
  \label{cartoon2}
\end{figure}

%

\section{Dynamics during the active tone intervals} \label{dyn_inputs}
We now study the possible dynamics of the full system during the active tone intervals $R \in \Phi$ under the condition $T\!D\!+\!D\!<\!T\!R$, for which lemmas \ref{lem:syn_decay} and \ref{lem:no_saturation} can be applied. We split this analysis by separating the cases $D\!>\!TD$ and $D\!\leq\!TD$. In this section we consider the case $D\!>\!TD$, and the other conditions is considered in section 8. The next theorem shows that the turning ON times of either unit can happen only at most once in $R$ and other results which will lead to the existence of only a limited number of states. 

\begin{lemma}[single OFF to ON transition]\label{lem:one_transition}
Consider an active tone interval $R\!=\![\alpha,\beta] \in \Phi$, and let A (B) be ON at a time $\bar{t} \in R$, then \newline
(1) A (B) is ON $\forall t \geq \bar{t}$, $t \in R$ \newline
(2) $\exists !\, t_A^* \, (t_B^*) \in R$ when A (B) turns ON  \newline
(3) $s_A(t\!-\!D)$ ($s_B(t\!-\!D)$) is decreasing for $t \in [\alpha,t_A^*\!+\!D]$ ($t \in [\alpha,t_B^*\!+\!D]$)
\end{lemma}
The previous Lemma is illustrated in the cartoon shown in Figure \ref{cartoon2} right. The proof is given in the Supplementary Material \ref{appendix_one_transition} and it implies the following Lemma. 
\begin{lemma}
Given any active tone interval $R \in \Phi$ we have:
\begin{enumerate}
 \item A (B) turns ON at time $\alpha$ $\Leftrightarrow$ A (B) is ON $\forall t \in (\alpha,\beta]$
 \item A (B) is OFF at time $\beta$ $\Leftrightarrow$ A (B) is OFF $\forall t \in R$
\end{enumerate}
\label{lem:ON-OFF-main}
\end{lemma}

Due to Lemma \ref{lem:one_transition} each unit may turn ON only once during each interval $R \in \Phi$. Thus the dynamics any state is determined precisely at the jump up points $t_A^*$ and $t_B^*$ for the units in $R$ (if these exist). 

\begin{definition}[MAIN and CONNECT states]\label{main_connect}
Any state (solution) of system \ref{model} is: 
\begin{itemize}
 \item MAIN if $\forall R \in \Phi$, if $\exists t^* \! \in \! R$ turning ON time for A or B, then $t^* \! = \! \min(R)$ \label{main}
 \item CONNECT if $\exists R \in \Phi$ and $\exists t^* \! \in \! R$, $t^*\!>\!\min(R)$ turning ON time for A or B \label{connect}
\end{itemize}
\begin{figure}[htbp]
  \centering
  \includegraphics[width=0.5\linewidth]{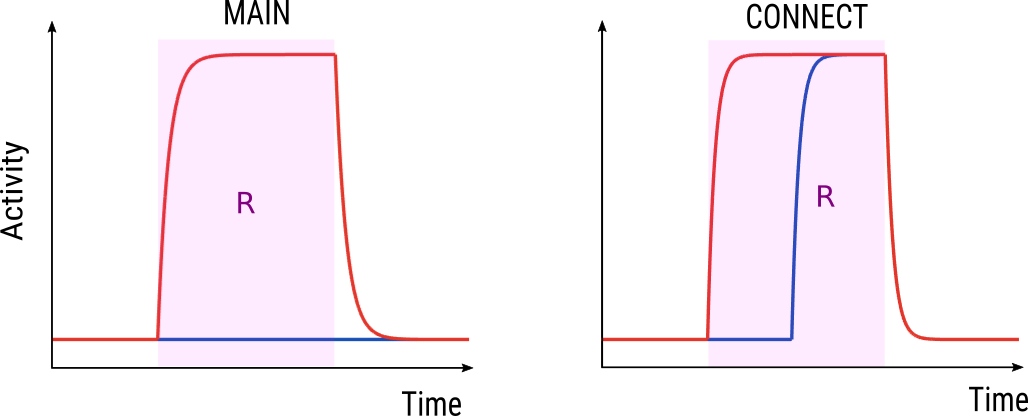}
  \caption{Example dynamics of the $u_A$ (red) and $u_B$ (blue) variables for MAIN and CONNECT states in an interval $R \in \Phi$. The left panel shows a MAIN state for which the A unit is OFF in $R$, while the B unit turns ON at time $t^*\!=\!\min(R)$. The right panel shows a CONNECT state for which the A unit turns ON at some time $t^*\!>\!\min(R)$, while the B unit turns ON at time $\min(R)$. }
  \label{cartoon4}
\end{figure}
\end{definition}

\begin{remark}
\label{matricial}
 MAIN states are either ON or OFF during any interval $R \in \Phi$, except (possibly) for a negligible interval of length $\sim 0$. Indeed due to differential convergence (Section \ref{differential}) one unit may turn ON at $\alpha$ following an infinitesimally small delay $\delta \sim \tau$, where $\delta$ is given by equation \ref{delta}. 
\end{remark}

\subsection{Classification of MAIN and CONNECT states - Matrix form} \label{classification_interval}
The results reported in the previous section constraint the possible dynamics during each active tone interval $R \in \Phi$. In this section we use these results to propose a classification of MAIN and CONNECT states based on their dynamics during these intervals and define the existence conditions for these states. 

Due to lemmas \ref{lem:no_saturation}, \ref{lem:one_transition} and \ref{lem:ON-OFF-main} the units of any state must be OFF at the start $R$ (orbits $(u_A,u_B)$ always start from $(0,0)$ at time $\alpha$), may turn ON at most once in $R$ and, if this occurs, it must remain ON until the end of $R$. Thus we have three possibilities: (1) both units are OFF in $R$, (2) only one unit turns ON once in $R$ or (3) both units turn ON once in $R$. These possibilities guarantee that any state in the network can be classified as MAIN or CONNECT. We note that condition \ref{U1} guarantees that (1) cannot occur $\forall R \in \Phi$, or $(0,0,0,0)$ would be an equilibria. Let us define the total inputs to the units for the A and B active tone intervals as a function of the synaptic quantity $s$:
\begin{equation}
f(s)\!=\!
\begin{cases} 
c\!-\!bs, & \mbox{if } R\!=\!I_A^k \\ 
d\!-\!bs, & \mbox{if } R\!=\!I_B^k 
\end{cases},
\quad 
g(s)\!=\!
\begin{cases} 
d\!-\!bs, & \mbox{if } R\!=\!I_A^k \\ 
c\!-\!bs, & \mbox{if } R\!=\!I_B^k
\end{cases}
\label{f_g}
\end{equation}

\textbf{Classification of MAIN states.} 
From the considerations given above the units' dynamics in $R\!=\! [\alpha,\beta]$ of any MAIN state is completely determined on the fast time scale at times $\alpha$ and $\beta$. Each unit can either turn ON at time $\alpha$ or be OFF at time $\beta$, depending on the system's parameters and on the following quantities: 
\[ \quad \underline{s}_A\!=\!s_A(\alpha\!-\!D),
   \quad \bar{s}_A\!=\!s_A(\beta\!-\!D),
   \quad \underline{s}_B\!=\!s_B(\alpha\!-\!D), 
   \quad \bar{s}_B\!=\!s_B(\beta\!-\!D) \]
Following the fixed point analyses we consider three conditions (summarized in Table \ref{M_conditions}):

\begin{itemize} \label{MAIN_definition}
\item \textbf{Both units turn ON at time $\alpha$}. This is equivalent to $(1,1)$ being the only equilibrium for the fast subsystem at time $\alpha$, which may occur under the conditions $M_{1-3}$. In summary, for case $M_1$ both units instantaneously turn ON at time $\alpha$. For case $M_2$ ($M_3$) unit B (A) turns ON after A (B) of an infinitesimal delay $\delta \sim \tau$ (see Section \ref{differential}).

\item \textbf{One unit turns ON at time $\alpha$ and the other unit is OFF at time $\beta$} - this corresponds to states satisfying one of conditions $M_{4-5}$. For case $M_4$ ($M_5$) A(B) turns ON at $\alpha$ and B(A) is OFF at $\beta$. Indeed $(1,0)$ ($(0,1)$) is the only stable equilibrium of the subsystem at times $\alpha$ and $\beta$, and thus $\forall t \in R$ due to Lemma \ref{lem:ON-OFF-main}. 

\item \textbf{A and B are OFF at time $\beta$} - it occurs when $(0,0)$ is the only stable equilibrium of the fast subsystem at time $\beta$, thus satisfying condition $M_6$. 
\end{itemize}

\begin {table}[h]
\begin{center}
\begin{tabular}{ |m{4em}|m{6em}|m{6em}|m{6em}|m{6em}|m{4em}| } 
 \hline
  \centering $M_1$ & \centering $M_2$ & \centering $M_3$ & \centering $M_4$ & \centering $M_5$ & $\quad M_6$ \\ \hline
 $\begin{aligned}[t] 
  f(\underline{s}_B) \!&\geq\! \theta \\ 
  g(\underline{s}_A) \!&\geq\! \theta 
 \end{aligned}$ & 
 $\begin{aligned}[t] 
  g(\underline{s}_A)\!&<\!\theta \\ 
  f(\underline{s}_B)\!&\geq\!\theta \\
  a\!+\!g(\underline{s}_A)\!&\geq\!\theta
 \end{aligned}$ & 
 $\begin{aligned}[t] 
  f(\underline{s}_B)\!&<\!\theta \\ 
  g(\underline{s}_A)\!&\geq\!\theta \\
  a\!+\!f(\underline{s}_B)\!&\geq\!\theta
 \end{aligned}$ & 
 $\begin{aligned}[t] 
  f(\underline{s}_B)\!&\geq\!\theta \\ 
  a\!+\!g(\bar{s}_A)\!&<\!\theta 
 \end{aligned}$ &  
 $\begin{aligned}[t] 
  g(\underline{s}_A)\!&\geq\!\theta \\ 
  a\!+\!f(\bar{s}_B)\!&<\!\theta 
 \end{aligned}$ & 
 $\begin{aligned}[t] 
  g(\bar{s}_A)\!&<\!\theta \\ 
  f(\bar{s}_B)\!&<\!\theta
 \end{aligned}$ 
 \\ \hline
\end{tabular}
\end{center}
\caption {Existence conditions for MAIN states in an interval $R \in \Phi$}
\label{M_conditions}
\end {table}

Figure \ref{cartoon5} shows the time histories of the MAIN states satisfying conditions $M_{1-5}$ in an interval $R \in \Phi$ ($M_6$ has been omitted since both units are inactive). Overall, this analysis proves that for a fixed interval $R \in \Phi$ any MAIN state of system \ref{model} satisfies only one of the above conditions $M_{1-6}$, and that any pair of MAIN states satisfying the same condition follow exactly the same dynamics in $R$. We can therefore define the dynamics of any MAIN state during any interval $R \in \Phi$ as follows.

\begin{definition} [MAIN classification]
We define the set of MAIN states in $R \in \Phi$ as $ M_R = \{ s=s(t) \mbox{ solutions of \ref{model} satisfying one of conditions } M_{1-6} \mbox{ in } R \} $
\end{definition}

\begin{figure}[htbp]
  \centering
  \includegraphics[width=0.95\linewidth]{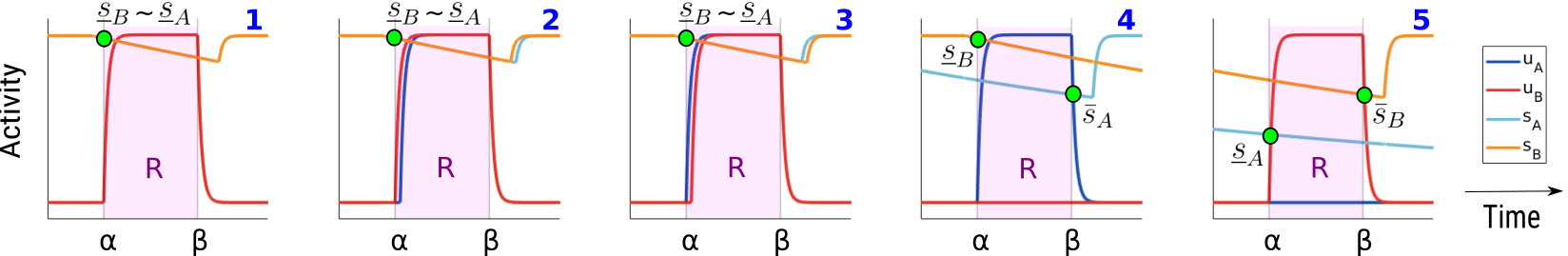}
  \caption{Example dynamics of a MAIN state satisfying condition $M_{1-5}$ in $R \in \Phi$.}
  \label{cartoon5}
\end{figure}

In next Theorem we construct a binary matrix representation for MAIN states defined by their existence conditions. This tool will enable us to define the existence conditions for $2T\!R$-periodic states and to rule out impossible ones (see Theorem \ref{MAIN_2TR}). 

\begin{theorem} \label{thm:matr_MAIN}
Let $R \in \Phi$. There is an injective map 
$$\rho^R \colon M_R \rightarrow B(2,2) \mbox{,} \quad s \mapsto V =
\begin{bmatrix}
    x_A & y_A \\
    x_B & y_B
\end{bmatrix}$$
with entries defined by
\begin{equation} \label{matricial_representation_MAIN}
x_A\!=\! H(f(\underline{s}_B)),
\; 
x_B\!=\! H(g(\underline{s}_A)),
\; 
y_A\!=\!
\begin{cases} 
1 & \textup{if } ax_B\!+\!f(\underline{s}_B) \!\geq\! \theta \\ 
0 & \textup{if } ax_B\!+\!f(\bar{s}_B) \!<\! \theta
\end{cases},
\; 
y_B\!=\!
\begin{cases} 
1 & \textup{if } ax_A\!+\!g(\underline{s}_A) \!\geq\! \theta \\ 
0 & \textup{if } ax_A\!+\!g(\bar{s}_A) \!<\! \theta
\end{cases}
\end{equation}
Moreover: 
\begin{equation}
\textup{Im}(\rho^R) = \Omega \eqdef \{V=\rho^R(s) : x_A \leq y_A, \: x_B \leq y_B, \: x_A=x_B=0 \Rightarrow y_A=y_B=0 \}
\label{im_rho_main}
\end{equation}
\end{theorem} 
\begin{proof}
A necessary condition for $\rho^R$ to be well defined is that $y_A$ and $y_B$ cannot be simultaneously equal to $0$ and $1$ (i.e. that both inequalities in their definition are not simultaneously satisfied). Due to the decay of the delayed synaptic variables in $R$ (Lemma \ref{lem:one_transition}) we have $\underline{s}_B\!\geq\!\bar{s}_B$. Moreover, since $f$ and $g$ are monotonically increasing, we have 
\begin{equation}
\label{f_g_monoto}
f(\underline{s}_B)\!\leq\!f(\bar{s}_B) \quad \mbox{and} \quad g(\underline{s}_B)\!\leq\!g(\bar{s}_B)
\end{equation}
Which proves that $y_A$ is exclusively equal to $0$ or $1$ (analogously for $y_B$).

Next, we notice that any matrix $V\!=\!\rho^R(s)$ satisfies the following:
\begin{equation} 
x_A \leq y_A, \quad x_B \leq y_B,  \quad x_A=x_B=0 \Rightarrow y_A=y_B=0
\label{yA-x_A}
\end{equation}
We prove the first inequality $x_A \leq y_A$ ($x_B \leq y_B$ is analogous). WLOG we assume $x_A\!=1$, and therefore $f(\underline{s}_B) \!\geq\! \theta$. Since $a\!\geq\!0$ and $x_B\!\geq\!0$ we have $ax_B\!+\!f(\underline{s}_B) \!\geq\! f(\underline{s}_B) \!\geq\! \theta$, thus implying $y_A\!=\!1$. The final part holds because, given $x_A\!=\!x_B\!=\!0$, we have $ ax_B\!+\!f(\bar{s}_B) \!\leq\! f(\underline{s}_B) \!<\! \theta, \quad ax_A\!+\!g(\bar{s}_A) \!\leq\! g(\underline{s}_A) \!<\! \theta $. 

From conditions \ref{f_g_monoto} and \ref{yA-x_A} it is easily checked that each element $s \in M_R$ satisfying condition $M_{i}$ has one of the following images $\rho^R(s)$:

\begin{small}
\begin{equation}
(M_1) \; 
\begin{bmatrix}
    1 & 1 \\
    1 & 1
\end{bmatrix}
, \;
(M_2) \; 
\begin{bmatrix}
    1 & 1 \\
    0 & 1
\end{bmatrix}
, \;
(M_3) \; 
\begin{bmatrix}
    0 & 1 \\
    1 & 1
\end{bmatrix}
, \;
(M_4) \; 
\begin{bmatrix}
    1 & 1 \\
    0 & 0
\end{bmatrix}
, \;
(M_5) \; 
\begin{bmatrix}
    0 & 0 \\
    1 & 1
\end{bmatrix}
, \;
(M_6) \; 
\begin{bmatrix}
    0 & 0 \\
    0 & 0
\end{bmatrix}
\nonumber
\end{equation}
\end{small}

Since any MAIN state has a distinct image, $\rho^R$ is well defined, injective, and $|Im(\rho^R)|\!=\!6$. Given that the total number of matrices $V \in B(2,2)$ satisfying conditions \ref{yA-x_A} are precisely 6 (no other matrix is possible), we have $Im(\rho^R)\!=\!\Omega$
\end{proof}

\textbf{Classification of CONNECT states.} Our classification and matrix form of CONNECT states follows analogously from that of MAIN states described previously. We recall that in such states at least one unit turns ON at some time in an active tone interval $R\!=\![\alpha,\beta]$. There are three cases to consider:
\begin{enumerate}
	\item Unit A(B) turns ON at time $\alpha$ and B(A) turns ON at time $t^*$, $\exists t^* \in (\alpha,\beta]$. 
	\item Unit A(B) is OFF at time $\beta$ and B(A) turns ON at time $t^*$, $\exists t^* \in (\alpha,\beta]$.  
	\item $\exists t^*, s^* \in (\alpha,\beta]$ times when the A and B unit turns ON. 
\end{enumerate}
These lead to the conditions in Table \ref{C_conditions}, which are explained in the Supplementary Material \ref{connect_appendix}. Case 1. lead to the conditions $C_{1-2}$, cases 2. lead to the conditions $C_{3-4}$, while case 3. leads to two possibilities depending on if A turns ON before or after B: $C_5^1$ and $C_5^2$. For simplicity do not distinguish between these cases and define ($C_5$) as referring to either condition. This leads to the following definition. 

\begin {table}[h]
\begin{center}
\begin{tabular}{ |m{5em}|m{5em}|m{5em}|m{5em}|m{5em}|m{5em}| } 
 \hline
 \centering $C_1$ & \centering $C_2$ & \centering $C_3$ & \centering $C_4$ & \centering $C_5^1$ & $\quad \quad C_5^2$ \\ \hline
 $\begin{aligned}[t] 
  f(\underline{s}_B)\!&\geq\!\theta \\ 
  a\!+\!g(\underline{s}_A)\!&<\!\theta \\
  a\!+\!g(\bar{s}_A)\!&\geq\!\theta
 \end{aligned}$ & 
 $\begin{aligned}[t] 
  g(\underline{s}_A)\!&\geq\!\theta \\ 
  a\!+\!f(\underline{s}_B)\!&<\!\theta \\
  a\!+\!f(\bar{s}_B)\!&\geq\!\theta
 \end{aligned}$ & 
 $\begin{aligned}[t] 
  g(\underline{s}_A)\!&<\!\theta \\ 
  g(\bar{s}_A)\!&\geq\!\theta \\
  a\!+\!f(\bar{s}_B)\!&<\!\theta
 \end{aligned}$ & 
 $\begin{aligned}[t] 
  f(\underline{s}_B)\!&<\!\theta \\ 
  f(\bar{s}_B)\!&\geq\!\theta \\
  a\!+\!g(\bar{s}_A)\!&<\!\theta
 \end{aligned}$ &  
 $\begin{aligned}[t] 
  t^*\!&\leq\!s^* \\ 
  f(\underline{s}_B)\!&<\!\theta \\
  f(\bar{s}_B)\!&\geq\!\theta \\
  a\!+\!g(\bar{s}_B)\!&\geq\!\theta
 \end{aligned}$ & 
 $\begin{aligned}[t] 
  t^* \!&>\! s^* \\ 
  g(\underline{s}_A)\!&<\!\theta \\
  g(\bar{s}_A)\!&\geq\!\theta \\ 
  a\!+\!f(\bar{s}_A)\!&\geq\!\theta
 \end{aligned}$ 
 \\ \hline
\end{tabular}
\end{center}
\caption {Existence conditions for CONNECT states in an interval $R \in \Phi$}
\label{C_conditions}
\end {table}

\begin{definition} [CONNECT classification]
We define the set of CONNECT states in $R \in \Phi$ as $ C_R = \{ s=s(t) \mbox{ solutions of \ref{model} satisfying one of conditions } C_{1-5} \mbox{ in } R \} $. 
\end{definition}

Similar to  MAIN states, the existence conditions for each CONNECT state in $R$ can equivalently be expressed using a binary matrix $W \in B(2,3)$.

\begin{theorem} \label{CONNECT_R}
Set $R \in \Phi$. There is an injective map: 
$$ \varphi^R\colon C_R \rightarrow B(2,3) \mbox{,} \quad
s \mapsto W = 
\begin{bmatrix}
    x_A & y_A & z_A \\
    x_B & y_B & z_B
\end{bmatrix} $$
With entries defined by:
\begin{align}
\begin{split}
x_A\!&=\!H(f(\underline{s}_B)), \: y_A\!=\!H(ax_B\!+\!f(\underline{s}_B)), \: z_A\!=\!H(a\!+\!f(\bar{s}_B)) \\
x_B\!&=\!H(g(\underline{s}_A)), \: y_B\!=\!H(ax_A\!+\!g(\underline{s}_A)), \: z_B\!=\!H(a\!+\!g(\bar{s}_A)) 
\end{split}
\label{matricial_representation_CONN}
\end{align}
And we have:
$$ Im(\varphi^R)\!=\! \{ 
W : x_A \!\leq\! y_A \!\leq\! z_A, 
x_B \!\leq\! y_B \!\leq\! z_B, 
x_A\!=\!x_B\!=\!0 \Rightarrow y_A\!=\!y_B\!=\!0, 
y_A\!<\!z_A \textup{ or } y_B\!<\!z_B
\} $$
\end{theorem}

The proof of this theorem is similar to the one of Theorem \ref{thm:matr_MAIN} and is given in the Supplementary Material \ref{appendix_CONN_matr}. As shown in this proof, each CONNECT state satisfying one of conditions $C_{1-5}$ has a corresponding image $\varphi^R(s)$ shown below. 

\begin{small}
\begin{equation}
(C_1) \; 
\begin{bmatrix}
    1 & 1 & 1 \\
    0 & 0 & 1
\end{bmatrix}
\quad 
(C_2) \; 
\begin{bmatrix}
    0 & 0 & 1 \\
    1 & 1 & 1
\end{bmatrix}
\quad 
(C_3) \; 
\begin{bmatrix}
    0 & 0 & 0 \\
    0 & 0 & 1
\end{bmatrix}
\quad 
(C_4) \; 
\begin{bmatrix}
    0 & 0 & 1 \\
    0 & 0 & 0
\end{bmatrix}
\quad 
(C_5) \; 
\begin{bmatrix}
    0 & 0 & 1 \\
    0 & 0 & 1
\end{bmatrix}
\quad 
\nonumber
\end{equation}
\end{small}

The previous two theorems naturally lead to the definition of the matrix form of the MAIN and CONNECT states in each interval $R \in \Phi$. 
\begin{definition} [Matrix form]
Let $R \in \Phi$ be an active tone interval:
\begin{itemize}
 \item The matrix form of a MAIN state $s \in M_R$ is $V\!=\!\rho^R(s)$ defined by \ref{matricial_representation_MAIN}.
 \item The matrix form of a CONNECT state $s \in C_R$ is $W\!=\!\varphi^R(s)$ defined by \ref{matricial_representation_CONN}.
\end{itemize}
\label{main_matr_def}
\end{definition}

\begin{remark} [Visualisation via the Matrix form] \label{visualization_matr}
The first (second) row of the matrix form of each MAIN state provide an intuitive visualization of its A (B) units' dynamics in $R$. Indeed, given $\delta$ as defined in Section \ref{differential} we may subdivide $R$ into $ R=[\alpha,\alpha\!+\!\delta] \cup [\alpha\!+\!\delta,\beta] $. The dynamics of the A unit at time $\alpha$ is given by $x_A$. If $x_A\!=\!1$ the A unit turns ON at time $\alpha$ and remains ON in $(\alpha,\beta]$. If $x_A=0$ and $y_A\!=\!1$ the A unit is OFF at time $\alpha$, turns ON at time $\alpha+\delta$ and remains ON in $(\alpha+\delta,\beta]$. If $y_A\!=\!0$ (which implies $x_A\!=\!0$) the A unit is OFF $\forall t \in R$. Similar considerations hold for the B unit. 

Similarly, the dynamics of the A (B) unit in $R$ of a CONNECT state is represented by the first (second) row of its matrix form. For example, for the state defined by condition $C_2$ unit A turns ON at some time $t^* \in (\alpha,\beta]$, while unit B turns ON at time $\alpha$. Given $\delta$ as defined in Section \ref{differential}, we may subdivide $R$ into $ R=[\alpha,\alpha\!+\!\delta] \cup [\alpha\!+\!\delta,t^*] \cup [t^*,\beta] $. From conditions $C_2$ we have $y_A\!=\!0$ (which implies $x_A\!=\!0$) and $z_A\!=\!1$. Thus A is OFF during $[\alpha,\alpha+\delta]$ and $[\alpha\!+\!\delta,t^*]$, turns ON at time $t^*$ and remains ON in $[t^*,\beta]$. Since $x_B\!=\!1$ (which implies $y_B\!=\!z_B\!=\!1$), the B unit turns ON at time $\alpha$ and remains ON in $[\alpha,\beta]$. 
\end{remark}

\begin{remark} [Matrix form extension for MAIN states] \label{link_main_connect_matr}
We showed that MAIN (CONNECT) states in an interval $R \in \Phi$ can be represented using a $2 \times 2$ ($2 \times 3$) binary matrix. However, MAIN states can also be equivalently represented using the same $2 \times 3$ matrix form $W$ defined for CONNECT states in the previous theorem, by replacing the definition of $z_A$ and $z_B$ with
$ z_A\!=\!H(ay_B\!+\!f(\bar{s}_B))$ and $z_B\!=\!H(ay_A\!+\!g(\bar{s}_A))$. One can check that each existence condition $M_{1-6}$  given in \ref{matricial_representation_CONN} defines one of the following $2 \times 3$ matrices:

\begin{footnotesize}
\begin{equation}
(M_1) 
\begin{bmatrix}
    1 & 1 & 1 \\
    1 & 1 & 1
\end{bmatrix}
(M_2)
\begin{bmatrix}
    1 & 1 & 1 \\
    0 & 1 & 1
\end{bmatrix}
(M_3)
\begin{bmatrix}
    0 & 1 & 1 \\
    1 & 1 & 1
\end{bmatrix}
(M_4) 
\begin{bmatrix}
    1 & 1 & 1 \\
    0 & 0 & 0
\end{bmatrix}
(M_5) 
\begin{bmatrix}
    0 & 0 & 0 \\
    1 & 1 & 1
\end{bmatrix}
(M_6)
\begin{bmatrix}
    0 & 0 & 0 \\
    0 & 0 & 0
\end{bmatrix}
\nonumber
\end{equation}
\end{footnotesize}

This result guarantees that we can represent all the states in the system using a general $2 \times 3$ matrix form (used in Section \ref{sec_2TR_states}).  
\end{remark}

So far we have shown the existence conditions for MAIN and CONNECT states in any active tone interval $R$. The following lemma regards the conditions for which LONG states can occur. This will enable us to complete the existence conditions for all states outside the active tone intervals. 
\begin{lemma} [LONG conditions] \label{lem:LONG_conditions}
A state is LONG if and only if $\exists R=[\alpha,\beta] \in \Phi$:
\vspace{1mm}
\begin{enumerate}
 \item A and B turn ON at times $t^*_A$ and $t^*_B \in R$, respectively. 
 \item $a\!-\!b s_A(\beta\!-\!D) \geq \theta$ and $a\!-\!b s_B(\beta\!-\!D) \geq \theta$.
\end{enumerate}
\vspace{1mm}
Moreover, both units are ON in $[\beta,t^*\!+\!D]$, turn OFF at time $t^*\!+\!D$, and are OFF in $(t^*\!+\!D,t_{up}]$, where $ t^*=\min\{ t^*_A,t^*_B \} \quad \mbox{and}$ and $t_{up}=\min_{s \in I}\{ s > t \}$. 
\end{lemma}
\begin{figure}[htbp]
  \centering
  \includegraphics[width=0.5\linewidth]{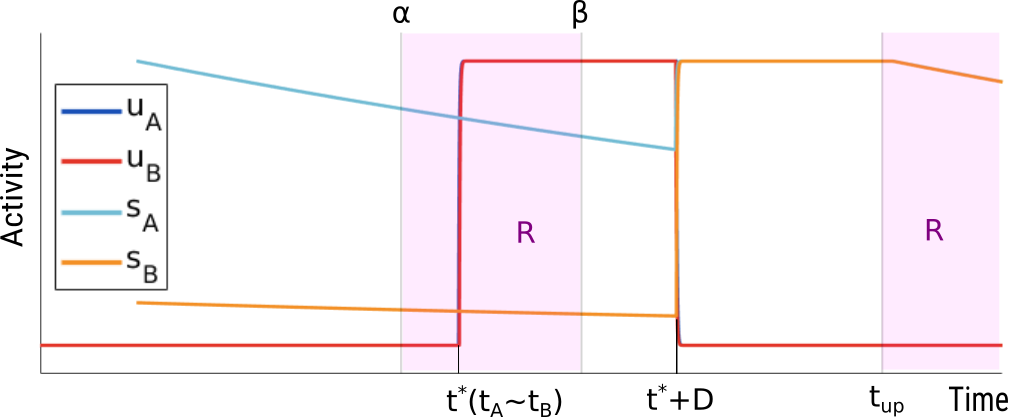}
  \caption{Example dynamics of a LONG state showing quantities used in Lemma \ref{lem:LONG_conditions}. }
  \label{cartoon6}
\end{figure}
The proof of this lemma is given in the Supplementary Material \ref{appendix_LONG_thm}. The idea of the proof is sketched in Figure \ref{cartoon6}. Both units of a LONG stats must be ON at time $\beta$ due to Theorem \ref{thm:uON}, which proves to 2. Furthermore, since both synaptic variables are monotonically decaying in the interval $[\alpha,t^*\!+\!D]$ this activity must persist until time $t^*\!+\!D$, when one of these variables (or both) jump to 1. 

\section{2TR-periodic states} \label{sec_2TR_states}
In the previous section we classified network states in a generic active tone interval. In this section we use this analysis to study $2T\!R$-periodic states under the conditions $D\!>\!TD$ and $T\!D\!+\!D\!<\!T\!R$. We analytically derive the parameter conditions leading to the existence of all $2T\!R$-periodic states in the system and use the matrix form to rule out which states cannot exist. 

\begin{definition}
A state $\psi \!=\! \psi(t) \!=\! (u_A(t),u_B(t),s_A(t),s_B(t))$ is $2T\!R$-periodic if $ \psi(t\!+\!2T\!R) = \psi(t)$, $\forall t \in \mathbb{R}$. We call $SM$ and $LM$ ($SC$ and $LC$) the sets of $2T\!R$-periodic MAIN (CONNECT) states of the SHORT and LONG type, respectively. 
\end{definition}

Before analyzing these states it is important to first assess the model's symmetry.  
\begin{remark} [$\mathbb{Z}_2$ symmetry]
\label{symmetry}
  System \ref{model} is symmetric under a transformation swapping the A and B indexes in system \ref{model} and by applying the time shift $T\!R$ to the active tone input functions. Indeed let us rewrite the model as a general non-autonomous dynamical system
$$ \dot{v}(t)=z(v(t),i_A(t),i_B(t)), \quad v\!=\!(u_A,u_B,s_A,s_B)$$
Now consider the map $\kappa$ whose action swaps the A and B indices of all variables, defined as
$$ \kappa: v=(u_A,u_B,s_A,s_B,i_A,i_B) \mapsto (u_B,u_A,s_B,s_A,i_B,i_A) $$
Since $i_A(t\!+\!T\!R)\!=\!i_B(t)$ and $i_B(t\!+\!T\!R)\!=\!i_A(t)$, $\forall t \in \mathbb{R}$, we have 
$$ \kappa(z(v(t),i_A(t),i_B(t))) \!=\! z(\kappa(v(t\!+\!T\!R),i_B(t\!+\!T\!R),i_A(t\!+\!T\!R))),$$
which proves the model is symmetric under the transformation $\kappa$ time shifted by $T\!R$. Given that no  symmetric transformation other than $\kappa$ and the identity exist, the system is $\mathbb{Z}_2$-equivariant. Thus, given a periodic solution $v(t)$ with period $T$, its $\kappa$-conjugate cycle $\kappa(v(t\!+\!TR))$ must also be a solution with equal period (asymmetric cycle), except in the case that $v(t)\!=\!\kappa(v(t)), \, \forall t \!\in\! [0,T]$ (symmetric cycle).  Asymmetric cycles always exist in pairs: the cycle and its conjugate. We note that in-phase and anti-phase limit cycles with period $2T\!R$ are both symmetric cycles. 
\end{remark}

To study  $T\!R$-periodic states we can replace the set of active tone intervals $I$ with:
$$I=I_1 \cup I_2=[0,T\!D] \cup [T\!R,T\!R\!+\!T\!D]$$
As shown in the previous section, for any state $\psi \in SM$ the activities of both units during each interval $I_i$, with $i=1,2$, can be represented by a matrix $V_i$. This matrix uniquely depends on the values of the delayed synaptic variables at times $\alpha_i=(i\!-\!1)T\!R$ and $\beta_i=(i\!-\!1)T\!R\!+\!T\!D$. More precisely, in equations \ref{matricial_representation_MAIN} we must substitute $\underline{s}_A$ with $s^{i-}_A$, $\bar{s}_A$ with $s^{i+}_A$, $\underline{s}_B$ with $s^{i-}_B$ and $\bar{s}_B$ with $s^{i+}_B$, where:
\begin{equation}
s^{i-}_A\!=\!s_A(\alpha_i-D), \quad
s^{i-}_B\!=\!s_B(\alpha_i\!-\!D), \quad
s^{i+}_A\!=\!s_A(\beta_i\!-\!D), \quad
s^{i+}_B\!=\!s_B(\beta_i\!-\!D) \\
\label{syn_values}
\end{equation}

\subsection{SHORT states} \label{short_2TR}
It turns out (see Theorem below) that for SHORT MAIN and CONNECT states these values depend on the following quantities, as stated in the next Theorem.  
\begin{equation}
 N^- \!=\! e^{-(TR\!-\!TD\!-\!D)/\tau_i}, \; N^+ = e^{-(TR\!-\!D)/\tau_i}, \; 
 M^- \!=\! e^{-(2TR\!-\!TD\!-\!D)/\tau_i}, \; M^+ = e^{-(2TR\!-\!D)/\tau_i} 
 \label{N_M_MAIN}
\end{equation}
We note that $N^- \!\geq\! N^+ \!\geq\! M^- \!\geq\! M^+$. The dependence of the synaptic variables on these quantities is crucial, because it guarantees that the existence conditions shown in Table \ref{M_conditions} depend uniquely on the model parameters for $2T\!R$-periodic states. 

\begin{theorem} \label{MAIN_2TR}
There is an injective map 
$$\rho\colon SM \rightarrow B(2,4) \mbox{,} \quad
\psi \mapsto V =
\begin{bmatrix} [c|c]
    V_1 & V_2 
\end{bmatrix}
=
\begin{bmatrix} [cc|cc]
    x_A^1 & y_A^1 & x_A^2 & y_A^2 \\
    x_B^1 & y_B^1 & x_B^2 & y_B^2
\end{bmatrix}$$

Where $V_1$ ($V_2$) is the matrix form of $\psi$ in $I_1$ ($I_2$) defined by \ref{matricial_representation_MAIN}, and: 
\begin{equation}
s_B^{i \pm}\!=\!N^{\pm}y_B^j\!+\!M^{\pm}(1-y_B^j)y_B^i, 
\quad  
s_A^{i \pm}\!=\!N^{\pm}y_A^j\!+\!M^{\pm}(1-y_A^j)y_A^i, 
\quad \forall j\!=\!1,2, j\!\neq\!i
\label{syn_values_MAIN_SHORT}
\end{equation} 
In addition, 
$$Im(\rho)=\Omega \eqdef \{V =
\begin{bmatrix} [c|c]
    V_1 & V_2 
\end{bmatrix} :
V_1 \in Im(\rho^{I_1}), V_2 \in Im(\rho^{I_1}) \mbox{ satisfying 1-4 below}  \}$$
\begin{enumerate} \label{MAIN_conditions}
 \item $y_A^1=y_B^2=1 \Rightarrow x_A^1=x_B^2$ and $y_A^2=y_B^1=1 \Rightarrow x_A^2=x_B^1$
 \item $y_B^1=y_B^2 \Rightarrow x_A^1 \geq x_A^2$ and $y_A^1=y_A^2 \Rightarrow x_B^2 \geq x_B^1$
 \item $y_A^2=1 \Rightarrow x_B^1 \leq r$ and $y_B^1=1 \Rightarrow x_A^2 \leq r$, for any entry $r$ in $V$
 \item $y_A^2=y_B^2$, $y_A^1=y_B^1 \Rightarrow x_A^1 \geq x_B^1$ and $x_B^2 \geq x_A^2$
\end{enumerate}
\end{theorem}
\begin{proof}
The proofs of equations \ref{syn_values_MAIN_SHORT} and of conditions 1-4 is given in the Supplementary Material \ref{appendix_thm_MAIN_R}. The validity of these conditions implies $Im(\rho) \subseteq \Omega$. In the next paragraph we will prove that $Im(\rho)\!=\!\Omega$. Assume for now this to be true. The definition of the entries of $V$ and identities \ref{syn_values_MAIN_SHORT} give multiple necessary and sufficient conditions for determining the dynamics of the corresponding MAIN state $\psi \!=\! \rho^{-1}(V)$ in the intervals $I_1$ and $I_2$, respectively. Due to the model's symmetry (Remark \ref{symmetry}) $V$ is the image of either a symmetrical or an asymmetrical state $\psi$. In the latter case, there exists a matrix $V' \in \Omega$ for a state conjugate to $\psi$. One can easily show that $V'$ is simply defined given $V$ by swapping the first (second) row of $V_1$ with the second (first) row of $V_2$. Notably, both $\psi$ and $\psi'$, and thus also $V$ and $V'$, exist under the same parameter conditions. The second rows of Table \ref{tab:MAIN_table} shows all matrices $V \in \Omega$ that are an image of either of a symmetrical state or one of two conjugate states and their names (1st row). Given that $I$, $AP$ and $ID$ are the only symmetrical cycles (in-phase and anti-phase), from Remark \ref{symmetry} all other states have another existing conjugate cycles that exists under the same conditions.

\begin {table}[h]
\begin{center}
\begin{tabular}{ |m{1em}|m{3em}|m{3em}|m{3em}|m{3em}|m{3em}|m{3em}|m{3em}|m{3em}|m{3em}| } 
 \hline
  - & \centering $S$ & \centering $SB$ & \centering $SD$ & \centering $AP$ & \centering $AS$ & \centering $ASD$ & \centering $I$ & \centering $ID$ & $\quad IB$ \\ \hline

 \begin{turn}{90} Matrix \vspace{1mm} \end{turn} 
 & $\begin{array}{cccc}
  1100 \\ 
  0000 \\
 \end{array}$
 & $\begin{array}{cccc}
  1100 \\ 
  1100 \\
 \end{array}$
 & $\begin{array}{cccc}
  1100 \\ 
  0100 \\
 \end{array}$
 & $\begin{array}{cccc}
  1100 \\ 
  0011 \\
 \end{array}$
 & $\begin{array}{cccc}
  1111 \\ 
  0011 \\
 \end{array}$
 & $\begin{array}{cccc}
  1101 \\ 
  0011 \\
 \end{array}$
 & $\begin{array}{cccc}
  1111 \\ 
  0000 \\
 \end{array}$
 & $\begin{array}{cccc}
  1101 \\ 
  0111 \\
 \end{array}$
 & $\begin{array}{cccc}
  1111 \\ 
  1111 \\
 \end{array}$
 \\ \hline
 
 \begin{turn}{90} Conditions \vspace{1mm} \end{turn} & 
 $\begin{aligned}[t] 
  C_1 \!&<\! \theta \\ 
  C_2^+ \!&<\! \theta \\ 
  C_3^+ \!&<\! \theta 
 \end{aligned}$ & 
 
 $\begin{aligned}[t] 
  C_3^+ \!&<\! \theta \\ 
  C_8^- \!&\geq\! \theta 
 \end{aligned}$ & 

 $\begin{aligned}[t] 
  C_4^- \!&\geq\! \theta \\ 
  C_2^- \!&\geq\! \theta \\ 
  C_3^+ \!&<\! \theta \\ 
  C_8^- \!&<\! \theta  
 \end{aligned}$ & 
 
 $\begin{aligned}[t] 
  C_2^+ \!&<\! \theta \\ 
  C_3^- \!&\geq\! \theta 
 \end{aligned}$ &  
 
 $\begin{aligned}[t] 
  C_3^- \!&\geq\! \theta \\ 
  C_5^+ \!&<\! \theta \\ 
  C_8^- \!&\geq\! \theta 
 \end{aligned}$ & 
 
 $\begin{aligned}[t] 
  C_2^- \!&\geq\! \theta \\ 
  C_3^- \!&\geq\! \theta \\ 
  C_5^+ \!&<\! \theta \\ 
  C_8^- \!&<\! \theta 
 \end{aligned}$ & 
 
 $\begin{aligned}[t] 
  C_1 \!&\geq\! \theta \\ 
  C_6^+ \!&<\! \theta 
 \end{aligned}$ &  
 
 $\begin{aligned}[t] 
  C_3^- \!&\geq\! \theta \\ 
  C_5^- \!&\geq\! \theta \\ 
  C_7^- \!&<\! \theta 
 \end{aligned}$ &  
 
 $\begin{aligned}[t] 
  C_7^- \!\geq\! \theta 
 \end{aligned}$ 
 \\ \hline 
 \begin{turn}{90} Short \vspace{1mm} \end{turn} &  
 $-$ & 
 $C_{9}\!<\!\theta$ & 
 $C_{9}\!<\!\theta$ & 
 $-$ & 
 $\begin{aligned}[t] C_{10}\!<\!\theta \end{aligned}$ &
 $\begin{aligned}[t] C_{10}\!<\!\theta \end{aligned}$ &
 $-$ & 
 $\begin{aligned}[t] C_{10}\!<\!\theta \end{aligned}$ &
 $\begin{aligned}[t] C_{10}\!<\!\theta \end{aligned}$ 
 \\ \hline
\end{tabular}
\end{center}
\caption {Matrix form and existence conditions of all $2T\!R$-periodic SHORT MAIN states. Names (first row) were chosen following our proposed link between states and percepts in auditory streaming (see Section \ref{motivation}). Names starting with S correspond to segregation (no unit responds to both tones), I to integration (one unit responds to both tones, the other is inactive or responding to both tones, too) and AP to bistability (one unit responds to both tones, the other unit to every other tone). The letter D corresponds to states for which one units turns ON with a small delay after the other unit in one active tone interval. The letter B corresponds to states for which both units follow the same dynamics. }
\label{tab:MAIN_table}
\end {table}

In the next part we define the conditions for existence of each of the states reported in the third row of Table \ref{tab:MAIN_table}, which are equivalent to the well-definedness conditions of the corresponding matrix form $V \in \Omega$. These conditions depend on:

\begin{small}
\begin{equation} 
\label{eq1}
 \begin{split}
  C_1 \!=\! d, \quad
  C_2^{\pm} \!=\! a\!-\!bM^{\pm}\!+\!d, \quad 
  C_3^{\pm} \!=\! \!c\!-\!bN^{\pm}, \quad
  C_4^{\pm} \!=\! \!c\!-\!bM^{\pm}\!, \quad
  C_5^{\pm} \!=\! a\!-\!bN^{\pm}\!+\!d, \quad  \\
  C_6^{\pm} \!=\! a\!-\!bN^{\pm}\!+\!c, \quad
  C_7^{\pm} \!=\! \!d\!-\!bN^{\pm}\!, \quad
  C_8^{\pm} \!=\! \!d\!-\!bM^{\pm}\!, \quad
  C_9 \!=\! \!a-\!bM^{+} \quad  
  C_{10} \!=\!a\!-\!bN^{+}  
 \end{split}
\end{equation}
\end{small}

One determine conditions for the well-definiteness of each matrix $V \in \Omega$ from the definitions of the entries of $V_1$ and $V_2$ given in \ref{yA-x_A} and using formulas \ref{syn_values_MAIN_SHORT}. Notably, all the existence conditions uniquely depend on the system's parameters. When determining these conditions one notices that many of them are redundant, and can be simplified using the following properties: $N^- \!\geq\! N^+ \!\geq\! M^- \!\geq\! M^+$, $d \!\leq\! c$ and $a \geq 0$. In the next paragraph, we show one example ($AS$) and leave the remaining for the reader to prove. The names and the sets of inequalities defining each state is reported in the middle row of Table \ref{tab:MAIN_table}. We note such inequalities are well-posed, meaning that there is a region of parameter where they are all satisfied. This effectively proves that for each matrix $V \in \Omega$ there exists a state $\psi\!=\!\rho^{-1}(V) \in SM$ whose dynamics during intervals $I_1$ and $I_2$ are defined by the entries of $V$. 

We now prove that $AS$' existence conditions in Table \ref{tab:MAIN_table} are well-defined, that is:
\begin{equation}
 V_{AS}=\begin{bmatrix} [cc|cc]
    1 & 1 & 1 & 1 \\
    0 & 0 & 1 & 1
 \end{bmatrix} 
 \quad \Leftrightarrow \quad
 C_3^- \!\geq\! \theta, \quad
 C_5^+ \!<\! \theta, \quad
 C_8^- \!\geq\! \theta. \quad
 \label{AS}
\end{equation}
From condition (1) in \ref{MAIN_conditions} we have that
$$ x_A^1\!=\!1 \Rightarrow y_A^1\!=\!1, \quad x_A^2\!=\!1 \Rightarrow y_A^2\!=\!1, \quad x_B^2\!=\!1 \Rightarrow y_B^2\!=\!1, \quad y_B^1\!=\!0 \Rightarrow x_B^1\!=\!0, $$
This obviously leads to the follow equivalence
\begin{equation}
 V_{AS}=\begin{bmatrix} [cc|cc]
    1 & 1 & 1 & 1 \\
    0 & 0 & 1 & 1
 \end{bmatrix} 
 \quad \Leftrightarrow \quad
 x_A^1 \!=\! 1, \quad
 x_B^2 \!=\! 1, \quad
 x_A^2 \!=\! 1, \quad
 y_B^1\!=\!0.
 \nonumber
\end{equation}
Using the definition of the entries defined in \ref{yA-x_A} and the identities for the synaptic quantities given in equations \ref{syn_values_MAIN_SHORT} we observe the following: 
\vspace{2mm} \begin{enumerate}
 \item $y_A^1\!=\!1 \left( y_B^2\!=\!1 \right) \Rightarrow s_A^{2-}\!=\!N^- \left( s_B^{1-}\!=\!N^- \right) $, which implies $x_B^2\!=\!x_A^1\!=\!H(c\!-\!bN^-)$
 \item $y_B^1\!=\!0 \mbox{ and } y_B^2\!=\!1 \Rightarrow s_B^{2-}\!=\!M^-$. From this $x_A^2\!=\!H(d\!-\!bM^-)$  
 \item $y_A^2\!=\!1 \Rightarrow s_A^{1+}\!=\!N^+$. This and $y_B^1\!=\!0$ give $y_B^1\!=\!H(a\!+\!d\!-\!bN^+)$
\end{enumerate}
\vspace{2mm} Overall, from the cases (1-3) above we obtain \vspace{-2mm}
\begin{equation}
 x_A^1 \!=\! 1, x_B^2 \!=\! 1 \Leftrightarrow C_3^- \!\geq\! \theta, 
 \quad x_A^2 \!=\! 1 \Leftrightarrow C_8^- \!\geq\! \theta, 
 \quad y_B^1\!=\!0 \Leftrightarrow C_5^+ \!<\! \theta.
 \nonumber
\end{equation}
This completes the proof for both the claim \ref{AS} and the Theorem.
\end{proof}

\begin{remark} [Conditions $C_9$ and $C_{10}$] \label{extra_cond_remark}
The middle row of Table \ref{tab:MAIN_table} shows existence condition that determine the dynamics of each state in the intervals $I_1$ and $I_2$. However, they do not guarantee these states being OFF in $[0,2T\!R]\!-\!I$ (ie being SHORT). From Lemma \ref{lem:LONG_conditions} there are two cases to consider: 
\begin{enumerate}
 \item If both units turn ON during interval $I_1$ or $I_2$ one must guarantee that the second condition of Lemma \ref{lem:LONG_conditions} is not valid in each interval $I=[\alpha,\beta]=I_1$ or $I_2$ during which this occurs, one must impose 
 \begin{equation}
 \min \{ a\!-\!b s_A(\beta\!-\!D),a\!-\!b s_B(\beta\!-\!D) \} \!<\! \theta 
 \label{extra_conditions}
 \end{equation}
 This condition is expressed differently for each MAIN state in Table~\ref{tab:MAIN_table}:
 \begin{itemize}
  \item For $SB$ and $SD$ both units turn ON during interval $I_1$ ($I_2$ for their conjugate state). Equations \ref{syn_values_MAIN_SHORT} lead to $s_A(T\!D\!-\!D)\!=\!s_B(T\!D\!-\!D)=M^+$. Thus condition $C_9\!<\!0$ guarantees that inequalities \ref{extra_conditions} are satisfied. 
  \item For $AS$ and $ASD$ both units are ON during $I_2$ ($I_1$ for their conjugate state). To guarantee condition \ref{extra_conditions} at time $\beta\!=\!T\!R\!+\!T\!D$ one notices that equations \ref{syn_values_MAIN_SHORT} give $s_A(T\!R\!+T\!D\!-\!D)\!=N^+$ and $s_B(T\!R\!+T\!D\!-\!D)=M^+$. Thus, condition $C_{10}\!<\!0$ guarantees that condition \ref{extra_conditions} is satisfied. 
  \item For states $ID$ and $IB$ we notice that condition \ref{extra_conditions} is symmetrical on both intervals $I_1$ and $I_2$. Thus we may restrict the study on interval $I_1$. Similar to the two previous cases the application of equations \ref{syn_values_MAIN_SHORT} gives $s_A(T\!D\!-\!D)\!=\!s_B(T\!D\!-\!D)\!=\!N^+$. Thus we obtain $C_{10}\!<\!0$.
 \end{itemize}
 The bottom row of Table \ref{tab:MAIN_table} contains the additional conditions on $C_9$ and $C_{10}$ to be applied to each of the states analysed above. 
 \item If during both intervals $I_1$ and $I_2$ at least one unit is OFF the first condition of Lemma \ref{lem:LONG_conditions} is not satisfied, thus the state is SHORT with no extra conditions. These considerations hold for $S$, $AP$ and $I$. 
\end{enumerate}	
\end{remark}

\begin{remark} [Table \ref{tab:MAIN_table}]
The conditions in the middle and bottom rows of Table \ref{tab:MAIN_table} complete the existing conditions for all $2T\!R$-periodic SHORT MAIN states. Indeed these conditions covers all possible matrix forms and corresponding states. The middle row shows conditions determining the dynamics within in the intervals $I_1$ and $I_2$. The bottom row shows conditions that guarantee units to be OFF in $[0,2T\!R]\!-\!I$. 
\end{remark}

Figure \ref{figure4}A shows time histories for each $2T\!R$-periodic SHORT MAIN states in Table \ref{tab:MAIN_table}. We note that the conditions given in this table allow us to determine the regions where each of these states exists in the parameter space. To visualize 2-dimensional existence regions when varying pairs of model parameters we defined a new parameter $D\!F \in [0,1]$ and set $d\!=\!cD\!F$ ($D\!F$ is a scaling factor for the inputs from tonotopic locations). Figure \ref{figure4}B shows the two dimensional region of existence of states of each of these states at varying $DF$ and input strength $c$. 

\begin{figure}[htbp]
  \centering
  \includegraphics[width=0.8\linewidth]{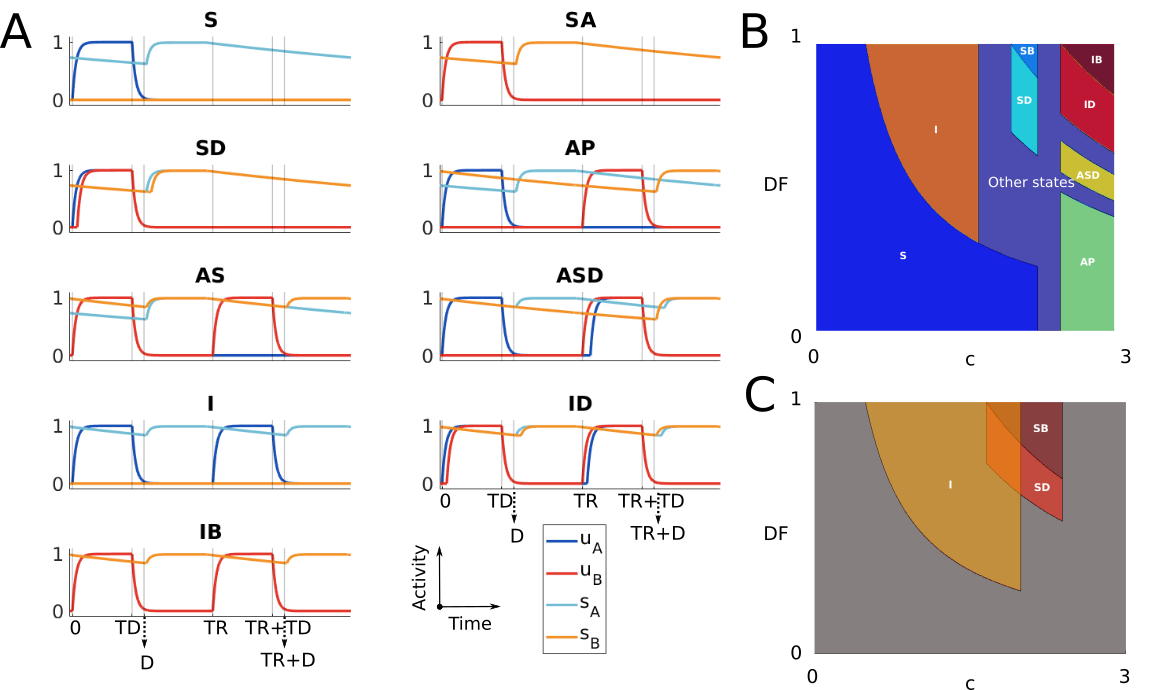}
  \caption{A. Time histories of $2T\!R$-periodic SHORT MAIN states B. Existence regions of states in A. when varying $DF$ and $c$. C. Existence regions for states $I$, $S\!B)$ and $S\!D$ at varying $c$ and $D\!F$. Parameters in B and C are $\tau_i\!=\!0.4$, $\theta\!=\!0.5$, and in B (C)  $T\!D \!=\! 0.03 \, (0.005)$, $D \!=\! 0.03 \, (0.015)$, $P\!R\!=\!17 \, (5)$, $a \!=\! 0.6 \, (0.4)$, $b \!=\! 2 \, (3)$. }
 \label{figure4}
\end{figure}

From Table \ref{tab:MAIN_table} we can establish the coexistence of MAIN states, as shown in the next theorem.

\begin{theorem} [Multistability] The state $I$ may coexist
  with $S\!B$ or $S\!D$. Any other pair of $2T\!R$-periodic SHORT MAIN
  states cannot coexist.
\end{theorem}

The proof of this theorem is in the Supplementary Material \ref{multistability_appendix}. Figure \ref{figure4}C shows a parameter regime show which states $I$ coexists with $S\!B$ and $S\!D$. 

The analysis for $2T\!R$-periodic SHORT CONNECT states is similar to that of SHORT MAIN states, which we now summarize.

\begin{theorem}  \label{CONNECT_2TR}
There is an injective map: 
$$ \varphi \colon SC \rightarrow B(2,6) \mbox{,} \quad
\psi \mapsto W =
\begin{bmatrix} [c|c]
    W_1 & W_1 
\end{bmatrix}
=
\begin{bmatrix} [ccc|ccc]
    x_A^1 & y_A^1 & z_A^1 & x_A^2 & y_A^2 & z_A^2 \\
    x_B^1 & y_B^1 & z_B^1 & x_B^2 & y_B^2 & z_B^2
\end{bmatrix} $$
Where, for $i\!=\!1,2$, $W_i$ is the matrix forms of $\psi$ in $I_i$ defined in \ref{matricial_representation_CONN}. Then:
$$Im(\varphi) \!=\! \{W\!=\!\varphi(\psi) \mbox{, where } W \mbox{ is one of the matrices shown in Table \ref{tab:CONNECT_table}} \}$$

\begin {table}[h]
\begin{tabular}{ |c|c|c|c|c|c|c|c|c| } 
 \hline
  \centering \scriptsize $ZcS^*$ & \centering \scriptsize $ZcAP$ & \centering \scriptsize $ZcAS^*$ & \centering \scriptsize $ZcI$ & \centering \scriptsize $ScAS^*$ & \centering \scriptsize $SDcAS^*$ & \centering \scriptsize $ScSD^*$ & \centering \scriptsize $APcAS^*$ & \scriptsize $APcI$ \\ \hline
  \centering \parbox{1cm}{\scriptsize $\begin{array}{c} 
  001 000 \\
  001 000 \\
 \end{array}$}
 & \centering \parbox{1cm}{\scriptsize $\begin{array}{c}
  001 000 \\
  000 001 \\
 \end{array}$}
 & \centering \parbox{1cm}{\scriptsize $\begin{array}{c}
  001 001 \\
  000 001 \\
 \end{array}$}
 & \centering \parbox{1cm}{\scriptsize $\begin{array}{c}
  001 001 \\
  001 001 \\
 \end{array}$}
 & \centering \parbox{1cm}{\scriptsize $\begin{array}{c}
  001 111 \\
  000 001 \\
 \end{array}$}
 & \centering \parbox{1cm}{\scriptsize $\begin{array}{c}
  001 111 \\
  000 011 \\
 \end{array}$}
 & \centering \parbox{1cm}{\scriptsize $\begin{array}{c}
  111 000 \\
  001 000 \\
 \end{array}$}
 & \centering \parbox{1cm}{\scriptsize $\begin{array}{c}
  111 001 \\
  000 111 \\
 \end{array}$}
 & \parbox{1cm}{\scriptsize $\begin{array}{c} 
  111 001 \\
  001 111 \\
 \end{array}$}
 \\ \hline
\end{tabular}
\caption {Matrix form of $2T\!R$-periodic SHORT CONNECT states (* asymmetrical states). These states' dynamics is connecting branches of pairs of MAIN states or the inactive state ($Z$). For example $ZcS$ connects the inactive state with one of the segregated states $S$, $SB$ or $SD$. The names chosen for CONNECT states (first row) contains the names of the two MAIN states separated by the letter c. }
\label{tab:CONNECT_table}
\end {table}
\end{theorem}

A complete version of this theorem (similar to Theorem \ref{MAIN_2TR}) proving the existence conditions for all SHORT CONNECT states is in the Supplementary Material \ref{appendix1}. Table \ref{tab:CONNECT_table} shows names (first row) and matrix forms (second row) of all possible $2T\!R$-periodic SHORT CONNECT states. We omit time histories for these states because they can be visualized from their matrix form (see Remark \ref{visualization_matr}).

\subsection{LONG MAIN states}
The analysis of LONG states is an extension of the SHORT states' one. In this section we briefly report the main ideas, the details in the Supplementary Material \ref{appendix2}. The first step is to extend the matrix form definition to LONG states by including a last column in the matrix form of SHORT MAIN states. This new column is selected to satisfy the properties of LONG states described in Lemma \ref{lem:LONG_conditions}. The matrix form for a state $\psi \in LM$ is the $2 \times 6$ binary matrix $V$ defined as
\begin{equation}
V=
\begin{bmatrix} [c|c||c|c]
    V_1 & \vec{w}^1 & V_2 & \vec{w}^2
\end{bmatrix}
=
\begin{bmatrix} [cc|c||cc|c]
    x_A^1 & y_A^1 & w^1 & x_A^2 & y_A^2 & w^2 \\
    x_B^1 & y_B^1 & w^1 & x_B^2 & y_B^2 & w^2
\end{bmatrix}
\nonumber
\end{equation}
Where $V_1$ and $V_2$ are the same matrix forms defined for MAIN SHORT states and the binary vectors $ \vec{w}^1$ and $\vec{w}^2$ are defined by
\begin{equation}
 w^1=H(ay_A^1-bs_A^{1+})H(ay_B^1-bs_B^{1+}) \quad  \mbox{and} \quad w^2=H(ay_A^2-bs_A^{2+})H(ay_B^2-bs_B^{2+}).
\end{equation}
We remind the reader that $s_A^{1+}\!=\!s_A(T\!D\!-\!D)$, $s_B^{1+}\!=\!s_B(T\!D\!-\!D)$, $s_A^{2+}\!=\!s_A(T\!R\!+\!T\!D\!-\!D)$ and $s_B^{2+}\!=\!s_B(T\!R\!+\!T\!D\!-\!D)$. 
Using a similar proof as the one of Theorem \ref{MAIN_2TR} we can use the matrix form to define the existence conditions of the states and exclude impossible ones. Table \ref{tab:MAIN_LONG_table} contains the names and matrix form of all the possible LONG MAIN states, and their existence conditions are reported in Table \ref{tab:LM_SC_table_main}. 


\begin {table}[h]
\begin{tabular}{ |c|c|c|c|c|c|c|c|c| } 
 \hline
  \centering \scriptsize $IL_1$ & \centering \scriptsize $IL_2^*$ & \centering \scriptsize $ASDL_1^*$ & \centering \scriptsize $ASL^*$ & \centering \scriptsize $SL^*$ & \centering \scriptsize $IDL_1$ & \centering \scriptsize $IDL_2^*$ & \centering \scriptsize $ASDL_2^*$ & \scriptsize $SDL^*$ \\ \hline
 \ \parbox{1cm}{\scriptsize $\begin{array}{c} 
  111 111 \\
  111 111 \\
 \end{array}$}
 & \centering \parbox{1cm}{\scriptsize $\begin{array}{c}
  111 110 \\
  111 110 \\
 \end{array}$}
 & \centering \parbox{1cm}{\scriptsize $\begin{array}{c}
  111 010 \\
  111 110 \\
 \end{array}$}
 & \centering \parbox{1cm}{\scriptsize $\begin{array}{c}
  111 000 \\
  111 110 \\
 \end{array}$}
 & \centering \parbox{1cm}{\scriptsize $\begin{array}{c}
  111 000 \\
  111 000 \\
 \end{array}$}
 & \centering \parbox{1cm}{\scriptsize $\begin{array}{c}
  111 011 \\
  011 111 \\
 \end{array}$}
 & \centering \parbox{1cm}{\scriptsize $\begin{array}{c}
  111 010 \\
  011 110 \\
 \end{array}$}
 & \centering \parbox{1cm}{\scriptsize $\begin{array}{c}
  111 000 \\
  011 110 \\
 \end{array}$}
 & \parbox{1cm}{\scriptsize $\begin{array}{c}
  111 000 \\
  011 000 \\
 \end{array}$}
 \\ \hline
\end{tabular}
\caption {Matrix form of the $2T\!R$-periodic LONG MAIN states (* asymmetrical states). Names (first row) are the same as the corresponding MAIN states, except from adding the final letter L, and an a subscript number to differentiate LONG states corresponding to the same MAIN state. }
\label{tab:MAIN_LONG_table}
\end {table}

The existence conditions of SHORT CONNECT and LONG MAIN states can be visualized as a 2D parameter projection, similar to Figure \ref{figure4}B for SHORT MAIN states. Figure \ref{figure6}A,C show two examples when varying parameters $(c,D\!F)$, and the remaining parameters have been fixed to satisfy $T\!D\!<\!D$ and $T\!D\!+\!D\!<\!T\!R$. Panels A. and C. respectively show the existence regions for SHORT CONNECT and LONG MAIN states. In panel A. SHORT MAIN states are shown in dark blue to help the comparison with Figure \ref{figure4}B (same parameters). Figure \ref{figure6}B,D show time histories for the SHORT CONNECT state $APcAS$ and the LONG MAIN state $SDL$.

\begin{figure}[htbp]
  \centering
  \includegraphics[width=0.7\linewidth]{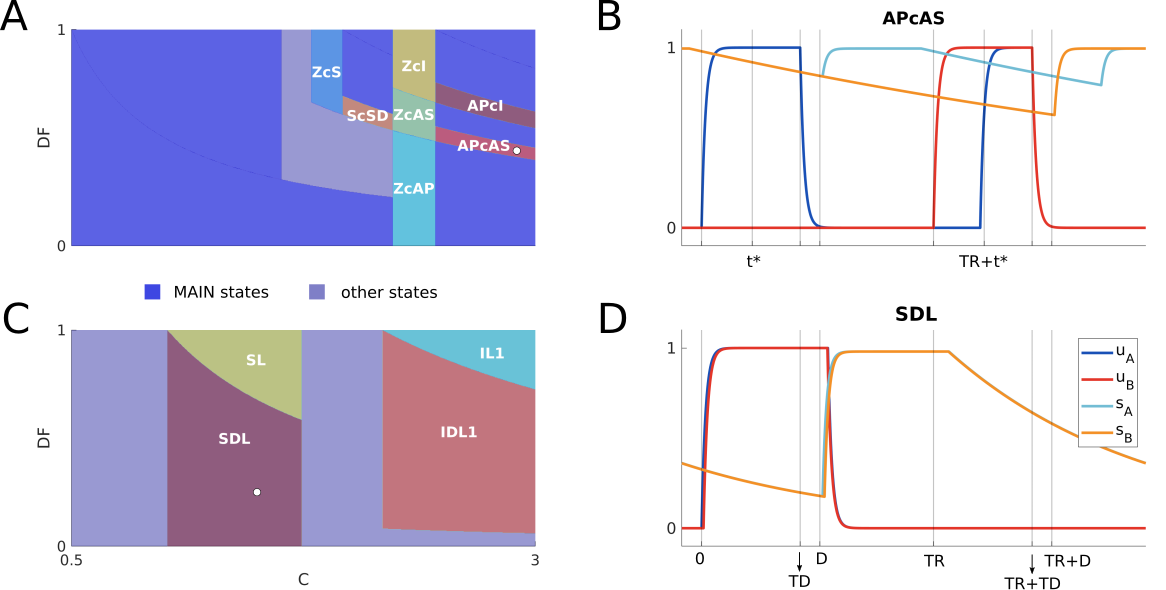}
  \caption{Visualization of SHORT CONNECT and LONG MAIN states. Panels A. and C. show regions of existence respectively for SHORT CONNECT and LONG MAIN states. SHORT MAIN states in panel A. are shown in dark blue.  States that are neither SHORT MAIN, SHORT CONNECT or LONG MAIN are shown in light blue. For the same parameters as in A. and D., panels B. and D. respectively show example time histories for a SHORT CONNECT state ($APcAS$) and a LONG MAIN state ($SDL$) with fixed $(c,D\!F)$ shown by white dots in A. and D. In panel A. the parameters are the same as in Figure \ref{figure4}B. In panel C. parameters are the same as in A. except for $\tau_i\!=\!0.05$ and $a\!=\!2$.}
  \label{figure6}
\end{figure}

\begin{remark} [CONNECT states] \label{MAIN_CONNECT_regions}
By comparing Figure \ref{figure6}A with Figure \ref{figure4}B (same parameters) we note that the union of the regions of existence of MAIN states is larger than the one of CONNECT states, hence why we call the first group MAIN. In addition SHORT CONNECT states connect branches of SHORT MAIN states, hence why we called them CONNECT (see Table \ref{tab:CONNECT_table}). 
\end{remark}

\subsubsection{Remaining states}
As shown in the Section \ref{sec_2TR_states}, $2T\!R$-periodic states can be SHORT MAIN ($SM$), SHORT CONNECT ($SC$), LONG MAIN ($LM$) or LONG CONNECT ($LC$) during each interval $I_1$ and $I_2$. We define $X|Y$ the set of states satisfying condition X during $I_1$ and Y during $I_2$, where $X,Y \in \{ SM,SC,LM,LC \}$. In Section \ref{sec_2TR_states} we have  the existence conditions of all possible states in some of these sets. More precisely:
\begin{itemize}
 \item The analysis of $SM|SM$ is summarized in Table \ref{tab:MAIN_table}
 \item The analysis of $SC|SM$, $SM|SC$ and $SC|SC$ is summarized in Table \ref{tab:CONNECT_table}
 \item The analysis of $LM|LM$, $SM|LM$ and $LM|SM$ is summarized in Table \ref{tab:MAIN_LONG_table}
\end{itemize}
The analysis of all remaining combinations of sets $X|Y$ are in the Supplementary Material \ref{appendix3} and concludes the existence conditions for all $2T\!R$-periodic states.

\section{Biologically relevant case: $2T\!R$-periodic states for $D\!\leq\!TD$} \label{biologically_relevant}
In this section we study model states and their link to auditory streaming under (1) $D\!\leq\!TD$ and (2) $T\!D\!+\!D\!<\!T\!R$. These inequalities are relevant to studying auditory streaming: condition (1) because delayed inhibition would be caused by factors that generate short delays, leading, condition (2) is guaranteed for the values of $T\!D$ and $T\!R$ typically tested in these experiments (further motivated in the Discussion).  

 By assuming that tonotopic inputs to the units are stronger than their mutual inhibition we derive analytically the existence conditions of all possible $2T\!R$-periodic states (Table \ref{tab:main_case2_1} and \ref{tab:MAIN_CONNECT_case2}). Overall, we find a total of 10 possible states (shown in Figure \ref{figure7}A). We link these states with the possible perceptual outcomes in the auditory streaming paradigm and find a qualitatively agreement between model and experiments when varying inputs' parameters $df$ and $P\!R$ (Figure \ref{figure7}B and C). Furthermore, the states' existence conditions let us formulate the coherence and fission boundaries separating the percepts as functions of $P\!R$ (Equations \ref{boundary_eqn}). 
 
 We now proceed to determine the detailed analysis of these $2T\!R$-periodic states. We consider active tone intervals $I\!=\!I_1 \cup I_2$, where $I_1\!=\![0,T\!D]$ and $I_2\!=\![T\!R,T\!R\!+\!T\!D]$. We assume that 
 \begin{equation}
	 c\!-\!b\!\geq\!\theta
	 \label{U3}
 \end{equation}
 a condition that allows unit A (B) to turn and remain ON at each A (B) active tone interval $I_1$ ($I_2$). Indeed from the model equations (\ref{model})--(\ref{inputs}), $\forall t \in I_1$, the total input to the A unit is $au_B\!-\!bs_B(t\!-\!D)\!+\!c\!\geq\!c\!-\!b$. Thus on the fast time scale, the A unit turns ON instantaneously at the start of $I_1$ and remains ON $\forall t \in I_1$. For analogous reasons the B unit is ON throughout $I_2$. This has two important consequences:
 \begin{enumerate}
  \item The synaptic variables $s_A(t\!-\!D)$ and $s_B(t\!-\!D)$ are constant and equal to 1 in $[D,T\!D\!+\!D]$ and $[T\!R\!+\!D,T\!R\!+\!T\!D\!+\!D]$, respectively. This implies that the total inputs to the B and A units are equal to $a\!-\!b\!+\!d$ in these intervals. 
  
  \item Both units are OFF $\forall t \in \mathbb{R}\!-\!I$ (i.e. no LONG states can exist). Indeed from point 1. above $s_A(t\!-\!D)$ ($s_B(t\!-\!D)$) is equal to 1 at time $T\!D$ ($T\!R\!+\!T\!D$) and the total input to the B (A) unit at this time is thus $a\!-\!b$, which is less than $\theta$ due to hypothesis \ref{U2}. Thus the B  (A) unit turns OFF instantaneously at time $T\!D$ ($T\!R\!+\!T\!D$), and it is followed by A (B) due to Section \ref{basin_01}. Since $(0,0)$ is an equilibrium for the fast subsystem with no input (see Section \ref{fast_dynamics_I}), we conclude that both unit are OFF until the next active tone input. 
 \end{enumerate}
 From point 1. the input to the B (A) unit in $[D,T\!D\!+\!D]$ ($[T\!R\!+\!D,T\!R\!+\!T\!D\!+\!D]$) is equal to $P\!=\!a\!-\!b\!+\!d$. This and point 2. imply that B and A can turn ON only in the intervals $L_1\!=\![0,D]$ and $L_2\!=\![T\!R,T\!R\!+\!D]$, respectively. We consider two cases.

\subsubsection{Case $P\!\geq\!\theta$} Since unit B is ON in $I_2$, unit A is ON in this interval, since its total input is $a\!-\!bs_A(t\!-\!D)\!+\!d\!\geq\!P\!\geq\!\theta$. This is true also for unit B in $I_1$. Moreover both unit turn OFF instantaneously at times $T\!D$ and $T\!R\!+\!T\!D$ (see point 2. above). Thus units evolve equally on each active tone interval (on the fast time scale). The only difference is that B (A) may turn ON a small delay $\delta \sim \tau$ after A (B) in $I_1$ ($I_B$). When evaluated at time $0$ ($T\!R$) the delayed variable $s_A$ ($s_B$) is equal to $N^-$. Due to the model symmetry there are only two possible states: $I$ and $I\!D$. For $I$ both units instantaneously turn ON at same time $0$ and $T\!R$, which occurs when $d\!-\!bN^-\!\geq\!\theta$ ($C_7^-\!\geq\!\theta$). If $d\!-\!bN^-\!<\!\theta$ we have the state $I\!D$, for which B (A) turns ON a small delay $\delta$ after A (B) in $I_1$ ($I_2$).
 
\begin {table}[h] \centering
\begin{tabular}{ |c|c| } 
 \hline
  \footnotesize $I$ & \footnotesize $ID$ \\ \hline
 $\begin{aligned}[t] 
  C_7^- \!\geq\! \theta \\
  P \!\geq\! \theta
 \end{aligned}$ & 

 $\begin{aligned}[t] 
  C_7^- \!<\! \theta \\
  P \!\geq\! \theta
 \end{aligned}$ 
 \\ \hline 
\end{tabular}
\caption {MAIN states existence conditions for $D\!<\!T\!D$ and $T\!D\!+\!D\!<\!T\!R$ and $P\!\geq\!\theta$. }
\label{tab:main_case2_1}
\end {table}

\subsubsection{Case $P\!<\!\theta$} In this case the B (A) unit is OFF in $[D,T\!D]$ ($[T\!R\!+\!D,T\!D]$) and outside the active tone intervals. The dynamics of the B and A units during the intervals $L_1$ and $L_2$ respectively is yet to be determined. Lemma \ref{lem:syn_decay} proves that the delayed synaptic variables are monotonically decaying in each of these intervals. We can use the classification of MAIN and LONG states presented in Sections \ref{classification_interval} by replacing interval $I$ with $L$, where $L\!=\!L_1$ or $L\!=\!L_2$. We fix $L\!=\!L_1$ ($L\!=\!L_2$). Since the A (B) unit is ON in $L$ due to condition \ref{U3}, MAIN states in $L$ can satisfy only conditions $M_1$, $M_2$ and $M_4$ ($M_1$, $M_3$ and $M_5$), since only these states are ON in $L$. By the same reasoning CONNECT states in $L$ can satisfy only condition $C_1$ ($C_2$). The matrix form of MAIN states can be extended to a $2 \times 3$ binary matrix (see Remark \ref{link_main_connect_matr}). Moreover, since A (B) is ON in $L_1$ ($L_2$) due to condition \ref{U3}, the matrix form of any $2T\!R$-periodic MAIN and CONNECT state can be written as
 $$ \begin{bmatrix} [ccc|ccc]
    1 & 1 & 1 & x_A^2 & y_A^2 & z_A^2 \\
    x_B^1 & y_B^1 & z_B^1 & 1 & 1 & 1 
  \end{bmatrix} $$ 
 The synaptic quantities defining the entries of the matrix form in $L_1$ and $L_2$ are
 \begin{equation}
  s_A^{2 \pm}\!=\!s_B^{1 \pm}\!=\!N^{\pm}, \quad
  s_A^{1 \pm}\!=\!\begin{cases} 
	R^{\pm} & \mbox{if } z_A^2=1 \\ 
	M^{\pm}  & \mbox{otherwise} 
  \end{cases}
  \quad \mbox{and} \quad
  s_B^{2 \pm}\!=\!\begin{cases} 
	R^{\pm} & \mbox{if } z_B^1=1 \\ 
	M^{\pm}  & \mbox{otherwise} 
  \end{cases}
  \label{syn_quantities_case2}
 \end{equation}
 Where $R^-\!=\!e^{-(T\!R\!-\!2D)/\tau_i}$ and $R^+\!=\!e^{-(T\!R\!-\!D)/\tau_i}$. Quantities $M^\pm$ and $N^\pm$ are defined in equations \ref{syn_values}. The proof of these identities is in the Supplementary Material \ref{appendix4}. By applying identities \ref{syn_quantities_case2} to the definition of the entries of the matrix form of MAIN or CONNECT states we obtain that $ z_A^2 = z_B^1 \Rightarrow x_A^2 = x_B^1 \mbox{ and } y_A^2 = y_B^1 $. 
 
 This condition reduces the total number of combination of binary matrices (and relative MAIN and CONNECT states) to the ones shown in Table \ref{tab:MAIN_CONNECT_case2}. The first 5 states in this table are MAIN and the last two are CONNECT and complete the set of all possible states. Using the identities \ref{syn_quantities_case2} on the definition of the entries in each state's matrix form and applying simplifications (i.e. the same analysis carried out in the previous sections) implies the existence conditions shown in the bottom row of Table \ref{tab:MAIN_CONNECT_case2}, where $ R_6^- = a-bR^-+d$ and $R_7^- = d-bR^-. $
  
 \begin {table}[h] \centering
\begin{tabular}{ |c|c|c|c|c|c|c| } 
 \hline
  \footnotesize $IS$ & \footnotesize $IDS$ & \footnotesize $AS^*$ & \footnotesize $ASD^*$ & \footnotesize $AP$ & \footnotesize $APcAS^*$ & \footnotesize $AScI$ \\ \hline
  
\ \parbox{1.3cm}{\footnotesize $\begin{array}{c} 
  111 | 111 \\
  111 | 111 \\
 \end{array}$}
 & \parbox{1.3cm}{\footnotesize $\begin{array}{c}
  111 | 011 \\
  011 | 111 \\
 \end{array}$}
 & \parbox{1.3cm}{\footnotesize $\begin{array}{c}
  111 | 000 \\
  111 | 111 \\
 \end{array}$}
 & \parbox{1.3cm}{\footnotesize $\begin{array}{c} 
  111 | 000 \\
  011 | 111 \\
 \end{array}$}
 & \parbox{1.3cm}{\footnotesize $\begin{array}{c}
  111 | 000 \\
  000 | 111 \\
 \end{array}$}
 & \parbox{1.3cm}{\footnotesize $\begin{array}{c} 
  111 | 000 \\
  001 | 111 \\
 \end{array}$}
 & \parbox{1.3cm}{\footnotesize $\begin{array}{c}
  111 | 001 \\
  001 | 111 \\
 \end{array}$}
 \\ \hline
 $ \begin{aligned}[t] 
  R_7^- \!\geq\! \theta \\
  P \!<\! \theta
 \end{aligned}$ & 

 $ \begin{aligned}[t] 
  R_7^- \!<\! \theta \\
  R_6^- \!\geq\! \theta \\
  P \!<\! \theta
 \end{aligned}$ & 

 $ \begin{aligned}[t] 
  C_5^+ \!<\! \theta \\
  C_8^- \!\geq\! \theta
 \end{aligned}$ & 
 
 $\begin{aligned}[t] 
  C_5^+ \!<\! \theta \\
  C_8^- \!<\! \theta \\
  C_2^- \!\geq\! \theta
 \end{aligned}$ & 

 $ \begin{aligned}[t] 
  C_2^+ \!<\! \theta
 \end{aligned}$ & 

 $ \begin{aligned}[t] 
  C_2^- \!<\! \theta \\
  C_2^+ \!\geq\! \theta \\
 \end{aligned}$ & 

 $ \begin{aligned}[t] 
  R_6^- \!<\! \theta \\
  C_5^+ \!\geq\! \theta \\
 \end{aligned}$ 
 \\ \hline 
\end{tabular}
\caption {Matrix forms of MAIN/CONNECT states for $D\!<\!T\!D$, $T\!D\!+\!D\!<\!T\!R$ and $P\!\geq\!\theta$. Asymmetrical states in *.}
\label{tab:MAIN_CONNECT_case2}
\end {table}

Figure \ref{figure7}A shows time histories for the states presented in Tables \ref{tab:main_case2_1} and \ref{tab:MAIN_CONNECT_case2}. Since the A(B) unit must be ON during the A(B) active tone interval for property \ref{U3} we there are no possible other network states. A proof analogous to that of multistability theorem in the Supplementary Material \ref{multistability_appendix} shows that all of these states exist in non-overlapping parameter regions. 

 \begin{figure}[htbp]
  \centering
  \includegraphics[width=0.9\linewidth]{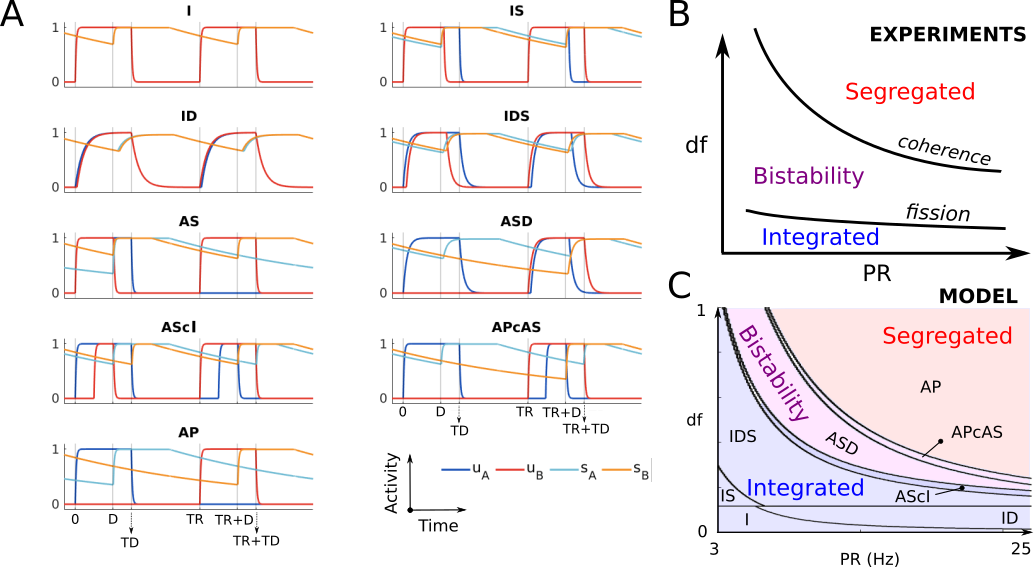}
  \caption{A. Time histories of all $2T\!R$-periodic states for $D\!<\!T\!D$ and $T\!D\!+\!D\!<\!T\!R$. B. Schematic diagram of the experimentally measured perceptual regions when varying $P\!R$ and $df$. C. Existence regions of the states in A. at varying $P\!R$ and $df$. States corresponding to integration, segregation or bistability are grouped by background colours (see Remark \ref{bistability_model}). Model parameters in C are: $\tau_i\!=\!0.2$, $\theta\!=\!0.5$, $T\!D\!=\!0.03$, $D\!=\!0.01$, $c\!=\!5$, $a\!=\!1$, $b\!=\!2$ and $m\!=\!6$. }
  \label{figure7}
\end{figure} 

\begin{remark} [Extension to the case $T\!D\!+\!D\!\geq\!T\!R$] \label{extension}
The condition $T\!D\!+\!D\!<\!T\!R$ enabled us to obtain a complete classification of network states via the application of Lemma \ref{lem:syn_decay}. However these states can exist also if $T\!D\!+\!D\!\geq\!T\!R$ with few adjustments in their existence conditions (see Supplementary Material \ref{appendix_TD_D_gr_TR}). We note that under this condition other $2T\!R$-periodic states exist, such as states where both units turn ON and OFF multiple times during each active tone interval (not shown). Since the condition $T\!D\!+\!D\!\geq\!T\!R$ is met for high values of $P\!R$ for which $T\!R \sim T\!D$, we explored this condition using computational tools (see Section \ref{comp_analysis}). 
\end{remark}

 \subsection{Model states and link with auditory streaming} \label{boundaries}
 We now show how states described in the previous section can explain the emergence of different percepts during auditory streaming. In the following framework each possible percepts is linked ($\leftrightarrow$) with the units' activities in the corresponding state: 
\begin{itemize}
 \item Integration $\leftrightarrow$ both units respond to all tones ($I$, $ID$, $IS$, $IDS$ and $AScI$).
 \item Segregation $\leftrightarrow$ no unit respond to both tones ($AP$).
 \item Bistability $\leftrightarrow$ one unit respond to both tones the other to only one tone ($AS$, $ASD$ and $APcAS$). This interpretation is motivated further in Remark \ref{bistability_model}. 
\end{itemize} 
 Thus all model states presented in the previous section belong to one perceptual class. The cartoon in Figure \ref{figure7}B shows the experimentally detected regions of parameters $d\!f$ and $P\!R$ where participants are more likely to perceive integration, segregation or bistability (van Noorden diagram - see Introduction). We now validate our proposed framework of rhythm tracking by comparing model states consistent with different perceptual interpretations (percepts) in the $(d\!f,P\!R)$-plane. In these tests the model parameter $d$ is scaled by $df$ according to the monotonically decreasing function $d\!=\!c \cdot (1\!-\!df^{1/m})$, where $m$ is a positive integer and $df$ is a unitless parameter in $[0,1]$ (motivated in Section \ref{motivation}). Figure \ref{figure7}C shows regions of existence of model states when fixing all other parameters (as reported in the caption). States classified as integration, segregation and bistability are grouped by blue, red and purple background colors to facilitate the comparison with Figure \ref{figure7}B. The existence regions of states corresponding to integration and segregation qualitatively matches the perceptual organization in the van Noorden diagram. 
 
 \textbf{Computation of the fission and coherence boundaries. } Our analytical approach enables us to formulate the coherence and fission boundaries as functions of $P\!R$ using the states' existence conditions. More precisely, the coherence boundary is the curve $d\!f_{coh}(P\!R)$ separating states $APcAS$ and $AP$, while the fission boundary is the curve $d\!f_{fiss}(P\!R)$ separating states $AScI$ and $IDS$: 

\begin{small}
 \begin{equation}
   \begin{gathered}
\label{boundary_eqn}
   d\!f_{coh}(P\!R) = [(a-bN^++c-\theta)/c]^m ,\\ d\!f_{fiss}(P\!R) = [(a-bM^++c-\theta)/c]^m,
 \end{gathered}
 \end{equation}
 \end{small}
where $N^+\!=\!e^{-(T\!R-D)/\tau_i}$ and $M^+\!=\!e^{-(2T\!R-T\!D)/\tau_i}$. The existence boundaries in Figure \ref{figure7}C (including these curves) naturally emerge from the model's properties and are robust to parameter perturbations. For example, parameters $a$ and $b$ can respectively shift and stretch the two curves $d\!f_{coh}(P\!R)$ and $d\!f_{fiss}(P\!R)$. For all parameter combinations these curves have an exponential decay in $T\!R$ that generates regions of existence similar to the van Noorden diagram. 
   
 \begin{remark} \label{bistability_model}
 The model predicts the emergence of integration, segregation and bistability in plausible regions of the parameter space. Yet, it currently cannot explain (1) how perception can switch between these two interpretations for fixed $d\!f$ and $P\!R$ values (i.e. perceptual  bistability) and (2) which of the two tone streams is followed during segregation (i.e. A-A- or -B-B). This could be resolved in a competition network model, such as the one proposed by \cite{rankin2015neuromechanistic}. The selection of which rhythm is being followed by listeners at a specific moment in time would be resolved by a mutually exclusive selection of either unit: the perception is either integration if a unit responding to both tones is selected or segregation if a unit responding to every other tone is selected (see Discussion). 
 \end{remark}
 
 \begin{remark} [A note on the word bistability]
   Bistability (as used in Figure \ref{figure7}C) corresponds to states that encode both integrated and segregated rhythms simultaneously, where one unit responds to both tones and the other to one tone (say unit A responds ABAB\ldots and unit B responds -B-B\ldots). This should not be confounded with the fact that this bistable state coexists with another --- by our definition --- bistable state (unit A responds A-A-\ldots and unit B responds ABAB\ldots).
 \end{remark}
 
 \section{Computational analysis with smooth gain and inputs} \label{comp_analysis}
 In this section we extend the analytical results by running numerical simulations that use a continuous rather than Heaviside gain function and inputs, and reducing the timescale separation ratio $\tau_i/\tau$ by an order of magnitude. We restrict our study to $D\!<\!T\!D$ (the biologically realistic case), but without imposing the condition $T\!D\!+\!D\!<\!T\!R$. This allows us to make predictions at high $P\!R$s, which go beyond the analytic predictions of the previous section (see Remark \ref{extension}). In summary, we find that this smooth, non-slow-fast regime generates similar states occupying slightly perturbed regions of stability. We consider a sigmoidal gain function $S(x)=[1+\exp(-\lambda x)]^{-1}$ with fixed slope $\lambda\!=\!30$, and we consider continuous inputs adapted from (\ref{continuous_inputs}).
 
 We classify integration (INT), segregation (SEG) and bistability (BIS) based on counting the number of threshold crossings during one periodic interval $[0,2T\!R]$. Let us call $n_A$ ($n_B$) the number of threshold crossings of unit A (B) and let $n\!=\!n_A\!+\!n_B$. Based on the correspondence between states and perception described in the previous section, states for which $n\!=\!4$ ($n\!=\!2$) correspond to integration (segregation) and states for which $n\!=\!3$ correspond to bistability. We run large parallel simulations to systematically study the convergence to the $2T\!R$-periodic states under changes in $d\!f$ and $P\!R$ and detect boundaries of transitions between different perceptual interpretations. We consider a grid of $l \times l$ uniformly spaced parameters $P\!R \in [1,40]$Hz and $d\!f\in[0,1]$ ($l\!=\!98$). For each node we run long simulations from the same initial conditions and compute the number of threshold crossings after the convergence to a stable $2T\!R$-periodic state for different values of $\tau$ (Figure \ref{figure8}A, B and C). There are 5 possible regions corresponding to one of four different values of $n \in \{0,2,3,4\}$. Three of these regions (as in panel A) correspond to the three colored regions found analytically in Figure \ref{figure7}C. Figure \ref{figure8}D shows example time histories of all the states in these five regions when $\tau\!=\!0.01$ (the values of $P\!R$ and $df$ are shown in white dots in panel B). 
 
 \begin{figure}[htbp]
  \centering
  \includegraphics[width=0.75\linewidth]{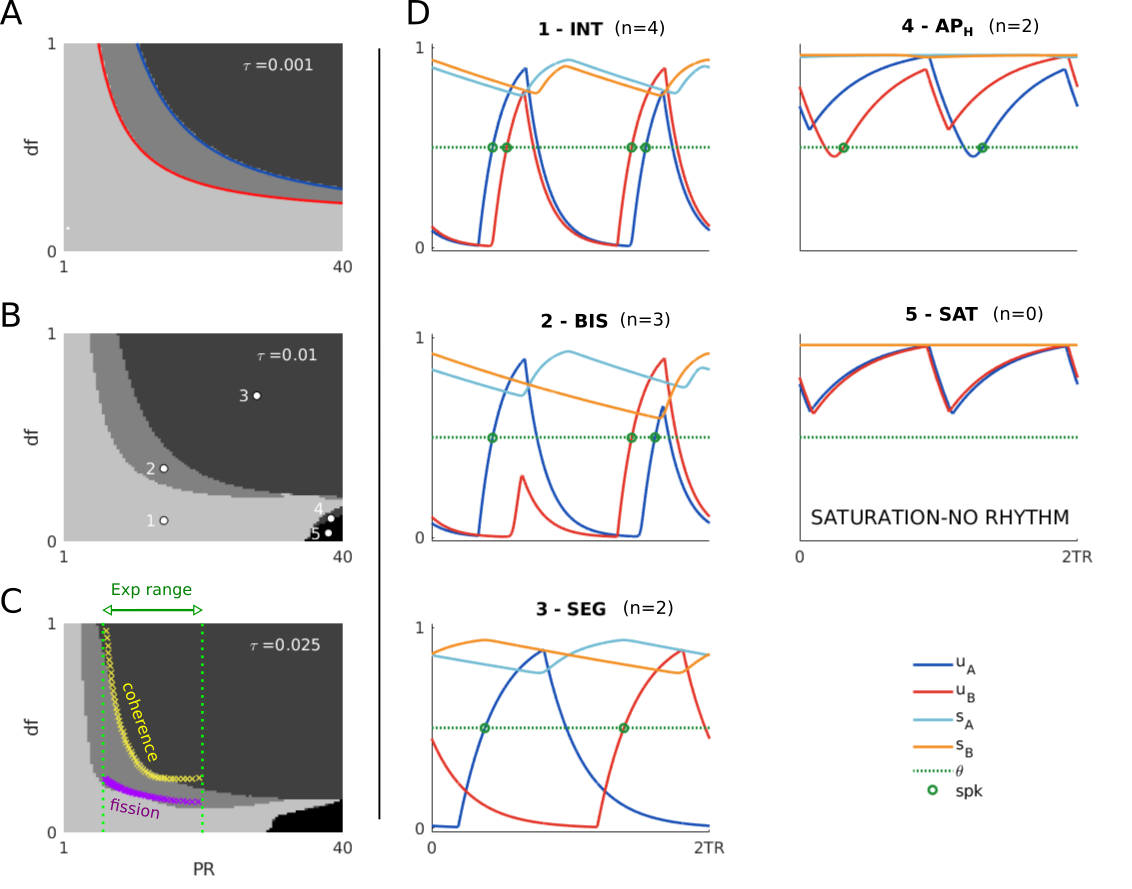}
  \caption{A-C. show the number of threshold crossings for both units $n$ in greyscale for simulated trajectories at varying $P\!R$ and $df$ (uniformly sampled in $96$ points) for different values of $\tau$ shown in top-right corner of each panel. Black corresponds to $n\!=\!0$ and the lightest gray to $n\!=\!4$. In A. the blue and red curves are the analytically predicted coherence and fission boundaries defined in equations \ref{boundary_eqn}. In C. yellow and purple crosses represent respectively the experimentally detected coherence and fission boundaries, replotted from Figure 2 in \cite{almonte2005integration}. D. Time histories for the model states in each of the five regions of panel B, with values of $P\!R$ and $df$ shown by white dots in panel B. All parameters are as in Figure \ref{fig:figure3new}. }
  \label{figure8}
\end{figure} 

 For low values of $\tau$ (panel A) the system is in the slow-fast regime. The blue and red curves show the analytically predicted coherence and fission boundaries for the Heaviside case under slow-fast regime defined in equations \ref{boundary_eqn}. These curves closely match the numerically predicted boundaries separating these regimes in the smooth system. For panel B and C $\tau$ is increased. All the existing states found in panel A persist and occupy the largest region of the parameter space, but the predicted fission and coherence boundaries perturb. We note that the selected values of $D$ and $T\!D$ in these figures lead to the condition $T\!D\!+\!D\!\geq\!T\!R$ for $P\!Rs$ greater than approximately $27$Hz, where the following two new $2T\!R$-periodic states appear: 
 
$AP_H$ - characterized by $n\!=\!2$. Both units oscillate at higher activity levels than the threshold $\sim \theta$. Since $n\!=\!2$ this state may correspond to segregation, but its perceptual relevance is difficult to assess, because it occurs in a small region of the parameter space and at high $P\!Rs$, which is outside the range tested in psychoacoustic experiments. 

$SAT$ - characterized by $n\!=\!0$. The activity of each units is higher than the threshold $\theta$ in $[0,2T\!R]$ (saturation). This state exists at (a) low $dfs$ and (b) high $P\!Rs$, greater than $30$Hz. Property (a) guarantees that inputs are strong enough to turn ON both units, while property (b) guarantees that that successive active tone intervals occur rapidly compared to the decay of the units' activities. If $\tau$ is high, although the units turn OFF between two successive tone intervals, their slow decays does not allow crossings of the threshold $\theta$. This state does not correspond to any percepts studied in the auditory streaming experiments (integration or segregation). However, $P\!R$ typically ranges between 5 and 20Hz in these experiments. The existence of this state may explain why perceivable isochronal rhythms above $\sim 30$ Hz are heard as a pure tone in the first (lowest) octave of human hearing. Indeed, when $df\!=\!0$ the model inputs represent the repetition of a single tone (B=A) with frequency $P\!R$. Our proposed framework linking percepts to neural states (see previous section) suggests that $SAT$ cannot track any rhythm simply because no unit crosses threshold. 

The coherence and fission boundaries detected from the network simulations in panel Figure \ref{figure8}C quantitatively match those from psychoacoustic experiments (yellow and purple crosses, the available data spans $P\!Rs$ in $\sim [7,20]$Hz). The model parameters chosen in the this figure (including $\tau$) have been manually tuned to match the data. Overall, we conclude that the proposed modelling framework is a good candidate for explaining the perceptual organization in the van Noorden diagram and for perceiving repeated tones (isochronal rhythms) at high frequencies as single pure tone in the lowest octave of human hearing. 

\section{Discussion}
We proposed a minimal firing rate model of ambiguous rhythm perception. Four delay differential equations represent two neural populations coupled by fast direct excitation and slow delayed inhibition that are forced by square-wave periodic inputs. Acting on different timescales, excitation and inhibition give rise to rich dynamics driven by cooperation and competition. We used analytical and computational tools to investigate periodic solutions 1:1 locked to the inputs (1:1 locked states) and their dependence on parameters influencing auditory perception.

The model incorporates neural mechanisms commonly found in auditory cortex (ACx). We hypothesised that pitch and rhythm are respectively encoded in tonotopic primary and secondary ACx \cite{musacchia2014thalamocortical}. Model units represent populations in secondary ACx - i.e. the belt or parabelt regions of auditory cortex - receiving inputs that mimic primary ACx responses \cite{hackett2014feedforward} to interleaved A and B tones \cite{fishman2004auditory}. This division of roles in ACx is supported by evidence for specific non-primary belt and parabelt regions encoding temporal features (i.e. rhythmicity) only present in sound envelope rather stimulus features (i.e. content like pitch) as in primary ACx~\cite{musacchia2014thalamocortical}. Model inputs depend on key parameters influencing psychoacoustic perception: the presentation rate ($P\!R$), the tones' pitch difference ($df$) and the tone duration ($T\!D$). The timescale separation between excitation and inhibition is consistent with AMPA and GABA synapses, respectively (widely found in cortex). The inhibition - with delay assumed fixed to $D$ - could be affected by factors including slower inhibitory activation times (vs excitatory), indirect connections and propagation times between the spatially separated A and B populations.

By posing the model in a slow-fast regime we studied 1:1 locked states for $T\!D\!+\!D\!<\!1/P\!R$, which enabled us to classify states and define a matrix representation (\emph{matrix form}). This mathematical tool helped us to formulate existence conditions and rule out impossible states, leading to a complete description of all 1:1 locked states. The condition $T\!D\!+\!D\!<\!1/P\!R$ is relevant to auditory streaming. Indeed, the factors that may play a role in generating delayed inhibition discussed above would most likely lead to short or moderate delays, for which this condition is guaranteed for the value of $P\!R$s and $T\!D$s typically considered in  experiments ($P\!R \in [5,20]$Hz and $T\!D \in [10,30]$ms; $T\!D$'s interpretation discussed below in Predictions). 


We proposed a classification of 1:1 locked states and for rhythms heard  during auditory streaming based on threshold crossing of the units' responses. More precisely, for ABAB integrated percepts both units respond to every tone and for segregated A-A- or -B-B percepts each unit responds to only one tone. Bistability corresponds to one unit responding to every tone and the other unit responding to every other tone. This interpretation of bistability can explain how both integrated and segregated rhythms may be perceived simultaneously, as reported in some behavioral studies \cite{denham2012characterising,denham2014stable}, but not the dynamic alternation between these two percepts~\cite{Pressnitzer2006,rankin2015neuromechanistic} (see the section ``Future work'' below). This classification enabled us to compare the states' existence regions to those of the corresponding percepts when varying $df$ and $P\!R$ in experiments (van Noorden diagram). A similar organization of these regions emerged naturally from the model and is robust to parameter perturbations. 

Finally, we carried out numerical analysis with a smooth gain function, smooth inputs and different levels of timescale separation to confirm the validity of the analytical approach. The simulations closely matched the analytical predictions under the slow-fast regime. Reducing the timescale separation shifts the regions of existence of the perceptually relevant states and produces a qualitatively close match the van Noorden diagram.  Numerical simulations extended this analysis to $T\!D\!+\!D\!\geq\!1/P\!R$, which led to the emergence of a high activity (saturated) state occurring at high $P\!R$s and low $df$. The case $T\!D\!+\!D\!\geq\!1/P\!R$ may lead to the existence other states not analyzed as they do not appear in the van Noorden $(P\!R,df)$-range.

\subsection{Models of neural competition}
Our proposed model addresses the formation of percepts but not switching between them, so-called auditory perceptual bistability~\cite{Pressnitzer2006,rankin2015neuromechanistic}. Future work will consider the present description acts as a front-end to a competition network (one can think of the present study as a reformulation of the pre-competition stages in \cite{rankin2015neuromechanistic}). Perceptual bistability (e.g. binocular rivalry) is the focus of many theoretical studies that feature mechanisms and dynamical states similar to those reported here. We note a key distinction here: the units are associated with tonotopic locations of the A and B tones, not with percepts as in many other models. In contrast with our study, firing rate models are widely used with fixed inputs, mutual inhibition (often assumed instantaneous), and a slow adaptation process that drives slow-fast oscillations~\cite{Laing2002,Shpiro2007,curtu2008mechanisms}. Periodic inputs associated with specific experimental paradigms have been considered in several models~\cite{Wilson2003,jayasuriya2012effects,vattikuti2016canonical,li2017attention,darki2020methods}.

\subsection{Models of auditory streaming}
The auditory streaming paradigm has been the focus of a wealth of electrophysiological and imaging studies in recent decades. However, it has received far less attention from modelers when compared with visual paradigms. Many existing models of auditory streaming have used signal-processing frameworks without a link to neural computations (recent reviews: \cite{snyder2017recent,szabo2016computational,rankin2019computational}). In contrast our model is based on a plausible network architecture with biophysically constrained and meaningful parameters. Simplifications (like the Heaviside gain function) provide the tractability to perform a detailed analysis of all states relevant to perceptual interpretations and find their existence conditions. Despite the model's apparent simplicity (4 DDEs) it produces a rich repertoire of dynamical states linked to perceptual interpretations. Our model is a departure from (purely) feature-based models because it incorporates a combination of mechanisms acting at timescales close to the interval between tones. By contrast, \cite{almonte2005integration} considers neural dynamics only on a fast time scale (less than TR). Further, \cite{rankin2015neuromechanistic} considers slow adaptation ($\tau>1$\,s) to drive perceptual alternations, assumes instantaneous inhibition and slow NMDA-excitation, a combination that precludes forward masking as reported in \cite{fishman2004auditory}. The entrainment of intrinsic oscillations to inputs was considered in \cite{wang2008oscillatory}, albeit using a highly redundant spatio-temporal array of oscillators.  Recently, a parsimonious neural oscillator framework was considered in \cite{perez2019uncoupling} but without addressing how the same percepts persist over a wide range of $P\!R$ (5-20\,Hz).

A central hypothesis for our model is that network states associated with different perceptual interpretations  are generated before entering into competition that produces perceptual bistability (as put forward in \cite{mill2013modelling} with a purely algorithmic implementation). Here network states are emergent from a combination of neural mechanisms: mutual fast, direct excitation and mutual slow acting, delayed inhibition. In contrast with \cite{rankin2015neuromechanistic} our model is sensitive to the temporal structure of the stimulus present in our stereotypical description of inputs to the model from primary auditory cortex and over the full range of stimulus presentation rates. 

\subsection{Predictions}\label{predictions}
In van Noorden's original work on auditory streaming boundaries in the $(df,P\!R)$-plane were identified: the temporal coherence boundary below which only integrated occurs and the fission boundary above which only segregated occurs. We  derived exact expressions for these behavioral boundaries   that match the van Noorden diagram. One of challenges in developing a model that reproduces the van Noorden diagram was to explain how a neural network can produce an integrated-like state at very large $df$-values and low $P\!R$s. Primary ACx shows no tonotopic overlap in this parameter range (A-location neurons exclusively respond to A tones) \cite{fishman2004auditory}. Our results show that  fast excitation can make this possible. Disrupting AMPA excitation is predicted to preclude the integrated state at large $df$-values. Furthermore, our results show that segregation relies on slow acting, delayed inhibition, which performs forward masking. Whilst the locus for this GABA-like inhibition cannot yet be specified, we predict that its disruption would promote the integrated percept.

Some model parameters (i.e.\ $T\!D$, $T\!R$, input strengths) can readily be tested in experiments by changing sound inputs. The model could predict the effect of such changes on perception. However, the role of $T\!D$ has yet to be investigated in experiments. In our model $T\!D$ better represents the duration of the primary ACx responses to tones, rather than the sound duration of each tone. This interpretation is supported by recordings of firing rates at tonotopic locations in Macaque primary ACx \cite{fishman2004auditory}. In these data $\sim 80 \%$ of the response is localized  shortly after the tone onset. This time window  is approximately constant $\sim 30$ms across different tone intervals, tone durations, $P\!R$ and $d\!f$ (unpublished results). 

Numerics for the smooth model predict a region at large $P\!R$s for which responses are saturated (no threshold crossings). These responses are consistent with rapidly repeating discrete sound events at rates above $30$Hz sounding like a low-frequency tone ($20$Hz is typically quoted as the lowest frequency for human hearing). At presentation rates above $30$Hz we predict a transition from hearing a modulated low-frequency tone to hearing two fast segregated streams as df is increased.

\subsection{Conclusion}

Our study proposed that sequences of tones are perceived as integrated or segregated through a combination of feature-based and temporal mechanisms. Here tone frequency is incorporated via  input-strengths and timing mechanisms are introduced via excitatory and inhibitory interactions at different timescales including delays. We suspect that the proposed architecture is not unique in being able to produce similar dynamic states and the van Noorden diagram. The implementation of globally excitatory inputs ($i_A(t)$ and $i_B(t)$ driving both units) rather than mutual fast-excitation is expected to produce similar results.

The resolution of competition between these states is not considered at present. Imaging studies implicate a network of brain areas (e.g. frontal and parietal) extending beyond auditory cortex for streaming \cite{cusack2005intraparietal,kanai2010human,kashino2012functional,kondo2018inhibition}, some of which are generally implicated in perceptual bistability \cite{vernet2015synchronous,wang2013brain,zaretskaya2010disrupting}.  The model could be extended to consider perceptual competition and bistability by incorporating a competition stage further downstream (in the same spirit as \cite{rankin2015neuromechanistic}). An extended framework would provide the ideal setting to explore perceptual entrainment through the periodic \cite{byrne2019auditory} or stochastic \cite{baker2019dynamic} modulation of a parameter like $df$.


\begin{backmatter}

\section*{Acknowledgements}
The authors thank Pete Ashwin and Jan Sieber for valuable feedback on earlier versions of this manuscript. 

\section*{Funding}
This work was funded by the EPSRC funding project Reference EP/R03124X/1.

\section*{Abbreviations}
ACx - auditory Cortex; $T\!D$ - tone duration; $T\!R$ - tone repetition time; $P\!R$ - presentation rate. 

\section*{Availability of data and materials}
Source code to reproduce the results presented will be made available on a public GitHub repository at the time of publication. 

\section*{Ethics approval and consent to participate}
Not applicable.

\section*{Competing interests}
The authors declare that they have no competing interests.

\section*{Consent for publication}
Not applicable.

\section*{Authors' contributions}
AF and JR were involved with the problem formulation, model design, discussion of results and writing the manuscript. AF carried out the mathematical analysis and numerical simulations. 


\bibliographystyle{vancouver} 
\bibliography{papers} 









\end{backmatter}

\section{Supplementary Material}

\subsection{Separatrices} \label{separatrices}
In this section we derive that separatrices of the degenerate fixed point $(s_1,s_2)$ of system
\begin{equation}
    \begin{array}{lcl} 
        u_A' & = & -u_A+H(a(u_B-s_2)) \\ 
        u_B' & = & -u_B+H(a(u_A-s_1))
    \end{array}
\end{equation} 
are given by
\begin{equation}
\begin{cases} 
(u_A-1)s_2/(s_1-1) & \mbox{if } u_A \leq s_1 \\ 
u_A(s_2-1)/s_1+1  & \mbox{otherwise}
\end{cases}
\nonumber
\end{equation}
We prove that these curves define the separatrices by showing the convergence of orbits from initial conditions $(u_A^0,u_B^0)$ in the top left corner in Figure \ref{fig:figure3} to $(1,1)$ (purple trajectories in Figure \ref{fig:figure3}). A similar proof holds for initial conditions in other regions of the phase-space and for convergence to $(0,0)$. Points $(u_A^0,u_B^0)$ in the top left corner belong to the set:
$$ \Omega_L = \{ (u_A,u_B) : u_A < s_1 \mbox{ and } u_B > (u_A-1)s_2/(s_1-1) \} $$
Since $\Omega_L \subset [0,u_A] \times [u_B,1]$, system \ref{fast-subsystem-simple} becomes: 
\begin{equation}
    \begin{array}{lcl} 
        u_A' & = & 1-u_A \\ 
        u_B' & = & -u_B 
    \end{array}
    \nonumber
\end{equation}
Consider an orbit starting from $(u_A^0,u_B^0) \in \Omega_L$. Since $u_A'>0$ the orbit will move towards the right until it reaches the vertical line $u_A=s_1$. The trajectory follows the same equations at all times $t$, since: 
$$ u_B(t) = u_B^0 \frac{u_A-1}{u_A^0-1} > s_2 \frac{u_A-1}{s_1-1}>s_2 $$
Where the last inequality holds because $s_1>u_A$. Thus, any trajectory ends on the top-right corner defined by: 
$$ \Omega_R = \{ (u_A,u_B) : u_A \geq s_1 \mbox{ and } u_B \geq s_2 \} $$
After the orbit reaches the curve $u_A=s_1$, $(u_A,u_B) \in \Omega_R$ it follows the system:
\begin{equation}
    \begin{array}{lcl} 
        u_A' & = & 1-u_A \\ 
        u_B' & = & 1-u_B 
    \end{array}
    \nonumber
\end{equation}
Since $u_A'>0$ and $u_B'>0$ the trajectory continues to satisfy these equations and will converge to $(1,1)$ (both turn ON simultaneously). Similar results hold for the Sigmoidal case (see Supplementary Material \ref{basin_sigmoid}).

\subsection{Basins of attraction for the fast subsystem with Sigmoid gain} \label{basin_sigmoid}
Here we numerically analyze the units' fast dynamics after replacing the Heaviside function $H$ with a Sigmoid gain function with threshold $0$ and slope $\lambda$ for parameter values for which points $(0,0)$ and $(1,1)$ coexist and compare with the results presented in Remark 3.1 for the Heaviside gain. We consider the following system: 
\begin{equation}
    \begin{array}{lcl} 
        u_A' & = & -u_A+S(a(u_B-s_2)) \\ 
        u_B' & = & -u_B+S(a(u_A-s_1))
    \end{array}
\end{equation} 
Parameter $a$ acts as a multiplicative factor on the slope $\lambda$. Figure \ref{s1} shows qualitatively similar phase portrait and the basins of attraction between the case with the Heaviside and Sigmoid gains (slope $\lambda\!=\!20$ and $a\!=\!1$). The stable equilibrium points $(0,0)$ and $(1,1)$ (black circles), the $u_A$- and $u_B$-nullclines (blue and red) and the saddle-separatrices (yellow and orange curves) discussed in Remark 3.1 for the Heaviside case persist and are slightly shift in the Sigmoid case. Furthermore, the degenerate $(s_1,s_2)$ saddle for the Heaviside case becomes a standard saddle point and slightly deviates from $(s_1,s_2)$ (red circles). The equilibia for the Sigmoidal case were detected numerically with Newton's method. Saddle separatrices (yellow and orange curves) were also found numerically via backward integration from an initial point near the saddle, in the unstable direction of the eigenvector. 

\begin{figure}[htbp]
  \centering
  \includegraphics[width=1\linewidth]{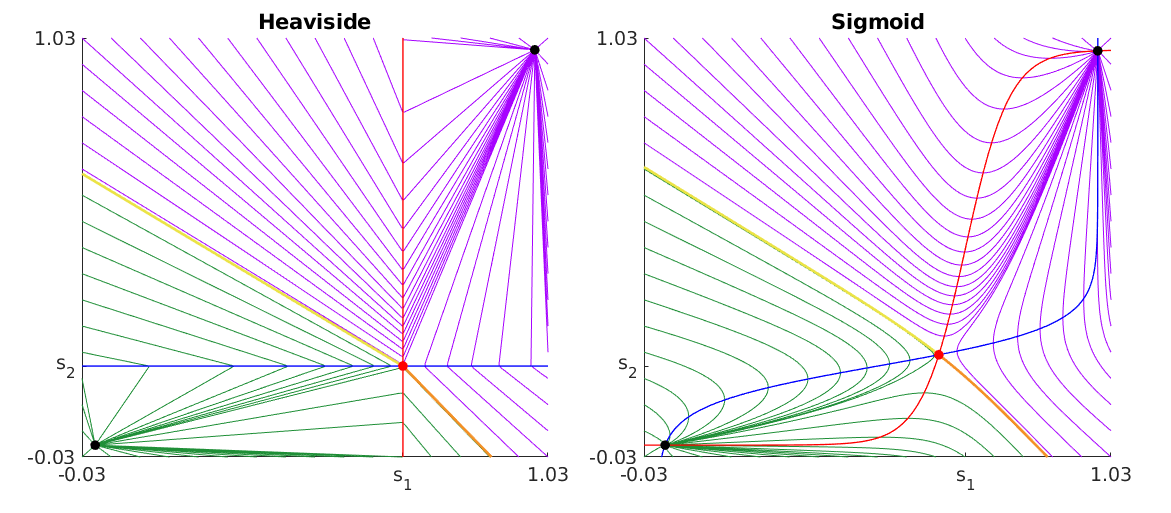}
  \caption{Phase portrait and basin of attraction for system \ref{fast-subsystem-simple} with $s_1\!=\!0.7$ and $s_2\!=\!0.4$ with gain function given by a Heaviside (left) or a Sigmoid with slope $\lambda\!=\!20$ and $a\!=\!1$ (right). The left panel is redrawn from Figure 3. Purple and green lines show orbits converge to $(1,1)$ and $(0,0)$, respectively in the Heaviside case, or to equilibria $\sim (1,1)$ and $\sim (0,0)$ in the Sidmoid case (black circles). The $u_A$- and $u_B$-nullclines are shown in blue and red, respectively. Yellow and orange lines show the saddle-separatrices of the point $(s_1,s_2)$ (red circle). Point $(s_1,s_2)$ is a degenerate saddle for the Heaviside case and a standard saddle for the Sidmoid case. }
  \label{s1}
\end{figure}

\subsection{Fast dynamics in the absence of inputs} \label{thm_ON_appendix}
\begin{theorem}[dynamics in $\mathbb{R}\!-\!I$]
For any $t \in \mathbb{R}\!-\!I$: 
\begin{enumerate}
 \item If A or B is OFF at time $t$, both units are OFF in $(t,t^*]$, where $$t^*=\min_{s \in I} \{s > t \}$$
 \item If A or B is ON at time $t$, both units are ON in $[t_*,t)$, where $$t_*=\max_{s \in I} \{s < t \}$$
\end{enumerate}
\end{theorem}
\begin{proof}
 We begin by proving 1. Due to Section \ref{fast_dynamics_I} the fast subsystem \ref{fast-subsystem-I} with no inputs ($c\!=\!d\!=\!0$) has only two possible equilibria at any time in $[t_*,t^*]$: $P\!=\!(0,0)$ and $Q\!=\!(1,1)$. At time $t_*$, if $Q$ is not an equilibium or $(u_A,u_B)$ is in the basin of attraction of $P$ the system instantaneously converges to $P$ (i.e. both units are/turn OFF). Since $P$ is an equilibium at any time in $\mathbb{R}\!-\!I$ the units remain OFF throughout $[t_*,t^*] \subset \mathbb{R}\!-\!I$, which proves the theorem. Next, assume that $Q$ is also an equilibium  and that $(u_A,u_B)$ instantaneously converges to $Q$ at time $t_*$ (i.e. both units are/turn ON at time $t^*$). By hypothesis of point 1. one unit is OFF at time $t$. By continuity there must be a turning OFF time in $\tilde{t} \in [t_*,t)$. This can occur only if $Q$ is not an equilibrium at time $\tilde{t}$, due to the dynamics of the slow variables. Thus since $P$ is an equilibrium at any time in $\mathbb{R}\!-\!I$ both units turn OFF at time $\tilde{t}$ and remain OFF in $(t,t^*] \subset [\tilde{t},t^*]$. This concludes the proof of 1. %

  We prove {2.}~by contradiction. Suppose there $\exists \bar{t} \in [t_*,t)$ when one unit is OFF. From 1. we have both units  OFF in $(\bar{t},t^*]$. This is absurd given that one unit is ON at time $t \in (\bar{t},t^*]$.
\end{proof}

\subsection{Synaptic decay lemma} \label{syn_decay_appendix}
\begin{lemma}[synaptic decay] 
If $T\!D\!+\!D<T\!R$ the delayed synaptic variables $s_A(t\!-\!D)$ and $s_B(t\!-\!D)$ are monotonically decreasing in $[\alpha_k^A,\alpha_k^A\!+\!D]$ or $[\alpha_k^B,\alpha_k^B\!+\!D]$, $\forall k \in \mathbb{N}$
\end{lemma}
\begin{proof}
 This lemma is illustrated in Figure \ref{cartoon2}A. From Remark \ref{remark1} the synaptic variable $s_A$ ($s_B$) is monotonically decreasing except for when A (B) turns ON. Due to Theorem \ref{thm:uON} such an event cannot occur at any time $t \in \mathbb{R}\!-\!I$. Thus, it is sufficient to prove that $t\!-\!D \in \mathbb{R}\!-\!I$. Without loss of generality (WLOG) consider $L\!=\![2k T\!R,2k T\!R\!+\!D]$ and $t \in L$, which implies:
$ 2k T\!R\!-\!D \leq t\!-\!D \leq 2k T\!R $. To complete the proof, the condition $T\!D\!+\!D\!<\!T\!R$ implies:
$$ 2k T\!R\!-\!D \geq (2(k\!-\!1)\!+\!1)T\!R\!+\!T\!D = \beta_{k-1}^B \implies \beta_{k-1}^B \leq t\!-\!D \leq \alpha_k^A $$ 
The last inequalities imply $t\!-\!D \in \mathbb{R}-I$ and conclude the proof.
\end{proof}

\subsection{No saturation lemma} \label{no_saturation_appendix}
\begin{lemma}[no saturated states] 
If $T\!D\!+\!D\!<\!T\!R$ both units are OFF in the intervals $(\alpha_k^A\!+\!T\!D+\!D,\alpha_{k}^B]$ and $(\alpha_k^B\!+\!T\!D+\!D,\alpha_{k+1}^A]$, $\forall k \in \mathbb{N}$.
\vspace{-0.5mm}
\end{lemma}
\begin{proof}
 We prove the theorem for the interval $(\alpha_k^A\!+\!TD+\!D,\alpha_{k}^B]$ (extension to  other intervals is analogous). By contradiction suppose $\exists \bar{t}$ in this interval when either unit, say A, is ON. Since $T\!D\!+\!D<T\!R$ we have  $\bar{t} \in \mathbb{R}\!-\!I$. Theorem \ref{thm:uON} implies both units are ON in $[t_*,\bar{t})$, where $t_*\!=\!\alpha_k^A\!+\!T\!D$. Thus, at time $p_*\!=\!t_*\!+\!D \in [t_*,\bar{t})$ the delayed synaptic variables tends to 1 following the fast system \ref{fast-model}. From this and condition \ref{U2} we have that $ a\!-\!bs_A(p_*\!-\!D) \sim a\!-\!b<\theta$ and $a\!-\!bs_B(p_*\!-\!D) \sim a\!-\!b<\theta$. Hence $(0,0)$ must be the only stable equilibrium at time $p^*$, which is absurd since BOTH units are ON at this time. 
\end{proof}

\subsection{Single OFF to ON transition Lemma} \label{appendix_one_transition}
Here we prove the following Lemma, that derives from Lemma \ref{lem:syn_decay} and Lemma \ref{lem:no_saturation}. 
\begin{lemma}[single OFF to ON transition]
Let $D\!>\!TD$ and $T\!D\!+\!D\!<\!T\!R$ and consider an active tone interval $R\!=\![\alpha,\beta] \in \Phi$. Let A (B) be ON at a time $\bar{t} \in R$, then \newline
(1) A (B) is ON $\forall t \geq \bar{t}$, $t \in R$ \newline
(2) $\exists !\, t_A^* \, (t_B^*) \in R$ when A (B) turns ON  \newline
(3) $s_A(t\!-\!D)$ ($s_B(t\!-\!D)$) is decreasing for $t \in [\alpha,t_A^*\!+\!D]$ ($t \in [\alpha,t_B^*\!+\!D]$)
\end{lemma}
\begin{proof}
We prove this Lemma for the A unit and for the interval $R=I_k^A$, i.e. we assume that $\alpha\!=\!\alpha_k^A$ and $\beta\!=\!\beta_k^A$, where $R\!=\![\alpha,\beta]$. The extension to the other intervals and for the B unit is analogous. Let us call $\gamma\!=\!\gamma_k^A$. Since $T\!D\!+\!D\!<\!T\!R$ we can apply Lemma \ref{lem:syn_decay}, which implies $s_A(t\!-\!D)$ and $s_B(t\!-\!D)$ to be monotonically decreasing in $[\alpha,\gamma]$. Moreover, since $D\!>\!T\!D$ we have that $R \subseteq [\alpha,\gamma]$. Thus the delayed synaptic variables are monotonically decreasing in $R$. \\
We now prove (1). On the fast time scale $(u_A,u_B)$ follow the fast subsystem \ref{fast-subsystem-I} at time $\bar{t}$ and may converge to one of the four equilibria described in \ref{equilibria-fast}. However, since A is ON at time $\bar{t}$ trajectories converge to either  $(1,0)$ or $(1,1)$. 

In the first case (convergence to $(1,0)$) we have 
$$ c \geq b s_B(\bar{t}\!-\!D)+\theta. $$
Due to the decay of the synaptic variables, the same inequality holds $\forall t \geq \bar{t} \in R$. This condition is guaranteed only for the two equilibrium points $(1,0)$ and $(1,1)$. Therefore any orbit either remains fixed at $(1,0)$ or undergo a transition to $(1,1)$.

In the second case (convergence to $(1,1)$) we have 
$$ a\!+\!c \geq b s_B(\bar{t}\!-\!D)\!+\!\theta $$
$$ a\!+\!d \geq b s_A(\bar{t}\!-\!D)\!+\!\theta. $$
Due to the decay of the synaptic variables these inequalities hold $\forall t \geq \bar{t} \in R$. Therefore $(1,1)$ remains an equilibrium at such times. In both cases (convergence to $(1,0)$ or $(1,1)$) the A unit is ON $\forall t \geq \bar{t} \in R$, proving (1).

We now prove (2). Lemma \ref{lem:no_saturation} implies that A is OFF for some $t\!<\!\alpha$. Suppose that A is ON at time $\bar{t}$. For continuity, there $\exists t_A^* \in R$ when the A unit undergoes an OFF to ON transition, thus proving the first claim. The uniqueness of $t_A^*$ follows by contradiction. Suppose the existence of two distinct OFF to ON transition times $p^*$, $q^* \in R$ for the A unit. We can assume that $p^*\!<\!q^*$. Since A turns ON at time $q^*$, there $\exists r^* \in R$ with  $p^*\!<\!r^*\!<\!q^*$ such that A is OFF at time $r^*$. The fact that A turns ON at time $p^*$ and is OFF at time $r^*\!>\!p^*$ contradicts (1). 

Lastly we prove (3) for $s_A(t\!-\!D)$. Since $[\alpha,t_A^*\!+\!D]$ is the union of closed intervals $R$ and $[\beta,t_A^*\!+\!D]$, proving that $s_A(t\!-\!D)$ is monotonically decreasing in each of these subintervals would suffice. We previously proved that $s_A(t\!-\!D)$ is monotonically decreasing in $R$. Thus, we are left to prove that the same property holds in $[\beta,t_A^*\!+\!D]$. 

Due to Remark \ref{remark1} we have to prove that A cannot turn ON at any time in the interval $[\beta\!-\!D,t_A^*]$. Due to point (2) of the current lemma the turning ON time $t_A^*$ for A exists and is unique in the interval $R$. Therefore A does not turn ON in $[\alpha,t_A^*]$. Moreover since $D\!<\!T\!R$ we have $\beta\!-\!D\!=\!\beta_k^A\!-\!D\!>\!\alpha_{k-1}^B$, which leads to $[\beta\!-\!D,\alpha] \subset \mathbb{R}\!-\!I$. From Theorem \ref{thm:uON} we have that A cannot turn ON in $[\beta\!-\!D,\alpha]$. Thus we have that A cannot turn ON in $[\beta\!-\!D,\alpha] \cup [\alpha,t_A^*] \!=\! [\beta\!-\!D,t_A^*]$, which yields the desired result.  
\end{proof}

\subsection{Classification of CONNECT states} \label{connect_appendix}
To define a classification and matrix form for CONNECT states we consider the following cases:
\begin{itemize}
	\item \textbf{A(B) turns ON at time $\alpha$ and B(A) turns ON at time $t^*$, $\exists t^* \in (\alpha,\beta]$}. These two conditions are equivalent to $(1,0)$ ($(0,1)$) and $(1,1)$ being equilibria for the the fast subsystem at time $\alpha$ and $\beta$, respectively. We  note that the validity of the previous statement is due to $(1,0)$ being in the basin of attraction of $(1,1)$ for any set of parameters (as shown in Figure \ref{fig:figure3}). There are two conditions for which this occurs:
\begin{enumerate} [label=(\subscript{C}{{\arabic*}})]
 \item $f(\underline{s}_B)\!\geq\!\theta$, $a\!+\!g(\underline{s}_A)\!<\!\theta$ and $a\!+\!g(\bar{s}_A)\!\geq\!\theta$
 \item $g(\underline{s}_A)\!\geq\!\theta$, $a\!+\!f(\underline{s}_B)\!<\!\theta$ and $a\!+\!f(\bar{s}_B)\!\geq\!\theta$
\end{enumerate}
$C_1$ ($C_2$) describes the case where the B (A) units turn ON within the interval $R$ and the A (B) unit is ON at time $\alpha$. 
\item \textbf{A(B) is OFF at time $\beta$ and B(A) turns ON at time $t^*$, $\exists t^* \in (\alpha,\beta]$}. These two events correspond to $(0,0)$ and $(0,1)$ ($(1,0)$) being equilibria for the the fast subsystem at time $\alpha$ and $\beta$, respectively. The following conditions lead to the following cases:
\begin{enumerate} [label=(\subscript{C}{{\arabic*}})]
 \setcounter{enumi}{2}
 \item $g(\underline{s}_A)\!<\!\theta$, $g(\bar{s}_A)\!\geq\!\theta$ and $a\!+\!f(\bar{s}_B)\!<\!\theta$
 \item $f(\underline{s}_B)\!<\!\theta$, $f(\bar{s}_B)\!\geq\!\theta$ and $a\!+\!g(\bar{s}_A)\!<\!\theta$
\end{enumerate}
$C_3$ ($C_4$) describes the case where the A (B) units is OFF at time $\beta$ and the B (A) turns ON within $R$. 

\item \textbf{$\exists t^*, s^* \in (\alpha,\beta]$ times when the A and B unit turns ON}. The conditions leading to this case are different depending on if A turns ON before or after B, that is: 
\begin{enumerate}
 \item A turns ON before B - if $t^*\!\leq\!s^*$, $f(\underline{s}_B)\!<\!\theta$, $f(\bar{s}_B)\!\geq\!\theta$ and $a\!+\!g(\bar{s}_B)\!\geq\!\theta$
 \item B turns ON before A - if $t^* \!>\! s^*$, $g(\underline{s}_A)\!<\!\theta$, $g(\bar{s}_A)\!\geq\!\theta$ and $a\!+\!f(\bar{s}_A)\!\geq\!\theta$
\end{enumerate} 
In both cases, $(0,0)$ and $(1,1)$ are equilibria for the fast subsystem respectively for $t\!<\!\min\{t^*,s^*\}$ and $t\!\geq\!\max\{t^*,s^*\}$. In the first and second cases respectively $(1,0)$ and $(0,1)$ are equilibria for $t \in [t^*,s^*)$ ($t \in [s^*,t^*)$). For simplicity we decide not to distinguish between the cases 1. and 2. and define ($C_5$) as referring to either condition.
\end{itemize}

\subsection{CONNECT matrix form} \label{appendix_CONN_matr} 
\begin{theorem} \label{appendix_CONNECT_R}
Set $R \in \Phi$. There is an injective map: 
\begin{align*}
\varphi^R\colon C_R & \rightarrow B(2,3)\\
s &\mapsto W = 
\begin{bmatrix}
    x_A & y_A & z_A \\
    x_B & y_B & z_B
\end{bmatrix} 
\end{align*}
With entries defined by:
\begin{align}
\begin{split}
x_A\!&=\!H(f(\underline{s}_B)), \: y_A\!=\!H(ax_B\!+\!f(\underline{s}_B)), \: z_A\!=\!H(a\!+\!f(\bar{s}_B)) \\
x_B\!&=\!H(g(\underline{s}_A)), \: y_B\!=\!H(ax_A\!+\!g(\underline{s}_A)), \: z_B\!=\!H(a\!+\!g(\bar{s}_A)) 
\end{split}
\label{appendix_matricial_representation_CONN}
\end{align}
And we have:
$$ Im(\varphi^R)\!=\!\Gamma\!\eqdef\! \{ 
W : x_A \!\leq\! y_A \!\leq\! z_A, 
x_B \!\leq\! y_B \!\leq\! z_B, 
x_A\!=\!x_B\!=\!0 \Rightarrow y_A\!=\!y_B\!=\!0, 
y_A\!<\!z_A \textup{ or } y_B\!<\!z_B
\} $$
\end{theorem}
\begin{proof}
We first prove that the entries of any matrix $W\!=\!\varphi^R(s)$ satisfy the three conditions in $\Gamma$. It is easy to show that, since $a\! \geq \!0$, $f(\underline{s_A})\! \leq f(\bar{s_A})$ and $f(\underline{s_B})\! \leq f(\bar{s_B})$ the first conditions, i.e. $x_A \leq y_A \leq z_A$ and $x_B \leq y_B \leq z_B$, hold. The condition $x_A\!=\!x_B\!=\!0 \Rightarrow y_A\!=\!y_B\!=\!0$ simply follows from identities \ref{appendix_matricial_representation_CONN}. One can  check that any CONNECT state defined by conditions $C_i$, $\forall i=1,..,5$ satisfies  $y_A\!<\!z_A \mbox{ or 	} y_B\!<\!z_B$. 

Using $x_A \leq y_A \leq z_A$ and $x_B \leq y_B \leq z_B$ one can easily see that each CONNECT state satisfying one of conditions $C_{1-4}$ has a corresponding image $\varphi^R(s)$ shown below. The case $C_5$ is treated separately, since both A and B turn ON at times $t^*$ and $s^*$, respectively. 
\begin{itemize}
 \item If $t^*\!\leq\!s^*$ it is clear that $f(s_B(t^*))\!=\!\theta$ and $g(s_A(t^*))\!<\!\theta$. Thus, since $s_A$ and $g$ are respectively decreasing and increasing functions in $R$, we must have $g(\underline{s}_A)=g(s_A(0))\!<\!g(s_A(t^*))\!<\theta$. In addition $a\!+\!f(\bar{s}_B)\!\geq\!f(\bar{s}_B)\!\geq\!\theta$ and $a\!+\!g(\bar{s}_B)\!\geq\!\theta$. 
 \item If $t^* \!>\! s^*$ similar considerations lead to $f(\underline{s}_B)\!<\theta$. In addition $a\!+\!g(\bar{s}_A)\!\geq\!g(\bar{s}_A)\!\geq\!\theta$ and $a\!+\!f(\bar{s}_A)\!\geq\!\theta$. 
\end{itemize}
In both cases we thus have $x_A\!=\!x_B\!=\!0$ (which leads to $y_A\!=\!y_B\!=\!0$) and $z_A\!=\!z_B\!=\!1$. 
\begin{equation}
(C_1) \; 
\begin{bmatrix}
    1 & 1 & 1 \\
    0 & 0 & 1
\end{bmatrix}
\quad 
(C_2) \; 
\begin{bmatrix}
    0 & 0 & 1 \\
    1 & 1 & 1
\end{bmatrix}
\quad 
(C_3) \; 
\begin{bmatrix}
    0 & 0 & 0 \\
    0 & 0 & 1
\end{bmatrix}
\quad 
(C_4) \; 
\begin{bmatrix}
    0 & 0 & 1 \\
    0 & 0 & 0
\end{bmatrix}
\quad 
(C_5) \; 
\begin{bmatrix}
    0 & 0 & 1 \\
    0 & 0 & 1
\end{bmatrix}
\quad 
\nonumber
\end{equation}
Since any CONNECT state has a distinct image, $\varphi^R$ is well defined and injective. It is trivial to prove that $Im(\varphi^R) \subseteq \Gamma$. However, since $|\Gamma|\!=\!6$, we must have $Im(\varphi^R) \!=\! \Gamma$. 
\end{proof}

\subsection{Proof of the LONG states theorem} \label{appendix_LONG_thm}

\begin{lemma} [LONG states] 
A state is LONG if and only if $\exists R=[\alpha,\beta] \in \Phi$ such that
\vspace{1mm}
\begin{enumerate}
 \item A and B turn ON at times $t^*_A$ and $t^*_B \in R$, respectively. 
 \item $a\!-\!b s_A(\beta\!-\!D) \geq \theta$ and $a\!-\!b s_B(\beta\!-\!D) \geq \theta$.
\end{enumerate}
\vspace{1mm}
Moreover, both units are ON in $[\beta,t^*\!+\!D]$, turn OFF at time $t^*\!+\!D$, and are OFF in $(t^*\!+\!D,t_{up}]$, where 
$$ t^*=\min\{ t^*_A,t^*_B \} \quad \mbox{and} \quad t_{up}=\min_{s \in I}\{ s > t \}. $$
\end{lemma}

\begin{proof}
($\Rightarrow$) Consider a LONG state. By definition one unit is ON at time $t$, for some $t \!\geq\! t_0 \in \mathbb{R}\!-\!I$. Thus $t \in T \cup S = (\beta_k^A,\alpha_k^B) \cup (\beta_k^B,\alpha_{k+1}^A)$, for some $k \in \mathbb{N}$ (where $T \cap S \!=\! \emptyset$). WLOG suppose $t \in T$. We will prove the claim for $R=[\alpha_k^A,\beta_k^A]$. Theorem \ref{thm:uON} implies both units being ON in $[\beta,t)$, where $\beta\!=\!\beta_k^A$. The application of Lemma \ref{lem:one_transition} at time $\bar{t}\!=\!\beta \in R$ implies the existence of (unique) OFF to ON transition times $t_A^*, t_B^* \in R$ for the A and B units, respectively, which proves point 1. Since both units are ON in $[\beta,t)$ for $t\!>\!\beta$, they are ON at time $\beta\!+\!h$, for $h\!>\!0$ arbitrarily small. At this time the inputs are OFF ($\beta\!+\!h \in \mathbb{R}\!-\!I$) and the delayed synaptic variables act on the slow time scale (due to point 3. in Lemma \ref{lem:one_transition}). Therefore $(1,1)$ must be an equilibrium point for $(u_A,u_B)$ in the fast subsystem with no inputs at time $\beta\!+\!h$, and must satisfy the condition given in Section \ref{fast_dynamics_I}: $a\!-\!b s_A(\beta\!-\!D\!+\!h) \geq \theta$ and $a\!-\!b s_B(\beta\!-\!D\!+\!h) \geq \theta$. Taking the limit as $h \rightarrow 0$ concludes the first part of the proof. 

($\Leftarrow$) Point 1 of Lemma \ref{lem:one_transition} guarantees both unit being ON at time $t\!=\!\beta$. Since $a\!-\!bs_B(\beta)\! \!\geq \theta\! $ and $a\!-\!bs_A(\beta) \!\geq\! \theta$ we have that $(1,1)$ is a stable fixed point for the fast subsystem \ref{fast-subsystem-I}. Moreover, from point 3 of Lemma \ref{lem:one_transition} $s_A(t\!-\!D)$ and $s_B(t\!-\!D)$ are monotonically decreasing for $t \in [\beta,t^*+D]$, where $t^*=\min\{ t^*_A,t^*_B \}$. Thus, on the fast time scale, $a\!-\!bs_B(t\!-\!D)\!\geq\!\theta $ and $a\!-\!bs_A(t\!-\!D)\!\geq\!\theta $, which implies that $(1,1)$ is a stable equilibrium for the system in $[\beta,t^*\!+\!D]$. Since $TD\!<\!D$, $t^*\!+\!D \!>\! \beta$. Therefore, there $\exists t \in [\beta,t^*\!+\!D] \in \mathbb{R}\!-\!I$ where both units are ON, ending this part of the proof. 

Lastly we prove the remaining claims of the Lemma. We already proved that both units are ON in $[\beta,t^*\!+\!D]$ in ($\Leftarrow$) above. To prove the remaining claims we assume $t^*\!=\!t^*_A$ (a similar proof holds if $t^*\!=\!t^*_B$). At time $t\!=\!t^*\!+\!D$, $s_A(t\!-\!D)$ jumps up to 1. Since $a\!-\!bs_A(t\!-\!D)\!=\!a\!-\!b\!<\!\theta$ due to condition \ref{U2}, $(0,0)$ is the only equilibrium at time $t$. Therefore the B units instantaneously turns OFF at time $t$. For Theorem \ref{thm:uON}, also the A unit turns OFF instantantaneously after a small delay $\delta \sim \tau$. Both units are OFF in $[t^*\!+\!D\!+\!\delta,t_{up}]$. By taking the limit $\tau \rightarrow 0$ we thus have that A and B are OFF in $(t^*\!+\!D,t_{up}]$
\end{proof}

\subsection{Proof of the remaining claims of Theorem \ref{MAIN_2TR}} \label{appendix_thm_MAIN_R}
We restate Theorem \ref{MAIN_2TR} for clarity. 
\begin{theorem} 
There is an injective map: 
\begin{align*} 
\rho\colon SM & \rightarrow B(2,4)\\
\psi &\mapsto V =
\begin{bmatrix} [c|c]
    V_1 & V_2 
\end{bmatrix}
=
\begin{bmatrix} [cc|cc]
    x_A^1 & y_A^1 & x_A^2 & y_A^2 \\
    x_B^1 & y_B^1 & x_B^2 & y_B^2
\end{bmatrix}
\end{align*}
Where, for $i\!=\!1,2$, $V_i$ are the matrix forms of $\psi$ during the interval $I_i$ defined in \ref{matricial_representation_MAIN}, and: 
\begin{equation}
s_B^{i \pm}\!=\!N^{\pm}y_B^j\!+\!M^{\pm}(1-y_B^j)y_B^i, 
\quad \mbox{and} \quad 
s_A^{i \pm}\!=\!N^{\pm}y_A^j\!+\!M^{\pm}(1-y_A^j)y_A^i, 
\quad \forall i,j\!=\!1,2, i\!\neq\!j
\label{syn_quant_appendix}
\end{equation} 
In addition, 
$$Im(\rho)=\Omega \eqdef \{V =
\begin{bmatrix} [c|c]
    V_1 & V_2 
\end{bmatrix} :
V_1 \in Im(\rho^{I_1}), V_2 \in Im(\rho^{I_1}) \mbox{ satisfying 1-4 below}  \}$$
\begin{enumerate} 
 \item $y_A^1=y_B^2=1 \Rightarrow x_A^1=x_B^2$ and $y_A^2=y_B^1=1 \Rightarrow x_A^2=x_B^1$
 \item $y_B^1=y_B^2 \Rightarrow x_A^1 \geq x_A^2$ and $y_A^1=y_A^2 \Rightarrow x_B^2 \geq x_B^1$
 \item $y_A^2=1 \Rightarrow x_B^1 \leq r$ and $y_B^1=1 \Rightarrow x_A^2 \leq r$, for any entry $r$ in $V$
 \item $y_A^2=y_B^2$, $y_A^1=y_B^1 \Rightarrow x_A^1 \geq x_B^1$ and $x_B^2 \geq x_A^2$
\end{enumerate}
\end{theorem}
\begin{proof}
Here we prove equations \ref{syn_values_MAIN_SHORT} and conditions 1-4. The remaining claims of the theorem are proven in the main text. 
From Theorem \ref{thm:matr_MAIN} it is clear that the map $\rho\!=\!\rho(\psi)$ is well defined and injective. We now prove \ref{syn_values_MAIN_SHORT} for $i\!=\!2$, $j\!=\!1$ and $s_B$, since all other cases are similar. That is: 
$$ s_B^{2 \pm}\!=\!N^{\pm}y_B^1\!+\!M^{\pm}(1-y_B^1)y_B^2 $$
Since $y_B^1$ and $y_B^2$ are binary, we have three cases to consider:

\begin{itemize}
 \item Case $y_B^1\!=\!1$. From Remark \ref{visualization_matr}, $y_B^1\!=\!1$ implies the B unit to be ON at time $T\!D$. Since $\phi$ is SHORT the B unit turns OFF at time $T\!D$, and due to Remark \ref{thm:uON} it remains OFF $\forall t \in (T\!D,T\!R]$. Thus the delayed synaptic variable $s_B(t\!-\!D)$ is equal $\sim 1$ at time $T\!D\!+\!D$ and decays (slowly) in the interval $I_2$, evolving according to: 
\begin{equation}
s_B(t\!-\!D)\!=\!e^{-(t-T\!D\!-\!D)/\tau_i}, \quad \forall t \in I_2
\nonumber
\end{equation}
Thus evaluating this function at times $T\!R \in I_2$ and $\: T\!R\!+\!T\!D \in I_2$ yields: 
$$ s_B^{2-}\!=\!s_B(T\!R\!-\!D)\!=\!N^{-} \quad \mbox{and} \quad s_B^{2+}\!=\!s_B(T\!R\!+T\!D\!-\!D)\!=\!N^{+}. $$
\item Case $y_B^1\!=\!0$ and $y_B^2\!=\!1$. With a proof similar to the case above, the second condition ($y_B^2\!=\!1$) implies the B unit being ON at time $T\!R\!+\!T\!D$, and being OFF $\forall t \in (TR\!+\!T\!D,2T\!R]$. The first condition ($y_B^1\!=\!0$) implies B being OFF at time $T\!D\!$, and therefore $\forall t \in [0,T\!D]$, due to Lemma \ref{lem:ON-OFF-main}. Thus, since $\psi$ is $2T\!R$-periodic, B must be OFF in $[2T\!R,2T\!R\!+\!T\!D]$. Moreover, since $\phi$ is SHORT, B is OFF in $(T\!D,T\!R] \cup (T\!R\!+\!T\!D,2T\!R] \subset\mathbb{R}\!-\!I$. In particular, since $\psi$ is $2T\!R$-periodic, B must be OFF also in $(2T\!R\!+\!T\!D,3T\!R] \subset\mathbb{R}\!-\!I$. Overall we have that B is ON at time $T\!R\!+\!TD$ and OFF during $(T\!R\!+\!TD,3T\!R]$. Thus the delayed synaptic variable $s_B(t\!-\!D)$ is equal $\sim 1$ at time $T\!D\!+\!D$ and decays (slowly) in the interval $T\!=\!(T\!R\!+\!T\!D\!+\!D,3T\!R\!+\!D]$, evolving according to: 
\begin{equation}
s_B(t\!-\!D)\!=\ e^{-(t-T\!R\!-T\!D\!-\!D)/\tau_i}, \quad \forall t \in T
\nonumber
\end{equation}
Since $T\!D\!+\!D\!<\!T\!R$ and $T\!D\!<\!D$ we have $3T\!R \in T$ and $\: 3T\!R\!+\!T\!D \in T$. Evaluating $s_B(t\!-\!D)$ at these times leads to $s_B(3T\!R\!-\!D)\!=\!M^{-}$ and $s_B(3T\!R\!+T\!D\!-\!D)\!=\!M^{+}$. Therefore the $2T\!R$ periodicity of $\psi$ implies:
$$ s_B^{2-}\!=\!s_B(T\!R\!-\!D)\!=\!M^{-} \quad \mbox{and} \quad s_B^{2+}\!=\!s_B(T\!R\!+T\!D\!-\!D)\!=\!M^{+}.$$
\item Case $y_B^1\!=\!0$ and $y_B^2\!=\!0$. These conditions imply B being OFF during both $[0,T\!D]$ and $[T\!R,T\!R\!+\!T\!D]$. Moreover it must be OFF also in $[T\!D,T\!R] \cup [T\!R\!+\!T\!D,2T\!R] \subset\mathbb{R}\!-\!I$ since $\phi$ is SHORT. Overall, the B unit is thus OFF $\forall t \in [0,3TR]$. This means that the delayed synaptic variables $(s_A,s_B)$ follow the slow subsystem, which have only one possible periodic solution: the fixed point $(0,0)$. This leads to $s_B\!=\!0$. 
\end{itemize}

We now show that the entries of $V\!=\! \rho(\psi)$ satisfy conditions 1-4, which proves that $Im(\rho) \subseteq \Omega$. We only prove one of the two statements for points 1,2 and 3. The proof of second statements is analogous. We recall that, given the definition of function $f$ and $g$ given in \ref{f_g}, the 1st and 3rd columns of $V$ are defined by:
$$ x_A^1=H(c-bs_B^{1-}), \quad x_B^1=H(d-bs_A^{1-}), \quad x_A^2=H(d-bs_B^{2-}), \quad x_B^2=H(c-bs_A^{2-}) $$

\begin{enumerate} 
 \item Given the $x_A^1$ and $x_B^2$ equations above, we need to prove $s_B^{1-}=s_A^{2-}$. Assuming $y_A^1\!=\!y_B^2\!=\!1$, from equations \ref{syn_quant_appendix} we have:
 $$ s_B^{1-} = N^{-}y_B^2\!+\!M^{-}(1-y_B^2)y_B^1 = N^{-} = N^{-}y_A^1\!+\!M^{-}(1-y_A^1)y_A^2 = s_A^{2-} $$
 \item If $y_B^1\!=\!y_B^2$ simple substitutions in \ref{syn_quant_appendix} lead to $s_B^{1-}\!=\!s_B^{2-}$. Since $c\! \geq\!d$ we have:
 $$x_A^1=H(c-bs_B^{1-})=H(c-bs_B^{2-}) \geq H(d-bs_B^{2-})=x_A^2$$
 \item Substituting $y_A^2\!=\!1$ in the formula for $s_A^{1-}$ in \ref{syn_quant_appendix} implies $s_A^{1-}\!=\!N^-$ and $s_B^{i-}, s_A^{i-}\!\leq\!N^-\!=\!s_A^{1-}$, $\forall i\!=\!1,2$. The latter inequalities and $c\! \geq\!d$ imply
 $$ x_B^1 \leq x_B^i, \quad x_B^1 \leq x_A^i, \quad \forall i\!=\!1,2, $$ 
 since $V_1$ and $V_2$ are matrix forms of $\psi$ in $I_1$ and $I_2$, respectively, their entries satisfy the first line of system \ref{yA-x_A}, which imply $x_B^i\! \leq \!y_B^i$ and $x_A^i\! \leq \!y_A^i$, $\forall i\!=\!1,2$. This proves that $x_B^1 \leq y_B^i$ and $x_B^1 \leq y_A^i$, $\forall i\!=\!1,2$, and concludes the proof. 
 \item If $y_A^2\!=\!y_B^2$ and $y_A^1\!=\!y_B^1$, simple substitutions in \ref{syn_quant_appendix} lead to $s_B^{1-}\!=\!s_A^{1-}$ and $s_B^{2-}\!=\!s_A^{2-}$. These equalities, together with $c\! \geq\!d$ imply:
 $$ x_A^1\!=\!H(c\!-\!bs_B^{1-}) \!\geq\! H(d\!-\!bs_A^{1-})\!=\!x_B^1 \quad \mbox{and} \quad x_B^2\!=\!H(c\!-\!bs_A^{2-}) \!\geq\! H(d\!-\!bs_B^{2-})\!=\!x_A^2 $$
\end{enumerate}
End
\end{proof}

\subsection{Multistability} \label{multistability_appendix}
\begin{theorem} [Multistability] The state $I$ may coexist
  with $S\!B$ or $S\!D$. Any other pair of $2T\!R$-periodic SHORT MAIN
  states cannot coexist.
\end{theorem}
\begin{proof}
The inequalities shown in black in Table \ref{tab:multistability} report all the existence conditions for MAIN SHORT states analyzed in the main manuscript and summarized in Table \ref{tab:MAIN_table}. Using the properties $a \geq 0$, $N^+ \!\geq\! M^+$ and $c \geq d$ on the quantities $C_i^\pm$ defined in \ref{eq1} one can easily show that 
\begin{equation}
 1)\; C_2^-\!\geq\!C_8^- \quad 2)\; C_3^+\!\geq\!C_6^+ \quad 3)\; C_3^-\!\geq\!C_7^-  \quad 4)\; C_3^-\!\geq\!C_7^-  \quad 5)\; C_5^-\!\geq\!C_7^-,
 \label{multistability_cond}
\end{equation}
which imply the inequalities reported in blue in Table \ref{tab:multistability}. 

\begin{table}[h] \centering
\begin{center}
\renewcommand{\arraystretch}{1.3}
\begin{tabular}{ |c|c|c|c|c|c|c|c|c|c|} 
 \hline
 $C$ & $S$ & $SB$ & $SD$ & $AP$ & $AS$ & $ASD$ & $I$ & $ID$ & $IB$ \\ \hline
 $1$ & $C_1\!<\!\theta$ & & & & & & $C_1\!\geq\!\theta$ & & \\ \hline
 $2$ & $C_2^+\!<\!\theta$ & \textcolor{blue}{$C_2^-\!\geq\!\theta$} & $C_2^-\!\geq\!\theta$ & $C_2^+\!<\!\theta$ & \textcolor{blue}{$C_2^-\!\geq\!\theta$} & $C_2^-\!\geq\!\theta$ & &  & \\ \hline
 $3$ & $C_3^+\!<\!\theta$ & $C_3^+\!<\!\theta$ & $C_3^+\!<\!\theta$ & $C_3^-\!\geq\!\theta$ & $C_3^-\!\geq\!\theta$ & $C_3^-\!\geq\!\theta$ & \textcolor{blue}{$C_3^-\!<\!\theta$} & $C_3^-\!\geq\!\theta$ & \textcolor{blue}{$C_3^-\!\geq\!\theta$} \\ \hline
 $4$ & & & $C_4^-\!\geq\!\theta$ & & & &  & & \\ \hline
 $5$ & & & & \textcolor{blue}{$C_5^+\!<\!\theta$} & $C_5^+\!<\!\theta$ & $C_5^+\!<\!\theta$ & & $C_5^-\!\geq\!\theta$ & \textcolor{blue}{$C_5^-\!\geq\!\theta$} \\ \hline 
 $6$ & & & & & & & $C_6^+\!<\!\theta$ & &  \\ \hline
 $7$ & & & & & & & & $C_7^-\!<\!\theta$ & $C_7^-\!\geq\!\theta$  \\ \hline
 $8$ & & $C_8^-\!\geq\!\theta$ & $C_8^-\!<\!\theta$ & & $C_8^-\!\geq\!\theta$ & $C_8^-\!<\!\theta$ & & &  \\ \hline
\end{tabular}
\renewcommand{\arraystretch}{1}
\end{center}
\caption {Existence conditions for MAIN SHORT states (black) and of the conditions derived from \ref{multistability_cond} (blue). }
\label{tab:multistability}
\end {table}

Inspecting this tables demonstrates that for each pair of MAIN SHORT states $(\psi_1,\psi_2)$ except $(I,S\!B)$ and $(I,S\!D)$ there exist at least one index $i$ for which either (a) $C_i^-\!<\!\theta$ for $\psi_1$ ($\psi_2$) and $C_i^-\!\geq\!\theta$ for $\psi_1$ ($\psi_2$) or (b) $C_i^+\!<\!\theta$ for $\psi_1$ ($\psi_2$) and $C_i^-\!\geq\!\theta$ for $\psi_1$ ($\psi_2$). Both (a) and (b) lead to conditions that cannot be satisfied simultaneously in the parameter space. This is obvious for case (a). For case (b) this holds because, since $N^- \!\geq\! N^+$ and $M^- \!\geq\! M^+$, we have $C_i^- \leq C_i^+$, $\forall i=2,..,8$. Figure \ref{figure4}C shows the stability regions for states $I$, $S\!B$ and $S$ at varying $c$ and $D\!F$, demonstrating that bistability between the pairs $(I,S\!B)$ and $(I,S\!D)$ can occur (note $I$ and $SD$ have a conjugate, hence we talk of \emph{multistability} for this Theorem). 
\end{proof}

\subsection{Analysis of $2T\!R$-periodic SHORT CONNECT states} \label{appendix1}
\begin{theorem} 
There is an injective map: 
\begin{align*}
\varphi \colon SC & \rightarrow B(2,6)\\
\psi &\mapsto W =
\begin{bmatrix} [c|c]
    W_1 & W_1 
\end{bmatrix}
=
\begin{bmatrix} [ccc|ccc]
    x_A^1 & y_A^1 & z_A^1 & x_A^2 & y_A^2 & z_A^2 \\
    x_B^1 & y_B^1 & z_B^1 & x_B^2 & y_B^2 & z_B^2
\end{bmatrix}
\end{align*}
Where, for $i\!=\!1,2$, $W_i$ is the matrix forms of $\psi$ during the interval $I_i$ defined in \ref{matricial_representation_CONN}, and: 
\begin{equation}
s_B^{i \pm}\!=\!N^{\pm}z_B^j\!+\!M^{\pm}(1-z_B^j)z_B^i, 
\quad \mbox{and} \quad 
s_A^{i \pm}\!=\!N^{\pm}z_A^j\!+\!M^{\pm}(1-z_A^j)z_A^i, 
\quad \forall i,j\!=\!1,2, i\!\neq\!j
\label{syn_values_CONNECT_SHORT}
\end{equation} 
In addition, let $\varphi^{I_1}$ ($\varphi^{I_2}$) be the map defined in Theorem \ref{CONNECT_R} for $\psi$ in $I_1$ ($I_2$). Then:
$$Im(\varphi) \!=\! \Gamma_{2T\!R}\!\eqdef\! \{W \!=\!
\begin{bmatrix} [c|c]
    W_1 & W_2 
\end{bmatrix} :
W_1 \in Im(\varphi^{I_1}), W_2 \in Im(\varphi^{I_1}) \mbox{ satisfy conditions 1-11}  \}$$
\begin{enumerate} 
 \label{connect_permutations} 
 \item (a) $z_A^i \geq y_A^i \geq x_A^i$ and (b) $z_B^i \geq y_B^i \geq x_B^i$, for $i=1,2$
 \item (a) If $x_A^i=x_B^i=0 \Rightarrow y_A^i=y_B^i=0$, for $i=1,2$
 \item (a) If $z_A^1=z_B^2=1 \Rightarrow x_A^1=x_B^2$ and (b) if $z_A^2=z_B^2=1 \Rightarrow x_A^2=x_B^1$
 \item (a) If $z_B^1=z_B^2 \Rightarrow x_A^1 \geq x_A^2$ and (b) if $z_A^1=z_A^2 \Rightarrow x_B^2 \geq x_A^1$
 \item (a) If $z_A^2=1 \Rightarrow x_B^1 \leq r$ and (b) if $z_B^1=1 \Rightarrow x_A^2 \leq r$, for any entry $r$ in $V$
 \item If $z_A^2=z_B^2$ and $z_A^1=z_B^2 \Rightarrow x_A^1 \geq x_B^1$ and $x_B^2 \geq x_A^2$ 
 \item If $z_A^1>y_A^1$ or $z_A^2>y_A^2$ or $z_B^1>y_B^1$ or $z_B^2>y_B^2$
 \item (a) $z_A^1 \neq 0$ or $z_A^2 \neq 0$ and (b) $z_B^1 \neq 0$ or $z_B^2 \neq 0$
 \item (a) $z_A^1=z_B^1=1, y_A^2=y_B^2 \Rightarrow z_B^2 \geq z_A^2$ and (b) $z_A^2=z_B^2=1, y_A^1=y_B^1 \Rightarrow z_A^1 \geq z_B^1$
 \item $z_A^2=1, z_B^1=1 \Rightarrow z_B^2=1$ 
 \item $z_A^1=z_B^2, z_B^1=z_A^2, x_A^1=x_B^2 \Rightarrow y_A^2=y_B^1$
\end{enumerate}
\end{theorem}
\begin{proof}
By definition, for each state $\psi \in SC$ at least one unit turns ON at some time $t^* \in (0,T\!D] \cup (T\!R,T\!R\!+\!T\!D]$. This means that	 $\psi$ may be MAIN during interval $I_1$ ($I_2$) and CONNECT during interval $I_2$ ($I_1$), or CONNECT during both intervals $I_1$ and $I_2$. In the latter scenario Theorem \ref{main_connect} implies that $\psi$ has a 2 by 3 matrix form $W_1$ ($W_2$) defined during interval $I_1$ ($I_2$). If $\psi$ is MAIN during $I_1$ ($I_2$), Remark \ref{link_main_connect_matr} guarantees that $\psi$ can still be represented during interval $I_1$ by the same matrix form of CONNECT states given in Theorem \ref{main_connect}. These considerations guarantee that the transformation given in Theorem \ref{main_connect} can be applied to both intervals $I_1$ and $I_2$, thus proving that the map $\varphi$ is well-defined and injective. 

We skip the proof of the identities \ref{syn_values_CONNECT_SHORT}, since it is analogous the one given in the proof of identies \ref{syn_values_MAIN_SHORT} of Theorem \ref{MAIN_2TR}. We now prove that each matrix $W \in Im(\varphi)$ satisfies conditions \ref{connect_permutations}. For conditions 1-5 and 8-9 we prove only conditions (a) since the (b) ones are analogous. The proof of the first two conditions follows trivially from the definition of the entries of $W$. We thus prove the other conditions below
\begin{enumerate}
 \setcounter{enumi}{2}
 \item $z_A^1\!=\!z_B^2\!=\!1 \Rightarrow s_A^{2-}\!=\!s_B^{1-}\!=\!N^- \Rightarrow x_B^1\!=\!x_A^1\!=\!H(c\!-\!bN^-)$ 
 \item $z_B^1\!=\!z_B^2 \Rightarrow s_B^{1-}\!=\!s_B^{2-} \Rightarrow x_A^1\!=\!H(c\!-\!bs_B^{1-})\!\geq\!H(d\!-\!bs_B^{2-})\!=\!x_A^2$
 \item Since $z_A^1\!=\!1 \Rightarrow s_A^{1-}\!=\!N^-$. Therefore, $x_B^1\!=\!H(d\!-\!N^-)$. Any entry $r$ of $V$ is either $H(c\!-\!s_A^{i\pm})$, $H(d\!-\!s_A^{i\pm})$, $H(c\!-\!s_B^{i\pm})$ or $H(d\!-\!s_B^{i\pm})$, for some $i\!=\!1,2$. Since $s_A^{i\pm},s_B^{i\pm}\!\leq\!N^-$ and $d\!\leq\!c$, we must have that $r \!\geq\!H(d\!-\!N^-)\!=\!z_A^1$
 \item Given $z_A^2\!=\!z_B^2$ and $z_A^1\!=\!z_B^1$ and $c\!\geq\!d$ we have that
 $$ x_A\!=\!H(c\!-\!bN^-z_B^2\!-\!bM^-(1\!-\!z_B^2)z_B^1)\! \geq \!H(d\!-\!bN^-z_A^2\!-\!bM^-(1\!-\!z_A^2)z_A^1)\!=\!x_B $$
 \item By definition, for any CONNECT state $s$ at least one must turn ON within the interval $I\!=\!I_1$ or $I\!=\!I_2$ (or both). If $I\!=\!I_1$, from Theorem \ref{CONNECT_R} we have that the matrix form $W_1 \in \Gamma$. In particular, it must satisfy $y_A^1\!<\!z_A^1$ or $y_B^1\!<\!z_B^1$. Similarly, if $I\!=\!I_2$, then $W_2 \in \Gamma$ and $y_A^2\!<\!z_A^2$ or $y_B^2\!<\!z_B^2$
 \item By contradiction suppose there exist a CONNECT state $s$ such that $z_A^i \!=\! 0$, for $i=1,2$. This leads to $s_A^{i\pm}=0$ and to $x_A^i \!=\!y_A^i \!=\! 0$ (from 1). Thus, since we hypothesise $c\!\geq\!\theta$ then $x_B^2\!=\!H(c)\!=\!1$, which guarantees $z_B^2\!=\!y_B^2\!=\!1$ (again from 1). Since $z_A\!=\!0$ we also have that $x_A^1\!=\!0$. This leads to $x_B^2\!=\!y_B^2\!=\!z_B^2\!=\!H(d)$. This leads to the matrix form 
\[\begin{bmatrix} [ccc|ccc]
    0 & 0 & 0 & 0 & 0 & 0 \\
    H(d) & H(d) & H(d) & 1 & 1 & 1
\end{bmatrix}\]
Since $y_A^i\!=\!z_A^i$ or $y_B^i\!=\!z_B^i$ for $i\!=\!1,2$, we have $W_1 \notin Im(\varphi^{I_1})$ and $W_2 \notin Im(\varphi^{I_2})$, which is absurd. 
 \item Given $z_A^1\!=\!z_B^1\!=\!1$ we have $s_A^{1+}\!=s_A^{1+}\!=\!N^+$. Since $y_A^2\!=\!y_B^2$ and $d\!\leq\!c$ we have that $z_B^2 \!=\!H(ay_A^2\!-\!bN^+\!+\!c) \!\geq\! H(ay_B^2\!-\!bN^+\!+\!d)\!=\!z_A^2$
 \item If $y_B^2\!=\!1$ from (1) we have $z_B^2\!=\!1$, which proves the claim. Thus we can assume that $y_B^2\!=\!0$. Condition $z_B^1\!=\!1$ implies $s_B^{2+}\!=\!N^+$. This identity and $y_B^2\!=\!0$ implies that $z_A^2\!=\!H(d\!-\!bN^+)$. Thus from the hypothesis $z_A^2\!=\!1$ we have $d\!-\!bN^+ \!\geq\! \theta$. Moreover, since $d\! \leq \! c$, $ay_B^2 \geq \! 0$, and $s_A^{2+}\! \leq \!N^+$ we must have $z_B^2\!=\!H(ay_B^2\!+\!c\!-\!bs_A^{2+})\!\geq\!H(d\!-\!bN^+)\!=\!1$
 \item Given $z_A^1\!=\!z_B^2, z_B^1\!=\!z_A^2, x_A^1\!=\!x_B^2$ it obviously follows that 
 $$ y_A^2\!=\!H(ax_B^2\!-\!bN^-z_B^1\!-\!bM^-(1\!-\!z_B^1)z_B^2\!+\!d)\!=\!H(ax_A^1\!-\!bN^-z_A^2\!-\!bM^-(1\!-\!z_A^2)z_A^1\!+\!d)\!=\!y_B^1 $$
\end{enumerate}

Next, we algorithmically find all matrices in $\Gamma_{2T\!R}$. We proceed by generating all 2 by 6 binary matrices matrices $W\!=\!\begin{bmatrix} [c|c] W_1 \!&\!W_2 \end{bmatrix}$ with entries satysfying conditions 1-11. In total, we find that $|\Gamma_{2T\!R}|\!=\!15$, thus implying $|Im(\varphi)|\!\leq\!15$. 

Due to the model's symmetry, for any matrix $W=\varphi(\psi)$ of an asymmetrical state $\psi$ there exist a matrix $W' \in \Gamma_{2T\!R}$ image of the state $\psi'$ conjugate to $\psi$, and this matrix is defined by swapping the first row of $W_1$ with the second row of $W_2$ and the second row of $W_1$ with the first row of $W_2$. Notably, both $\psi$ and $\psi'$, and thus also $W$ and $W'$, exist under the same parameter conditions. The top rows of Table \ref{tab:MAIN_table} shows all matrices $V \in \Omega$ that are an image of either of a symmetrical state or one of two conjugate states and their corresponding names (1st row). 

\begin {table}[h]
\begin{tabular}{ |c|c|c|c|c|c|c|c|c| } 
 \hline
  \footnotesize $ZcS^*$ & \footnotesize $ZcAP$ & \footnotesize $ZcAS^*$ & \footnotesize $ZcI$ & \footnotesize $ScAS^*$ & \footnotesize $SDcAS^*$ & \footnotesize $ScSD^*$ & \footnotesize $APcAS^*$ & \footnotesize $APcI$ \\ \hline
 \ \parbox{1.26cm}{\scriptsize $\begin{array}{c} 
  001 000 \\
  001 000 \\
 \end{array}$}
 & \parbox{1.26cm}{\scriptsize $\begin{array}{c}
  001 000 \\
  000 001 \\
 \end{array}$}
 & \parbox{1.26cm}{\scriptsize $\begin{array}{c}
  001 001 \\
  000 001 \\
 \end{array}$}
 & \parbox{1.26cm}{\scriptsize $\begin{array}{c}
  001 001 \\
  001 001 \\
 \end{array}$}
 & \parbox{1.26cm}{\scriptsize $\begin{array}{c}
  001 111 \\
  000 001 \\
 \end{array}$}
 & \parbox{1.26cm}{\scriptsize $\begin{array}{c}
  001 111 \\
  000 011 \\
 \end{array}$}
 & \parbox{1.26cm}{\scriptsize $\begin{array}{c}
  111 000 \\
  001 000 \\
 \end{array}$}
 & \parbox{1.26cm}{\scriptsize $\begin{array}{c}
  111 001 \\
  000 111 \\
 \end{array}$}
 & \parbox{1.26cm}{\scriptsize $\begin{array}{c}
  111 001 \\
  001 111 \\
 \end{array}$}
 \\ \hline
 $\begin{aligned}[t] 
  C_4^-\!&<\!\theta \\ 
  C_4^+\!&\geq\!\theta \\ 
  C_2^+\!&\geq\!\theta
 \end{aligned}$ & 
 
 $\begin{aligned}[t] 
  C_2^+\!&<\!\theta \\
  C_3^-\!&<\!\theta \\
  C_3^+\!&\geq\!\theta
 \end{aligned}$ & 

 see \ref{ZcAS} & 
 
 $\begin{aligned}[t] 
  C_3^-\!&<\!\theta \\
  C_3^+\!&\geq\!\theta \\
  C_5^+\!&\geq\!\theta
 \end{aligned}$ &  
 
 $\begin{aligned}[t] 
  C_3^+\!&\geq\!\theta \\ 
  C_5^+\!&<\!\theta \\ 
  C_8^-\!&\geq\!\theta \\ 
  C_6^-\!&<\!\theta
 \end{aligned}$ & 
 
 $\begin{aligned}[t]
 C_3^-\!&<\!\theta \\ 
 C_3^+\!&\geq\!\theta \\ 
 C_5^+\!&<\!\theta \\ 
 C_8^-\!&\geq\!\theta \\ 
 C_6^-\!&\geq\!\theta
 \end{aligned}$ & 
 
 $\begin{aligned}[t] 
  C_4^-\!&\geq\!\theta \\ 
  C_2^-\!&<\!\theta \\ 
  C_2^+\!&\geq\!\theta \\ 
  C_3^+\!&<\!\theta
 \end{aligned}$ &  
 
 $\begin{aligned}[t] 
  C_3^-\!&\geq\!\theta \\ 
  C_5^+\!&<\!\theta \\ 
  C_2^-\!&<\!\theta \\ 
  C_2^+\!&\geq\!\theta
 \end{aligned}$ &  
 
 $\begin{aligned}[t] 
  C_3^-\!&\geq\!\theta \\ 
  C_5^-\!&<\!\theta \\ 
  C_5^+\!&\geq\!\theta
 \end{aligned}$ 
 \\ \hline

 $\begin{aligned}[t] C_{9}\!<\!\theta \end{aligned}$ & 
 $-$ & 
 $\begin{aligned}[t] C_{10}\!<\!\theta \end{aligned}$ & 
 $\begin{aligned}[t] C_{10}\!<\!\theta \end{aligned}$ & 
 $\begin{aligned}[t] C_{10}\!<\!\theta \end{aligned}$ &
 $\begin{aligned}[t] C_{10}\!<\!\theta \end{aligned}$ &
 $\begin{aligned}[t] C_{9}\!<\!\theta \end{aligned}$ & 
 $\begin{aligned}[t] C_{10}\!<\!\theta \end{aligned}$ &
 $\begin{aligned}[t] C_{10}\!<\!\theta \end{aligned}$ 
 \\ \hline
\end{tabular}
\label{tab:CONNECT_table_appendix}
\caption {Matrix form and existence conditions of $2T\!R$-periodic SHORT CONNECT states. Asymmetrical states in *.}
\end {table}

The analysis of existence conditions for SHORT CONNECT states is slightly more involved than the one done in Theorem \ref{MAIN_2TR} for SHORT MAIN states. The reason is that for the well-definedness conditions for the entries of each SHORT MAIN state's matrix form are necessary and sufficient for determining the dynamics of each state in $I_1$ and $I_2$. In the case of CONNECT states, this property is not valid. Therefore, we analyse each of the remaining 15 matrices given in Table \ref{tab:CONNECT_table_appendix} separately using conditions $C_{1\!-\!5}$ and $M_{1\!-\!6}$. 
Similar to the proof of formulas \ref{syn_values_MAIN_SHORT} of Theorem \ref{MAIN_2TR}, one may show that that variables $s_A(t\!-\!D)$ and $s_B(t\!-\!D)$ of any SHORT CONNECT states are monotonically decreasing and depend on functions
$$ N(t)=e^{(-T\!R\!-\!D\!-\!t)/\tau_i} \quad \mbox{and} \quad M(t)=e^{(-2T\!R\!-\!D\!-\!t)/\tau_i}. $$
More precisely, these variables satisfy the following $\forall t \in I_1 \cup I_2$:
\begin{equation}
s_B(t\!-\!D)\!=\!N(t)z_B^j\!+\!M(t)(1-z_B^j)z_B^i, 
\mbox{ and } 
s_A(t\!-\!D)\!=\!N(t)z_A^j\!+\!M(t)(1-z_A^j)z_A^i, 
\; \forall i,j\!=\!1,2, i\!\neq\!j.
\nonumber
\end{equation} 
Obviously, this is an extension of the proof of \ref{syn_values_CONNECT_SHORT}, since these quantities can be obtained by evaluating the equations above at time $t\!=\!0, TD, T\!R$ and $T\!R\!+\!T\!D$. Using these identities we now prove that the existence conditions for each state shown in the third row of Table \ref{tab:CONNECT_table_appendix}. 
\begin{enumerate}
 \item \textbf{ZcS} - This state is CONNECT during interval $I_1$ (satisfying condition $C_5$) and MAIN during interval $I_2$ (satisfying condition $M_6$). Since $z_A^1\!=\!z_B^1\!=\!1$ and $z_A^2\!=\!z_B^2\!=\!0$ we have 
 $$ s_A(t\!-\!D)\!=\!s_B(t\!-\!D)\!=\!M(t), \; \forall t \in I_1 \quad \mbox{and} \quad s_A(t\!-\!D)\!=\!s_B(t\!-\!D)\!=\!N(t), \: \forall t \in I_2. $$
 In particular, evaluating these equations at time $t\!=\!0, TD, T\!R$ and $T\!R\!+\!T\!D$ we obtain 
 $$ s_A^{1 \pm}\!=\!s_B^{1 \pm}\!=\!M^{\pm} \quad \mbox{and} \quad s_A^{2 \pm}\!=\!s_B^{2 \pm}\!=\!N^{\pm}. $$
 Condition $C_5$ on the interval $I_1$ requires that A(B) turns ON at the (unique) times $t^*$($s^*$) in $(0,T\!D]$. It must be that $t^*\!\leq\!s^*$. Indeed, on the contrary suppose that B turns ON at time $s^*\!<\!t^*$. Thus we must have $d\!-\!bs_A(s^*\!-\!D)\!=\!\theta$ (i.e.\ point $(0,1)$ is an equilibrium for the fast subsystem at time $s^*$) and $c\!-\!bs_B(s^*\!-\!D)\!<\!\theta$ (i.e.\ point $(1,0)$ is not an equilibrium for the fast subsystem at time $s^*$). This is absurd because $c\!\geq\!d$ and $s_A(s^*\!-\!D)\!=\!s_B(s^*\!-\!D)\!=\!M(s^*)$. Thus necessary and sufficient existence conditions for ZcS are given by conditions $C_5$ under the case $t^*\!\leq\!s^*$, which are
$$ C_4^-\!=\!c-\!bM^-\!<\!\theta, \quad C_4^+\!=\!c-\!bM^+\!\geq\!\theta \quad \mbox{and} \quad C_2^+\!=\!a\!-\!bM^+\!d\!\geq\!\theta.$$	
 Lastly we need to ensure that ZcS satisfies condition $M_6$ on the interval $I_1$. More precisely, these conditions are $ c\!-\!bN^+\!<\!\theta$ and $ d\!-\!bN^+\!<\!\theta$. We notice that, since $N^+\!\geq\!M^-$, both of these conditions automatically hold due to $C_4^-\!=\!c-\!bM^-\!<\!\theta$. 
 \item \textbf{ZcAP} - This state is CONNECT for both intervals $I_1$ (satisfying condition $C_4$) and $I_2$ (satisfying condition $C_3$). Since $z_A^1\!=\!z_B^2\!=\!1$ and $z_A^2\!=\!z_B^1\!=\!0$ we have $s_B^{1 \pm}\!=\!s_A^{2 \pm}\!=\!N^{\pm}$ and $s_A^{1+}\!=\!s_B^{2+}\!=\!M^{+}$. Thus from the conditions given in $C_3$ we have that
 $$ C_3^-\!=\!c\!-\!bN^-\!<\!\theta, \quad C_3^+\!=\!c\!-\!bN^+\!\geq\!\theta, \quad C_2^+\!=\!a\!-\!bM^+\!+\!d\!<\!\theta. $$
 \item \textbf{ZcI} - This state is CONNECT for both intervals $I_1$ (satisfying condition $C_5$) and $I_2$ (satisfying condition $C_5$). Conditions $z_A^1\!=\!z_B^1\!=\!z_A^2\!=\!z_B^2\!=\!1$ lead to 
 $$ s_A(t\!-\!D)\!=\!s_B(t\!-\!D)\!=\!N(t) \quad \forall t \in I_1 \cup I_2. $$ 
 In particular, evaluating these equations at time $t\!=\!0, TD, T\!R$ and $T\!R\!+\!T\!D$ we obtain $ s_A^{1 \pm}\!=\!s_B^{1 \pm}\!=\!s_A^{2 \pm}\!=\!s_B^{2 \pm}\!=\!N^{\pm} $. Since the synaptic variables evolve equally on both intervals and due to the model's symmetry (see \ref{symmetry}) it must be that A and B turn ON at the same time $t^*$ during intervals $I_1$ and $I_2$ respectively, and B and A turn ON at the same time $s^*$ during intervals $I_1$ and $I_2$ respectively (applying condition $C_5$ on both intervals). Similar considerations made for the case ZcS lead to $t^*\!\leq\!s^*$. Thus the existence conditions for ZcAP are given by conditions $C_5$ under the case $t^*\!\leq\!s^*$, and they are
  $$ C_3^-\!=\!c\!-\!bN^-\!<\!\theta, \quad C_3^+\!=\!c\!-\!bN^+\!\geq\!\theta, \quad C_5^+\!=\!a\!-\!bN^+\!+\!d\!\geq\!\theta. $$
 \item \textbf{ScAS} - This state is CONNECT for both intervals $I_1$ (satisfying condition $C_4$) and $I_2$ (satisfying condition $C_1$). Since $z_A^1\!=\!z_A^2\!=\!z_B^2\!=\!1$ and $z_B^1\!=\!0$ we have $s_A^{1 \pm}\!=\!s_A^{2 \pm}\!=\!s_B^{1 \pm}\!=\!N^{\pm}$ and $s_B^{2\pm}\!=\!M^{\pm}$. Condition $C_4$ on interval $I_1$ leads to (1) $c\!-\!bN^-\!<\!\theta$, (2) $c\!-\!bN^+\!\geq\!\theta$ and (3) $a\!-\!bN^+\!+\!d\!<\!\theta$. Condition $C_1$ on interval $I_2$ lead to (4) $d\!-\!bM^-\!\geq\!\theta$, (5) $a\!-\!bN^-\!+\!c\!<\!\theta$ and (6) $a\!-\!bN^+\!+\!c\!\geq\!\theta$. Conditions (1) and (6) can be discarded because they derive respectively from conditions (5) and (2) (using the properties $N^-\!\geq\!N^+$ and $a\!\geq\!0$). Thus, the remaining conditions are 
 $$ C_3^+\!=\!c\!-\!bN^+\!\geq\!\theta, \quad C_5^+\!=\!a\!-\!bN^+\!+\!d\!<\!\theta, \quad C_8^-\!=\!d\!-\!bM^-\!\geq\!\theta \quad \mbox{and} \quad C_6^-\!=\!a\!-\!bN^-\!+\!c\!<\!\theta. $$
 \item \textbf{SDcAS} - This state is CONNECT for interval $I_1$ (satisfying condition $C_4$) and MAIN for interval $I_2$ (satisfying condition $M_2$). Like in the case of ScAS, since $z_A^1\!=\!z_A^2\!=\!z_B^2\!=\!1$ and $z_B^1\!=\!0$ we have $s_A^{1 \pm}\!=\!s_A^{2 \pm}\!=\!s_B^{1 \pm}\!=\!N^{\pm}$ and $s_B^{2\pm}\!=\!M^{\pm}$. Condition $C_4$ on the interval $I_1$ implies conditions (1-3) in ScAS. Condition $M_2$ on interval $I_2$ implies (4) $d\!-\!bM^-\!\geq\!\theta$, (5) $c\!-\!bN^-\!<\!\theta$ and (6) $a\!-\!bN^-\!+\!c\!\geq\!\theta$. Obviously, condition (1) can be discarded because it is the same as (5), and the remaining conditions thus are
 $$ C_3^-\!=\!c\!-\!bN^-\!<\!\theta, C_3^+\!=\!c\!-\!bN^+\!\geq\!\theta, C_5^+\!=\!a\!-\!bN^+\!+\!d\!<\!\theta, C_8^-\!=\!d\!-\!bM^-\!\geq\!\theta, C_6^-\!=\!a\!-\!bN^-\!+\!c\!\geq\!\theta. $$
 \item \textbf{ScSD} - This state is CONNECT for interval $I_1$ (satisfying condition $C_1$) and MAIN for interval $I_2$ (satisfying condition $M_6$). As in case ZcS we have $s_A^{1 \pm}\!=\!s_B^{1 \pm}\!=\!M^{\pm}$ and $s_A^{2 \pm}\!=\!s_B^{2 \pm}\!=\!N^{\pm}$. Condition $C_1$ on interval $I_1$ leads to $c\!-\!bM^-\!\geq\!\theta$, $a\!-\!bM^-\!+\!d\!<\!\theta$ and $a\!-\!bM^+\!+\!d\!\geq\!\theta$. Condition $M_6$ on interval $I_2$ implies (1) $d\!-\!bN^+\!<\!\theta$ and (2) $c\!-\!bN^+\!<\!\theta$. Obviously, since $d\!\leq\!c$, (2) implies (1), and thus (1) can be discarded. The remaing conditions are
  $$ C_4^-\!=\!c\!-\!bM^-\!\geq\!\theta, \quad C_2^-\!=\!a\!-\!bM^-\!+\!d\!<\!\theta, \quad C_2^+\!=\!a\!-\!bM^+\!+\!d\!\geq\!\theta, \quad C_3^+\!=\!c\!-\!bN^+\!<\!\theta. $$
 \item \textbf{APcAS} - This state is CONNECT for interval $I_1$ (satisfying condition $M_5$) and MAIN for interval $I_2$ (satisfying condition $C_2$). Similarly to the case ScAS we have that $s_A^{1 \pm}\!=\!s_A^{2 \pm}\!=\!s_B^{1 \pm}\!=\!N^{\pm}$ and $s_B^{2\pm}\!=\!M^{\pm}$. Condition $M_5$ on interval $I_1$ leads to $c\!-\!bN^-\!\geq\!\theta$ and $a\!-\!bN^+\!+\!d\!<\!\theta$. Condition $C_2$ on interval $I_2$ leads to $c\!-\!bN^-\!\geq\!\theta$ (again), $a\!-\!bM^-\!+\!d\!<\!\theta$ and $a\!-\!bM^+\!+\!d\!\geq\!\theta$. In summary these conditions are
 $$ C_3^-\!=\!c\!-\!bN^-\!\geq\!\theta, \quad C_5^+\!=\!a\!-\!bN^+\!+\!d\!<\!\theta, \quad C_2^-\!=\!a\!-\!bM^-\!+\!d\!<\!\theta, \quad C_2^+\!=\!a\!-\!bM^+\!+\!d\!\geq\!\theta. $$
 \item \textbf{APcINT} - This state is CONNECT for both intervals $I_1$ and $I_2$, satisfying condition $C_1$ and $C_2$ respectively. As for ZcI, conditions $z_A^1\!=\!z_B^1\!=\!z_A^2\!=\!z_B^2\!=\!1$ lead to
  $$ s_A(t\!-\!D)\!=\!s_B(t\!-\!D)\!=\!N(t) \quad \forall t \in I_1 \cup I_2. $$ 
 Thus we obtain $ s_A^{1 \pm}\!=\!s_B^{1 \pm}\!=\!s_A^{2 \pm}\!=\!s_B^{2 \pm}\!=\!N^{\pm} $. Moreover, since the synaptic variables evolve equally on both intervals and due to the model's symmetry it must be that A and B turn ON at the same time $t^*$ during intervals $I_1$ and $I_2$ respectively (applying conditions $C_{1-2}$ on $I_{1-2}$). Moreover conditions $C_1$ and $C_2$ are equal and lead to $C_3^-\!=\!c\!-\!bN^-\!\geq\!\theta$, $C_5^-\!=\!a\!-\!bN^-\!+\!d\!<\!\theta$ and $C_5^+\!=\!a\!-\!bN^+\!+\!d\!\geq\!\theta$. 
 \item \textbf{ZcAS} - Showing the existence conditions for this state is the most involved case. This state is CONNECT for both intervals $I_1$ (satisfying condition $C_4$) and $I_2$ (satisfying condition $C_5$). Since $z_A^1\!=\!z_A^2\!=\!z_B^2\!=\!1$ and $z_B^1\!=\!0$ we have 
 $$ s_A(t\!-\!D)\!=\!N(t) \quad \mbox{and} \quad s_B(t\!-\!D)\!=\!M(t), \forall t \in I_2. $$ 
 In particular, evaluating these equations at time $t\!=\!0, TD, T\!R$ and $T\!R\!+\!T\!D$ we obtain $s_A^{1 \pm}\!=\!s_A^{2 \pm}\!=\!s_B^{1 \pm}\!=\!N^{\pm}$ and $s_B^{2\pm}\!=\!M^{\pm}$. For condition $C_5$ on $I_2$ we have that B and A turn ON at times $t^*$ and $s^*$ in $(T\!R,T\!R\!+\!T\!D]$, respectively. We have two cases to consider: 
\begin{itemize}
  \item Case $t^*\!<\!s^*$. From the evolution of the synaptic variables and since they are monotonically decaying we may express existence conditions as follows:
  
(P1) $\exists t^* \!\in\! (T\!R,T\!R\!+\!T\!D]\!:  c\!-\!bN(t^*)\!=\!\theta \Leftrightarrow C_3^-\!=\!c\!-\!bN^-\!<\!\theta$ and $C_3^+\!=\!c\!-\!bN^+\!\geq\!\theta$. 

(P2) $\forall s \in (0,t^*): d\!-\!bM(s)\!<\!\theta \Leftrightarrow d\!-\!bM(t^*)\!<\!\theta$ 

(P3) $\exists s^* \in (t^*,T\!R\!+\!T\!D]: a\!-\!bM(s^*)\!+\!d\!\geq\!\theta \Leftrightarrow C_2^+\!=\!a\!-\!bM^+\!+\!d\!\geq\!\theta$

 Where (P1) guarantees that B turns ON at $t^*$, (P2) that A is OFF $\forall s\!\leq\!t^*, s \in I_2$ and (P3) that A turns ON at time $s^*$. Thus (P2) guarantees $s^*\!>\!t^*$. From (P1) we have that 
  $$t^* \!=\! N^{-1}((c\!-\!\theta)/b) \!=\! \tau_i \log((c\!-\!\theta)/b)\!+\!(T\!R\!-\!D)$$
 By substituting this identity in (P2) and we obtain that $d\!-\!bM(t^*)\!<\!\theta$ if and only if $K\!=\!c\!-\!(d\!-\!\theta)e^{T\!R/\tau_i}\!>\!\theta$. Lastly we need to guarantee conditions $C_4$ on $I_1$. Two conditions are $C_3^-\!=\!c\!-\!bN^-\!<\!\theta$ and $C_3^+\!=\!c\!-\!bN^+\!\geq\!\theta$, which are equivalent to case (P1). The second condition is that $C_5^+\!=\!a\!-\!bN^+\!+\!d\!<\!\theta$. 
 \item Case $t^*\!\geq\!s^*$. Similar to the previous case we can formulate the following conditions:

(Q1) $\exists s^* \in (T\!R,T\!R\!+\!T\!D]: d\!-\!bM(s^*)\!=\!\theta \Leftrightarrow C_8^-\!=\!d\!-\!bM^-\!<\!\theta$ and $C_8^+\!=\!d\!-\!bM^+\!\geq\!\theta$

(Q2) $\exists t^* \in (s^*\!-\!T\!R,T\!D]: c\!-\!bN(t^*)\!=\!\theta \Leftrightarrow c\!-\!bN(s^*)\!<\!\theta$ and $C_3^+\!=\!c\!-\!bN^+\!\geq\!\theta$

 Where (Q1) guarantees that A turns ON at $s^* \in (T\!R,T\!R\!+\!T\!D]$ and (Q2) that it turns ON at time $t^*$, wher $t^*\!-\!T\!R\!\geq\!s^*$ (ie one of conditions $C_4$ on $I_1$). From (Q1) we have that 
 $$s^* \!=\! N^{-1}((d\!-\!\theta)/b) \!=\! \tau_i \log((d\!-\!\theta)/b)\!+\!(2T\!R\!-\!D),$$
 Thus the first condition in (Q2) is equivalent to $K\!\leq\!\theta$. Condition $C_3^+\!\geq\!\theta$ and $a\!\geq\!0$ imply $a\!+\!c\!-\!bN^+\!\geq\!\theta$, thus completing conditions $C_2$ on $I_2$. Analogously to the previous case, the last condition to be ensures is $C_5^+\!=\!a\!-\!bN^+\!+\!d\!<\!\theta$. 
\end{itemize}
Thus, in summary, the conditions for both cases are: 
\begin{equation}
 \begin{cases} 
 C_3^-\!<\!\theta, C_3^+\!\geq\!\theta, C_2^+\!\geq\!\theta, C_5^+\!<\!\theta, & \mbox{if } K\!>\!\theta \\ 
 C_8^-\!<\!\theta, C_8^+\!\geq\!\theta, C_3^+\!\geq\!\theta, C_5^+\!<\!\theta, & \mbox{if } K\!\leq\!\theta. 
 \end{cases}
 \label{ZcAS}
\end{equation}
This completes the proof of the existence conditions for ZcAS. 
\end{enumerate}
Notably, we numerically simulated each state that correspong to a matrix in $\Gamma_{2T\!R}$, thus proving that this its conditions can be satisfied in a non-empty region of parameters. This proves that $Im(\rho)\!=\!\Gamma_{2T\!R}$.
\end{proof}

\subsection{Analysis of $2T\!R$-periodic MAIN LONG states} \label{appendix2}
In this section we analyze the existence conditions for $2T\!R$-periodic LONG MAIN states. To do so we use a similar analysis to the one described in the section \ref{short_2TR} of the main text. The first step is to extend the matrix form definition to LONG states. Due to Lemma \ref{lem:LONG_conditions}, LONG states can exist only if there exist one active tone interval  $R\!=\!I_1$ or $R\!=\!I_2$ for which two conditions are satisfied. Let us name $R\!=\![\alpha,\beta]$. The conditions are: 
\begin{enumerate}
 \item Both units must be ON at time $\beta$
 \item $a\!-\!bs_A(\beta\!-\!D)\!\geq\!\theta$ and $a\!-\!bs_B(\beta\!-\!D)\!\geq\!\theta$
\end{enumerate}

We can then extend the definition of the matrix form of MAIN LONG states by including a last column in the matrix form of SHORT MAIN states. More precisely, the matrix form for a state $\psi \in LM$ is the $2 \times 6$ binary matrix $V$ defined as
\begin{equation}
V=
\begin{bmatrix} [c|c||c|c]
    V_1 & \vec{w}^1 & V_2 & \vec{w}^2
\end{bmatrix}
=
\begin{bmatrix} [cc|c||cc|c]
    x_A^1 & y_A^1 & w^1 & x_A^2 & y_A^2 & w^2 \\
    x_B^1 & y_B^1 & w^1 & x_B^2 & y_B^2 & w^2
\end{bmatrix}
\nonumber
\end{equation}
Where $V_1$ and $V_2$ are the same matrix forms defined for MAIN SHORT states, respectively, with entries defined by equations \ref{matricial_representation_MAIN}. Entries of the binary vectors $ \vec{w}^1$ and $\vec{w}^2$ guarantee that condition 2. is met for LONG states and they are defined by
\begin{equation}
 w^1=H(ay_A^1-bs_A^{1+})H(ay_B^1-bs_B^{1+}) \quad  \mbox{and} \quad w^2=H(ay_A^2-bs_A^{2+})H(ay_B^2-bs_B^{2+}).
\end{equation}
We remind the reader that $s_A^{1+}\!=\!s_A(T\!D\!-\!D)$, $s_B^{1+}\!=\!s_B(T\!D\!-\!D)$, $s_A^{2+}\!=\!s_A(T\!R\!+\!T\!D\!-\!D)$ and $s_B^{2+}\!=\!s_B(T\!R\!+\!T\!D\!-\!D)$. 
These quantities appear also in the definition of the $V_1$ and $V_2$ entries. In the case of LONG MAIN states they depend on both $N^\pm$ and $M^\pm$ defined in equations \ref{N_M_MAIN} and on the following quantities: 
\begin{equation}
 N_L^- \!=\! e^{-(TR\!-\!2D)/\tau_i}, \; N_L^+ = e^{-(TR\!+\!TD\!-\!2D)/\tau_i}, \; 
 M_L^- \!=\! e^{-(2TR\!-\!2D)/\tau_i}, \; M_L^+ = e^{-(2TR\!+\!TD\!-\!2D)/\tau_i}. 
 \label{N_M_LONG}
\end{equation}
We note that $N_L^+\!\geq\!N^+$, $N_L^-\!\geq\!N^-$, $M_L^+\!\geq\!M^+$ and $M_L^-\!\geq\!M^-$. Using a similar analysis carried to prove equations \ref{syn_values_MAIN_SHORT} in Theorem \ref{MAIN_2TR} one can easily show that:
\begin{equation} 
\label{syn_values_MAIN_LONG}
 \begin{split}
s_B^{i \pm}\!&=\!w^jN_L^{\pm}\!+\!(1\!-\!w^j)y_B^jN^{\pm}\!+\!(1\!-\!w^j)(1\!-\!y_B^j)w_B^iM_L^{\pm}\!+\!(1\!-\!w^j)(1\!-\!y_B^j)(1\!-\!w^i)y_B^iM^{\pm} \\
s_A^{i \pm}\!&=\!w^jN_L^{\pm}\!+\!(1\!-\!w^j)y_A^jN^{\pm}\!+\!(1\!-\!w^j)(1\!-\!y_A^j)w^iM_L^{\pm}\!+\!(1\!-\!w^j)(1\!-\!y_A^j)(1\!-\!w^i)y_A^iM^{\pm}
 \end{split}
\end{equation}
To analyse LONG MAIN states $\psi \in LM$ we may restrict to the case where the interval $R$ for which properties (1-2) given above are satisfy is $R\!=\!I_1$ (the case $R\!=\!I_2$ will be analysed using symmetry principles). Properties (1-2) may then be rewritten as (a) both units are ON at time $\beta\!=\!T\!D$, and (b) $a\!-\!bs_A^{1+}\!\geq\!\theta$ and $a\!-\!bs_B^{1+}\!\geq\!\theta$. From (a) we have that $(1,1)$ is an equilibrium for the fast subsystem at time $T\!D$, which implies that $V_1$ satisfies one of $M_{1-3}$ during the interval $I_1$ (see Section \ref{MAIN_definition}). From (b) we obtain $w^1\!=\!1$. Before we consider separately each of cases $M_{1-3}$, we note that the entries of the matrix form of any MAIN LONG state $\psi$ satisfy the properties stated in the next theorem. 
\begin{theorem}
The matrix form $V$ of any LONG MAIN state $\psi \in LM$ satisfies:
\begin{enumerate}
 \item $x_A^2 \leq x_B^2$
 \item If $w^2=1 \Rightarrow x_A^2\!=\!x_B^1, x_B^2\!=\!x_A^1, y_A^2\!=\!y_B^1$ and $y_B^2\!=\!y_A^1$
 \item $x_A^2 \leq x_B^1$ and $x_B^2 \leq x_A^1$
 \item If $w^2=y_A^2=y_B^2=0 \Rightarrow x_A^1\!\geq\!x_B^1$ 
 \item $x_A^2 \leq y_A^2$ and $x_B^2 \leq y_B^2$
 \item $x_A^2=x_B^2=0 \Rightarrow y_A^2=y_B^2=0$
\end{enumerate}
\end{theorem}
\begin{proof}
Since $w^1\!=\!1$, from the identities \ref{syn_values_MAIN_LONG} we have $s_B^{2-}\!=\!s_A^{2-}\!=\!N_L^-$, which leads to $x_A^2\!=\!H(d\!-\!bN_L^-)$ and $x_B^2\!=\!H(c\!-\!bN_L^-)$. Since $d\!\leq\!c$, we have (1). Similarly, if $w^2\!=\!1$, we have $s_B^{1\pm}\!=\!s_A^{1\pm}\!=\!N_L^\pm$. This implies $x_A^2\!=\!H(d\!-\!bN_L^-)\!=\!x_B^1$ and $x_B^2\!=\!H(c\!-\!bN_L^-)\!=\!x_A^1$. Analogously, one can easily show that $y_A^2\!=\!y_B^1$ and $y_B^2\!=\!y_A^1$ using the definition of these entries given in the definitions \ref{matricial_representation_MAIN}. Since $w^1\!=\!1$ we have that $s_A^{2-}\!=\!s_B^{2-}\!=\!N_L^-\!\geq\!s_A^{1-}$, which proves (3). Under the hypothesis of (4) we have that $s_B^{1\pm}\!=\!s_A^{1\pm}\!=\!M_L^\pm$. This and $c\!\geq\!d$ implies $x_A^1\!=\!H(c\!-\!bM_L^-)\!\geq\!H(d\!-\!bM_L^-)\!=\!x_B^1$, proving (4). Since $\psi$ is MAIN, conditions (5-6) derive from Theorem \ref{thm:matr_MAIN}.
\end{proof}

The previous theorem allow us to restrict the number of possible LONG MAIN states. Indeed the possible matrix forms for states satisfying one of condition $M_{1-3}$ on the interval $I_1$ and satisfying conditions (1-7) are only the ones shown in the top rows of Table \ref{tab:MAIN_LONG_table_appendix}. These can be divided into: 
\begin{itemize}
 \item The first 5 matrices in Table \ref{tab:MAIN_LONG_table_appendix} correspond to the states satisfying $M_1$ in $I_1$
 \item The last 4 matrices in Table \ref{tab:MAIN_LONG_table_appendix} correspond to the states satisfying $M_2$ in $I_1$
 \item $\psi$ cannot satisfy $M_3$ in $I_1$ since conditions (1-7) lead to no possible matrix forms
\end{itemize}
Symmetry arguments lead to the obvious symmetrical conjugates for these states, and they complete the case where both units are ON at time $\beta\!=\!T\!R\!+\!T\!D$, and $a\!-\!bs_A^{2+}\!\geq\!\theta$ and $a\!-\!bs_B^{2+}\!\geq\!\theta$. 

\begin {table}[h]
\begin{tabular}{ |c|c|c|c|c|c|c|c|c| } 
 \hline
  \scriptsize $IL_1$ & \scriptsize $IL_2^*$ & \scriptsize $ASDL_1^*$ & \scriptsize $ASL^*$ & \scriptsize $SL^*$ & \scriptsize $IDL_1$ & \scriptsize $IDL_2^*$ & \scriptsize $ASDL_2^*$ & \scriptsize $SDL^*$ \\ \hline
 \ \parbox{1.24cm}{\scriptsize $\begin{array}{c} 
  111 111 \\
  111 111 \\
 \end{array}$}
 & \parbox{1.24cm}{\scriptsize $\begin{array}{c}
  111 110 \\
  111 110 \\
 \end{array}$}
 & \parbox{1.24cm}{\scriptsize $\begin{array}{c}
  111 010 \\
  111 110 \\
 \end{array}$}
 & \parbox{1.24cm}{\scriptsize $\begin{array}{c}
  111 000 \\
  111 110 \\
 \end{array}$}
 & \parbox{1.24cm}{\scriptsize $\begin{array}{c}
  111 000 \\
  111 000 \\
 \end{array}$}
 & \parbox{1.24cm}{\scriptsize $\begin{array}{c}
  111 011 \\
  011 111 \\
 \end{array}$}
 & \parbox{1.24cm}{\scriptsize $\begin{array}{c}
  111 010 \\
  011 110 \\
 \end{array}$}
 & \parbox{1.24cm}{\scriptsize $\begin{array}{c}
  111 000 \\
  011 110 \\
 \end{array}$}
 & \parbox{1.24cm}{\scriptsize $\begin{array}{c}
  111 000 \\
  011 000 \\
 \end{array}$}
 \\ \hline
 $\begin{aligned}[t] 
  D_7^-\!&\geq\!\theta
 \end{aligned}$ & 
 
 $\begin{aligned}[t]
  D_7^-\!&\geq\!\theta
 \end{aligned}$ & 

 $\begin{aligned}[t]
  D_7^-\!&<\!\theta \\
  D_5^-\!&\geq\!\theta \\
  D_3^-\!&\geq\!\theta
 \end{aligned}$ & 

 $\begin{aligned}[t] 
  D_3^-\!&\geq\!0 \\
  D_5^+\!&<\!\theta \\
  D_8^-\!&\geq\!\theta
 \end{aligned}$ &  
 
 $\begin{aligned}[t] 
  D_3^+\!&<\!\theta \\
  D_8^-\!&\geq\!\theta \\
 \end{aligned}$ & 
 
 $\begin{aligned}[t]
  D_3^-\!&\geq\!\theta \\ 
  D_7^-\!&<\!\theta \\ 
  D_5^-\!&\geq\!\theta \\ 
 \end{aligned}$ & 
 
 $\begin{aligned}[t] 
  D_3^-\!&\geq\!\theta \\
  C_7^-\!&<\!\theta \\
  D_5^-\!&\geq\!\theta
 \end{aligned}$ &  
 
 $\begin{aligned}[t] 
  D_3^-\!&\geq\!\theta \\ 
  D_5^-\!&\geq\!\theta \\ 
  D_8^-\!&<\!\theta \\ 
  D_2^-\!&\geq\!\theta
 \end{aligned}$ &  
 
 $\begin{aligned}[t] 
  D_4^-\!&\geq\!\theta \\ 
  D_8^-\!&<\!\theta \\ 
  D_2^-\!&\geq\!\theta \\
  D_3^+\!&<\!\theta 
 \end{aligned}$ 
 \\ \hline

 $\begin{aligned}[t] D_{10}\!&\geq\!\theta \\ \end{aligned}$ & 
 $\begin{aligned}[t] D_{10}\!&<\!\theta \\ C_{10}\!&\geq\!\theta \end{aligned}$ & 
 $\begin{aligned}[t] D_{10}\!&<\!\theta \\ C_{10}\!&\geq\!\theta \end{aligned}$ & 
 $\begin{aligned}[t] C_{10}\!&\geq\!\theta \\ \end{aligned}$ & 
 $\begin{aligned}[t] D_9\!&\geq\!\theta \\ \end{aligned}$ &
 $\begin{aligned}[t] D_{10}\!\geq\!\theta \\ \end{aligned}$ &
 $\begin{aligned}[t] D_{10}\!&<\!\theta \\ C_{10}\!&\geq\!\theta \end{aligned}$ & 
 $\begin{aligned}[t] C_{10}\!&\geq\!\theta \\ \end{aligned}$ & 
 $\begin{aligned}[t] D_{9}\!\geq\!\theta \\ \end{aligned}$ 
 \\ \hline
\end{tabular}
\caption {Matrix form and existence conditions of $2T\!R$-periodic LONG MAIN states (asymmetrical states in *).}
\label{tab:MAIN_LONG_table_appendix}
\end {table}

Next we prove the conditions for the MAIN LONG states shown in the middle row of Table \ref{tab:MAIN_LONG_table_appendix} using equations \ref{syn_values_MAIN_LONG}. For simplicity we write the following conditions using the analogous version of quantities \ref{eq1} in the case of LONG states.
\begin{equation} 
\label{eq2}
 \begin{split}
  D_2^{\pm} \!=\! a\!-\!bM_L^{\pm}\!+\!d, \quad 
  D_3^{\pm} \!=\! \!c\!-\!bN_L^{\pm}, \quad
  D_4^{\pm} \!=\! \!c\!-\!bM_L^{\pm}\!, \quad
  D_5^{\pm} \!=\! a\!-\!bN_L^{\pm}\!+\!d, \quad  \\
  D_6^{\pm} \!=\! a\!-\!bN_L^{\pm}\!+\!c, \quad
  D_7^{\pm} \!=\! \!d\!-\!bN_L^{\pm}\!, \quad
  D_8^{\pm} \!=\! \!d\!-\!bM_L^{\pm}\!, \quad
  D_9 \!=\! \!a-\!bM_L^{+} \quad  
  D_{10} \!=\!a\!-\!bN_L^{+}  
 \end{split}
\end{equation}

Next, we prove the existence conditions for each state separately. 

\begin{itemize}
 \item \textbf{$IL_1$} - This state satisfies conditions $M_1$ during both intervals $I_1$ and $I_2$. Due to the symmetry of the matrix form conditions $M_1$ are equal to conditions $M_2$. Since $w^2\!=\!1$ we have that $s_A^{1\pm}\!=\!s_B^{1\pm}\!=\!N_L^{\pm}$. From this, conditions $M_1$ on interval $I_1$ are $c\!-\!bN_L^-\!\geq\!\theta$ and $D_{7}^-\!=\!d\!-\!bN_L^-\!\geq\!\theta$. Since $c\!\geq\!d$, the condition $D_{7}^-\!\geq\!\theta$ is sufficient to imply $c\!-\!bN_L^-\!\geq\!\theta$. Since $y_A^1\!=\!y_B^1\!=\!1$, the identity $w^1\!=\!H(a\!-\!bN_L^+)\!=\!1$ implies $D_{10}\!\geq\!\theta$. 
 
 \item \textbf{$IL_2$} - Analogously to the previous case, this state satisfies conditions $M_1$ during both intervals $I_1$ and $I_2$. Since $w^1\!=\!1$, $w^2\!=\!0$ and $y_A^2\!=\!y_B^2\!=\!1$ we have $s_A^{1\pm}\!=\!s_B^{1\pm}\!=\!N^{\pm}$ and $s_A^{2\pm}\!=\!s_B^{2\pm}\!=\!N_L^{\pm}$. Since $c\!\geq\!d$ and $N_L^-\!\geq\!N^-$, conditions $M_1$ during both intervals $I_1$ and $I_2$ are simplified to obtain $D_7^-\!=\!d-bN_L^-\!\geq\!\theta$. In addition, $w^1\!=\!1$ and $w^2\!=\!0$ are equvalent to $D_{10}\!<\!\theta$ and $C_{10}\!\geq\!\theta$.
 
 \item \textbf{$ASDL_1$} - We notice that the same arguments used for $IL_2$ lead to $D_{10}\!<\!\theta$ and $C_{10}\!\geq\!\theta$, and to $s_A^{1\pm}\!=\!s_B^{1\pm}\!=\!N^{\pm}$ and $s_A^{2\pm}\!=\!s_B^{2\pm}\!=\!N_L^{\pm}$. This state satisfies conditions $M_1$ during interval $I_1$ and $M_3$ during interval $I_2$. The first set of conditions ($M_1$) lead to $C_7^-\!=\!d\!-\!bN^-\!\geq\!\theta$ (which implies also the second condition in $M_1$, ie $c\!-\!bN^+\!\geq\!\theta$). The second set of conditions ($M_3$) lead to $D_7^-\!=\!d\!-\!N_L^-\!<\!\theta$, $D_5^-\!=\!a\!+\!d\!-\!N_L^-\!\geq\!\theta$ and $D_3^-\!=\!c\!-\!N_L^-\!\geq\!\theta$.
 
 \item \textbf{$ASL$} - This state satisfies conditions $M_1$ during interval $I_1$ and $M_5$ during interval $I_2$. Since $w^1\!=\!1$ we have that $s_A^{2\pm}\!=\!s_B^{2\pm}\!=\!N_L^{\pm}$. Since $w^2\!=\!1$, $y_A^2\!=\!0$ and $y_B^2\!=\!1$ we have that $s_A^{1\pm}\!=\!N^{\pm}$ and $s_B^{1+}\!=\!M_L^{+}$. Conditions leading to $M_5$ during interval $I_2$ are $D_3^-\!=\!c\!-\!bN_L^-\!\geq\!0$ and $D_5^+\!=\!a\!-\!bN_L^+\!+\!d\!<\!\theta$. Conditions leading to $M_1$ during $I_1$ are $c\!-\!bN^-\!\geq\!\theta$, which is implied by $D_3^-\!\geq\!\theta$ (due to $N_L^-\!\geq\!N^-$) and $D_8^-\!=\!d\!-\!M_L^-\!\geq\!\theta$. Finally, as in case $IL_1$, $w^1\!=\!1$ implies $D_{10}\!=\!a\!-\!bN_L^-\!\geq\!\theta$ and $a\!-\!bM_L^-\!\geq\!\theta$. Since $N_L^-\!\geq\!M_L^-$ this second  condition derives from $D_{10}\!\geq\!\theta$ and it can therefore be excluded. Moreover we note that, since $y_A^2\!=\!y_A^2\!=0$, we must have $w^2\!=\!0$. Thus no other conditions are required. 
 
 \item \textbf{$SL$} - This state satisfies conditions $M_1$ during interval $I_1$ and $M_6$ during interval $I_2$. Given that $w^1\!=\!1$ we have $s_A^{2+}\!=\!s_B^{2+}\!=\!N_L^{+}$. Condition $M_6$ requires $D_3^+\!=\!c\!-\!bN_L^+\!<\!\theta$ (since it implies $d\!-\!bN_L^+\!<\!\theta$). Since $w^2\!=\!0$ and $y_A^2\!=\!y_B^2\!=\!0$ we have that $s_A^{1-}\!=\!s_B^{1-}\!=\!M_L^{-}$. Condition $M_1$ requires $D_8^-\!=\!d\!-\!bM_L^+\!\geq\!\theta$ (since it implies $c\!-\!bM_L^+\!\geq\!\theta$). Condition $D_9\!=\!d\!-\!bM_L^-\!\geq\!\theta$ guarantees that $w^1\!=\!1$. We note that, since $y_A^2\!=
\!y_A^2\!=0$, we must have $w^2\!=\!0$ with no extra conditions. 

 \item \textbf{$IDL_1$} - This state satisfies conditions $M_2$ during the interval $I_1$ and conditions $M_3$ during the interval $I_2$. Since this state is symmetrical $M_2$ and $M_3$ are give equal conditions. Analogously to the case $IL_1$ we have that $s_A^{1-}\!=\!s_B^{1-}\!=\!N_L^{-}$. Thus conditions for $M_2$ are $D_3^-\!=\!c\!-\!bN_L^-\!\geq\!\theta$, $D_7^-\!=\!d\!-\!bN_L^-\!<\!\theta$ and $D_5^-\!=\!a\!-\!bN_L^-\!+\!d\!\geq\!\theta$. Condition $w^1\!=\!1$ leads to $D_{10}=a\!-\!bN_L^+\!\geq\!\theta$. 
 
 \item \textbf{$IDL_2$} - Analogously to case $IL_2$ we obtain $s_A^{1\pm}\!=\!s_B^{1\pm}\!=\!N^{\pm}$ and $s_A^{2\pm}\!=\!s_B^{2\pm}\!=\!N_L^{\pm}$. This state ($IDL_2$) satisfies conditions $M_2$ on interval $I_1$ and $M_3$ on interval $I_2$. This leads to $D_3^-\!=\!c\!-\!bN_L^-\!\geq\!\theta$ (which implies $c\!-\!bN^-\!\geq\!\theta$), $C_7^-\!=\!d\!-\!bN^-\!<\!\theta$ (which implies $d\!-\!bN_L^-\!<\!\theta$) and $D_5^-\!=\!a\!+\!d\!-\!bN_L^-\!\geq\!\theta$ (which implies $d\!-\!bN^-\!<\!\theta$, hence $y_B^2\!=\!1$). Similar arguments to the ones shown in case $IL_2$ lead to $D_{10}\!<\!\theta$ and $C_{10}\!\geq\!\theta$
 
 \item \textbf{$ASDL_2$} - As in case $ASL$ we have that $s_A^{2\pm}\!=\!s_B^{2\pm}\!=\!N_L^{\pm}$, $s_A^{1-}\!=\!N^{-}$ and $s_B^{1-}\!=\!M_L^{-}$. This state satisfies conditions $M_2$ on interval $I_1$ and $M_5$ on interval $I_2$. For the same arguments as case $IDL_2$ we must have $D_3^-\!=\!c\!-\!bN_L^-\!\geq\!\theta$. Completing the conditions on $I_1$ requires $D_8^-\!=\!d\!-\!bM_L^+\!<\!\theta$ and $D_2^-\!=\!a\!-\!bM_L^-\!+\!d\!\geq\!\theta$. Completing the conditions on $I_2$ requires $D_5^-\!=\!a\!+\!d\!-\!bM_L^-\!\geq\!\theta$. As in case $ASL$ we also require $D_{10}\!\geq\!\theta$.  
 
 \item \textbf{$SDL$} - Analogously to case $SL$ we have $s_A^{2+}\!=\!s_B^{2+}\!=\!N_L^{+}$, $s_A^{1-}\!=\!s_B^{1-}\!=\!M_L^{-}$ and $D_9\!\geq\!\theta$. This state satisfies conditions $M_2$ during interval $I_1$ and $M_6$ during interval $I_2$. As shown in $SL$, conditions $M_6$ on interval $I_2$ implies $D_3^+\!<\!\theta$. Instead, conditions $M_3$ on interval $I_1$ are $D_4^-\!=\!c\!-\!bM_L^-\!\geq\!\theta$, $D_8^-\!=\!d\!-\!bM_L^-\!<\!\theta$ and $D_2^-\!=\!a\!-\!bM_L^-\!+\!d\!\geq\!\theta$.	
\end{itemize}
This conlcudes the proof of the existence conditions for all the LONG MAIN states shown in Table \ref{tab:MAIN_LONG_table_appendix}.

\subsection{Analysis of $2T\!R$-periodic $LM|SC$,  $LC|SC$, $LC|LC$ and $LC|SM$ states} \label{appendix3}
As shown in the Section \ref{sec_2TR_states}, $2T\!R$-periodic states can be SHORT MAIN ($SM$), SHORT CONNECT ($SC$), LONG MAIN ($LM$) or LONG CONNECT ($LC$) during each interval $I_1$ and $I_2$. We define $X|Y$ the set of states satisfying condition X during $I_1$ and Y during $I_2$, where $X,Y \in \{ SM,SC,LM,LC \}$. In Section \ref{sec_2TR_states} we have  the existence conditions of all possible states in some of these sets. More precisely:
\begin{itemize}
 \item The analysis of $SM|SM$ is summarised in Table \ref{tab:MAIN_table}
 \item The analysis of $SC|SM$, $SM|SC$ and $SC|SC$ is summarized in Table \ref{tab:CONNECT_table_appendix}
 \item The analysis of $LM|LM$, $SM|LM$ and $LM|SM$ is summarized in Table \ref{tab:MAIN_LONG_table_appendix}
\end{itemize}
In this section we study the remaining combinations of $X|Y$ sets. For all such sets at least one between $X$ and $Y$ are of the LONG type (ie $LC$ or $LM$). Due to the model's symmetry, we can limit our analysis to the sets where $X$ is LONG, i.e. for LONG states during $I_1$ ($LC|Z$ and $LM|Z$, where $Z \in \{ SM,SC,LM,LC \}$). Indeed, states $Z|LC$ and $Z|LM$ can be obtained respectively from states in $LC|Z$ and $LM|Z$ and by applying the symmetry principles. 

The next theorem shows that the matrix form for these states allow us to determining all states that can exist in the parameter space. Indeed the entries of these matrices must satisfy properties (1-6) below. 

\begin{theorem} [Conditions for LONG states in $I_1$] \label{remainig_thm}
Any LONG state in $I_1$ satisfies:
\begin{enumerate}
 \item If $w^2\!=\!0 \Rightarrow x_A^2\!\leq\!x_B^1, x_B^2\!\leq\!x_A^1, y_A^2\!\leq\!y_B^1$ and $y_B^2\!\leq\!y_A^1$ 
 \item If $w^2\!=\!0, y_A^2\!=\!y_B^2\!=\!1 \Rightarrow x_A^1\!\geq\!x_B^1$
 \item If $w^2\!=\!1 \Rightarrow x_A^1\!\geq\!x_B^1$
 \item If $w^2\!=\!1$ and $x_A^2\!=\!1$ or $x_B^2\!=\!1 \Rightarrow x_A^1\!\geq\!x_B^2, x_B^1\!\geq\!x_A^2, y_A^1\!\geq\!y_B^2$ and $y_B^1\!\geq\!y_A^2$
 \item If $w^2\!=\!1$ and $x_A^1\!=\!1$ or $x_B^1\!=\!1 \Rightarrow x_A^2\!\geq\!x_B^1, x_B^2\!\geq\!x_A^1, y_A^2\!\geq\!y_B^1$ and $y_B^2\!\geq\!y_A^1$
 \item If $V_2$ has all zero entries $\Rightarrow x_A^1\!\geq\!x_B^1$ 
\end{enumerate}
\end{theorem}

\begin{proof}
Due to Lemma \ref{lem:LONG_conditions} for any LONG state in $I_1$ both units turn are ON at time $T\!D$, and turn OFF at time $t^*\!+\!D$, for some $t^* \in [0,T\!D]$. Consequently both delayed synaptic variables exponentially decay during the interval $I_2$ starting from $t^*\!+\!2D$. This leads to $s_A^{2-}\!=\!s_B^{2-}\!=\!e^{-(T\!R\!-t^*\!-\!2D)/\tau_i}$. We notice that, since $t^*\!\geq\!0$ we have 
\begin{equation}
s_A^{2-}\!=\!s_B^{2-}\!\geq\!N_L^-
\label{long_I1_cond1}
\end{equation}
If $w^2\!=\!0$ (the hypothesis in 1.)\ the state is SHORT in $I_2$ (both units turn/are OFF at time $T\!D$). This means we can apply identies \ref{syn_values_MAIN_LONG} on the interval $I_1$ and obtain 
\begin{equation}
s_A^{1-}\!=\!s_B^{1-}\!\leq\!N^-
\label{long_I1_cond2}
\end{equation}
Inequalities \ref{long_I1_cond1} and \ref{long_I1_cond2} thus imply $s_A^{1-}\!\leq\!s_B^{2-}$ and $s_B^{1-}\!\leq\!s_A^{2-}$. By definition $x_A^1\!=\!H(c\!-\!bs_B^{1-})$ and $x_B^2\!=\!H(c\!-\!bs_A^{2-})$. Thus we have $x_B^2\!\leq\!x_A^1$ (analogously we have $x_A^2\!\leq\!x_B^1$). Moreover, $y_B^2\!=\!H(ax_A^2\!+\!c\!-\!bs_A^{2-})\leq\!H(ax_B^1\!+\!c\!-\!bs_B^{1-})\!=\!y_A^1$ (and $y_A^2\!\leq\!y_B^1$), proving 1. 

One of the hypothesis of 2. is $w^2\!=\!0$. Thus we can apply identies \ref{syn_values_MAIN_LONG} analogously to the previous case. Since $y_A^2\!=\!y_B^2\!=\!1$, these identies lead to $s_A^{1-}\!=\!s_B^{1-}\!=\!N^-$. Condition $c\!\geq\!d$ guarantees that $x_A^1\!=\!H(c\!-\!bN^-)\!\geq\!H(d\!-\!bN^-)=x_B^1$, thus proving 2. 

We proceed by proving 3. Condition $w^2\!=\!1$ guarantees the corresponding states to be LONG in $I_2$. Due to the $2T\!R$ periodicity we have $s_A^{1-}\!=\!s_B^{1-}\!=\!e^{-(T\!R\!-\!s^*\!-\!2D)/\tau_i}$, for some $s^* \in [0,T\!D]$. This and $d\!\leq\!c$ imply $x_A^1\!=\!H(c\!-\!bs_B^{1-})\!\geq\!H(d\!-\!bs_A^{1-})\!=\!x_B^1$, which proves 3. 

Assuming the hypothesis of 4 (5) at least one unit turns ON at time $T\!D$ ($0$). Lemma \ref{lem:LONG_conditions} thus implies $s^*\!=\!0$ ($t^*\!=\!0$). Therefore we have that $s_A^{1-}\!=\!s_B^{1-}\!=\!N_L^-$ ($s_A^{2-}\!=\!s_B^{2-}\!=\!N_L^-$), which implies $s_A^{1-}\!\leq\!s_B^{2-}$ and $s_B^{1-}\!\leq\!s_A^{2-}$ ($s_A^{2-}\!\leq\!s_B^{1-}$ and $s_B^{2-}\!\leq\!s_A^{1-}$). Using a proof similar to 1 we conclude 4 (5). 

Assuming the hypothesis of 6. both units are OFF in $I_2$. Therefore, both delayed synaptic variables decay monotonically starting from time $t\!=\!t^*\!+\!2D$ until time $t\!=\!2T\!R$. For the $2T\!R$ periodicity we thus have $s_A^{1-}\!=\!s_B^{1-}\!=\!e^{-(2T\!R\!-\!t^*\!-\!2D)/\tau_i}$. This, $d\!\leq\!c$ and the definition of $x_A^1$ and $x_B^1$ yield 6. 
\end{proof}

We applied Theorem \ref{remainig_thm} to investigate the possible combinations of states in all remaining sets $LC|Z$ and $LM|Z$, where $Z \in \{ SM,SC,LM,LC \}$. We subdivide this analysis in the following cases.

\textbf{Sets} $\mathbf{LM|SC}$ \textbf{and} $\mathbf{LM|LC}$ - Any state $\psi$ in either of these two sets is LONG and MAIN in $I_1$, and CONNECT in $I_2$. The LONG condition in $I_1$ implies that (a) both units are ON at time $\beta\!=\!T\!D$, and (b) $a\!-\!bs_A^{1+}\!\geq\!\theta$ and $a\!-\!bs_B^{1+}\!\geq\!\theta$.\ Condition (a) implies that $V_1$ must satisfy one of $M_{1-3}$ during the interval $I_1$.\ From (b) we obtain $w^1\!=\!1$. As shown in the proof of property 5.\ above, we have that $s_B^{2 \pm}\!=\!s_A^{2 \pm}\!=\!N_L^{\pm}$.\ The CONNECT condition in $I_2$ implies that $\psi$ must satisfy one of conditions $C_{1-5}$.\ However, since $d\!\leq\!c$, we must have $x_A^2\!=\!H(d\!-\!bN_L^-)\!\leq\!H(c\!-\!bN_L^-)\!\leq\!x_B^2$ and $z_A^2\!=\!H(a\!+\!d\!-\!bN_L^-)\!\leq\!H(a\!+\!c\!-\!bN_L^-)\!\leq\!z_B^2$. This excluded conditions the states satifying conditions $C_1$ and $C_4$ in $I_2$. Property 2.\ above guarantees that $LM|SC$ states satisfying condition $M_3$ in $I_1$ and $C_2$ or $C_5$ in $I_2$ cannot exist. The remaining set of $LM|SC$ states can exist in the parameter space and their name and matrix are given in in the first two rows of Table \ref{tab:LM_SC_table_main}. We numerically verified their existence by finding a parameter set for which they are stable using linear programming on their sets of existing conditions and by simulating their dynamics. For states in $LM|LC$ we notice that, since they are LONG in $I_2$, they cannot satisfy condition $C_3$ in this interval (both units would otherwise be OFF at time $T\!R\!+\!T\!D$). Due to properties 3.\ and 5.\ above none of remaining states (the ones satisfying conditions $C_2$ and $C_5$) can exist. Therefore, no $LM|LC$ state can exist. 
 
\begin {table}[h]
\begin{tabular}{ |c|c|c|c|c|c|c| } 
 \hline
  \footnotesize $AScIL$ & \footnotesize $ScASL^*$ & \footnotesize $ScIL^*$ & \footnotesize $AScIDL^*$ & \footnotesize $ScASDL^*$ & \footnotesize $ScIDL$ & \footnotesize $ScASDL_2^*$ \\ \hline
 \ \parbox{1.7cm}{\footnotesize $\begin{array}{c} 
  1111 0010 \\
  1111 1110 \\
 \end{array}$}
 & \parbox{1.7cm}{\footnotesize $\begin{array}{c}
  1111 0000 \\
  1111 0010 \\
 \end{array}$}
 & \parbox{1.7cm}{\footnotesize $\begin{array}{c}
  1111 0010 \\
  1111 0010 \\
 \end{array}$}
 & \parbox{1.7cm}{\footnotesize $\begin{array}{c} 
  1111 0010 \\
  0111 1110 \\
 \end{array}$}
 & \parbox{1.7cm}{\footnotesize $\begin{array}{c}
  1111 0000 \\
  0111 0010 \\
 \end{array}$}
 & \parbox{1.7cm}{\footnotesize $\begin{array}{c}
  1111 0010 \\
  0111 0010 \\
 \end{array}$}
 & \parbox{1.7cm}{\footnotesize $\begin{array}{c}
  011 0000 \\
  111 0010 \\
 \end{array}$}
 
 \\ \hline
 $\begin{aligned}[t] 
  D_3^-\!&\geq\!\theta \\
  C_7^-\!&\geq\!\theta \\
  D_5^-\!&<\!\theta \\
  D_2^+\!&\geq\!\theta
 \end{aligned}$ & 
 
 $\begin{aligned}[t]
  C_3^-\!&\geq\!\theta \\
  D_5^+\!&<\!\theta \\
  D_3^-\!&<\!\theta \\
  D_3^+\!&\geq\!\theta \\
  D_8^-\!&\geq\!\theta
 \end{aligned}$ & 

 $\begin{aligned}[t]
  C_3^-\!&\geq\!\theta \\
  D_5^+\!&\geq\!\theta \\
  D_3^-\!&<\!\theta \\
  D_3^+\!&\geq\!\theta \\
  D_7^-\!&\geq\!\theta
 \end{aligned}$ & 

 $\begin{aligned}[t] 
  D_3^-\!&\geq\!\theta \\
  D_5^-\!&<\!\theta \\
  D_5^+\!&\geq\!\theta \\
  C_5^-\!&\geq\!\theta \\
  D_7^-\!&<\!\theta
 \end{aligned}$ &  
 
 $\begin{aligned}[t] 
  C_3^-\!&\geq\!\theta \\
  D_5^+\!&<\!\theta \\
  D_3^-\!&<\!\theta \\
  D_3^+\!&\geq\!\theta \\
  D_8^-\!&<\!\theta \\
  D_2^-\!&\geq\!\theta 
 \end{aligned}$ & 
 
 $\begin{aligned}[t]
  C_3^-\!&\geq\!\theta \\
  D_5^+\!&\geq\!\theta \\
  D_3^-\!&<\!\theta \\
  D_3^+\!&\geq\!\theta \\
  D_7^-\!&<\!\theta \\
  C_5^-\!&\geq\!\theta 
 \end{aligned}$ & 
 
 $\begin{aligned}[t] 
  C_3^-\!&<\!\theta \\
  C_6^-\!&\geq\!\theta \\
  D_3^+\!&\geq\!\theta \\
  D_5^+\!&<\!\theta \\
  D_8^-\!&\geq\!\theta 
 \end{aligned}$ 
 \\ \hline

 $\begin{aligned}[t] D_{10}\!&<\!\theta \\ C_{10}\!&\geq\!\theta \end{aligned}$ & 
 $\begin{aligned}[t] C_{10}\!&\geq\!\theta \end{aligned}$ & 
 $\begin{aligned}[t] D_{10}\!&<\!\theta \\ C_{10}\!&\geq\!\theta \end{aligned}$ & 
 $\begin{aligned}[t] D_{10}\!&<\!\theta \\ C_{10}\!&\geq\!\theta \end{aligned}$ & 
 $\begin{aligned}[t] C_{10}\!&\geq\!\theta \end{aligned}$ & 
 $\begin{aligned}[t] D_{10}\!&<\!\theta \\ C_{10}\!&\geq\!\theta \end{aligned}$ & 
 $\begin{aligned}[t] C_{10}\!&\geq\!\theta \end{aligned}$
 \\ \hline
\end{tabular}
\caption {Matrix form and existence conditions of $2T\!R$-periodic $LM|SC$ states (asymmetrical states in *).}
\label{tab:LM_SC_table_main}
\end {table}

\textbf{Sets} $\mathbf{LC|SC}$, $\mathbf{LC|LC}$ \textbf{and} $\mathbf{LC|SM}$ - Any state in either of these two sets is LONG and CONNECT in $I_1$. The LONG condition implies that both units are ON at time $\beta\!=\!T\!D$, thus excluding CONNECT conditions $C_3$ or $C_4$ in $I_1$. Furthermore, as shown in the case of $LM|LC$ (previous case), this LONG condition also excludes CONNECT conditions $C_1$ and $C_4$ for $LC|SC$ and $LC|LC$ states in $I_2$, and conditions $M_2$ and $M_4$ for $LC|SM$ states in $I_2$. For states in $LC|LC$ we notice that, since they are LONG in $I_2$, they cannot satisfy condition $C_3$ in this interval (for an analogue reason of case $LM|LC$). Of the remaining states, the ones described by the following matrix forms cannot exist:
 \begin{equation}
	\begin{matrix} [cc]
		1111 & 0000 \\
		0011 & 0010
	\end{matrix}
	\quad \quad \mbox{and} \quad \quad
	\begin{matrix} [cc]
		1111 & 0000 \\
		0011 & 1110
	\end{matrix}
	\nonumber
 \end{equation}
 Indeed entries $w^1\!=\!1$ and $y_B^1\!=\!0$ imply respectively $a\!-\!bN^+\!\geq\!\theta$ and $a\!-\!M_L^-\!+\!d\!<\!\theta$. These two conditions imply $d\!<\!be^{(D\!-\!T\!R)/\tau_i}(e^{(D\!-\!T\!R)/\tau_i}\!-\!1)\!<\!0$. This is absurd since by hypothesis we must have $T\!R\!>\!D$ and $d\!\geq\!0$. Finally the application of properties 1-6 above reduces the number of possible states.  The remaining set of $LM|SC$ states can exist in the parameter space and their name and matrix are given in the first two rows of Table \ref{tab:remaining_table_main}. 
 
  \begin {table}[h]
\begin{tabular}{ |c|c|c|c|c|c|c|} 
 \hline
  \footnotesize $ScASDL_3^*$ & \footnotesize $APcIDL^*$ & \footnotesize $ScSDL^*$ & \footnotesize $APcIL$ & \footnotesize $ScASDL_4^*$ & \footnotesize $ScASL_2^*$ & \footnotesize $ZcIL^*$ \\ \hline
 \ \parbox{1.7cm}{\footnotesize $\begin{array}{c} 
  1111 0010 \\
  0011 0010 \\
 \end{array}$}
 & \parbox{1.7cm}{\footnotesize $\begin{array}{c}
  1111 0010 \\
  0011 1110 \\
 \end{array}$}
 & \parbox{1.7cm}{\footnotesize $\begin{array}{c}
  1111 0000 \\
  0011 0000 \\
 \end{array}$}
 & \parbox{1.7cm}{\footnotesize $\begin{array}{c} 
  1111 0011 \\
  0011 1111 \\
 \end{array}$}
 & \parbox{1.7cm}{\footnotesize $\begin{array}{c}
  0011 0000 \\
  1111 0010 \\
 \end{array}$}
 & \parbox{1.7cm}{\footnotesize $\begin{array}{c}
  0011 0000 \\
  0011 0010 \\
 \end{array}$}
 & \parbox{1.7cm}{\footnotesize $\begin{array}{c}
  0011 0010 \\
  0011 0010 \\
 \end{array}$} 
 
 \\ \hline
 $\begin{aligned}[t] 
  C_3^-\!&\geq\!\theta \\
  D_3^-\!&<\!\theta \\
  D_3^+\!&\geq\!\theta \\
  D_5^-\!&<\!\theta \\
  D_5^+\!&\geq\!\theta
 \end{aligned}$ & 
 
 $\begin{aligned}[t]
  D_3^-\!&\geq\!\theta \\
  C_5^-\!&<\!\theta \\
  D_5^+\!&\geq\!\theta 
 \end{aligned}$ & 

 $\begin{aligned}[t]
  D_3^+\!&<\!\theta \\
  D_4^-\!&\geq\!\theta \\
  D_2^-\!&<\!\theta \\
  D_2^+\!&\geq\!\theta 
 \end{aligned}$ & 

 $\begin{aligned}[t] 
  D_3^-\!&\geq\!\theta \\
  D_5^-\!&<\!\theta \\
  D_5^+\!&\geq\!\theta 
 \end{aligned}$ &  
 
 $\begin{aligned}[t] 
  C_6^-\!&<\!\theta \\
  D_6^+\!&\geq\!\theta \\
  D_5^+\!&<\!\theta \\
  D_8^-\!&\geq\!\theta \\
 \end{aligned}$ & 
 
 \parbox{1.7cm}{ See \ref{ScASL_2}} &
 
 \parbox{1.7cm}{ See \ref{ZcIL}} 
 \\ \hline

 $\begin{aligned}[t] D_{10}\!&<\!\theta \\ C_{10}\!&\geq\!\theta \end{aligned}$ & 
 $\begin{aligned}[t] D_{10}\!&<\!\theta \\ C_{10}\!&\geq\!\theta \end{aligned}$ & 
 $\begin{aligned}[t] D_{9}\!&\geq\!\theta \end{aligned}$ & 
 $\begin{aligned}[t] D_{10}\!&\geq\!\theta \end{aligned}$ & 
 $\begin{aligned}[t] C_{10}\!&\geq\!\theta \end{aligned}$ & 
  & 
 
 \\ \hline
\end{tabular}
\caption {Matrix form and existence conditions of $2T\!R$-periodic $LC|SC$, $LC|LC$ and $LC|SM$ states (asymmetrical states in *).}
\label{tab:remaining_table_main}
\end {table}

 The last two rows of Tables \ref{tab:LM_SC_table_main} and \ref{tab:remaining_table_main} show the conditions of existence of the corresponding $LM|SC$,  $LC|SC$, $LC|LC$ and $LC|SM$ states. Determining these is straightforward in most cases. Indeed, it requires using formulas \ref{syn_values_CONNECT_SHORT} and \ref{syn_values_MAIN_LONG} on the definition of the entries of each matrix form, and application of simplifications, analogously to the previous considered cases, except for $ScASL_2$ and $ZcIL$ (see Table \ref{tab:remaining_table_main}). These two need special attention, because they satisfy property $C_5$ in $I_1$, we cannot apply the formulas \ref{syn_values_CONNECT_SHORT} and \ref{syn_values_MAIN_LONG}. For $ZcIL$ the A unit turns ON before the B units in $I_1$ ($t^*<s^*$), because both synaptic variables $s_A$ and $s_B$ evolve equally during in this interval (on the fast time scale) and the total input to the A unit is greater than the one to the B unit at time $t^*$ , i.e.\ $c\!-\!bs_B(t^*)\!\leq\!d\!-\!bs_A(t^*)$. For $ScASL_2$ the two synaptic variables evolve differently on $I_1$, which may lead to $t^*\!<\!s^*$ or $t^*\!\geq\!s^*$. Later we will show that case $t^*\!\geq\!s^*$ cannot exist. Lastly there are three degenerate states that exist only under $\tau\!=\!0$, which cannot be numerically simulated. These states conclude all set of existing $2T\!R$-periodic states in the system under the case $T\!R\!\leq\!T\!D\!+\!D$ and $D\!\geq\!T\!D$. 

We proceed by describing the existence conditions for $ScASL_2$ for $t^*\!<\!s^*$ and state $ZcIL$. 
\begin{itemize}
 \item Case $ScASL_2$ for $t^*\!<\!s^*$ - This state satisfies conditions $C_5$ on $I_1$ and $C_3$ on $I_2$. Since $y_B^2\!=\!1$ and $w^2\!=\!0$ the B unit turns OFF at time $T\!R\!+\!T\!D$. Due to Lemma \ref{lem:LONG_conditions} the synaptic variable $s_B(t)$ exponentially decays starting from time $T\!R\!+\!T\!D\!+\!D$ and due to the $2\!T\!R$-periodicity we must have $s_B(t)\!=\!e^{-(T\!R\!+\!t\!-\!T\!D\!-\!D)/\tau_i}$, for $t \in [0,T\!D]$. From this we obtain $s_B(0)\!=\!N^-$ and $s_B(T\!D)\!=\!N^+$. Condition $C_5$ on $I_1$ with $t^*<s^*$ requires $C_3^-\!=\!c\!-\!bN^-\!<\!\theta$ and $c\!-\!bN^+\!\geq\!\theta$. The turning ON time for the A unit in $I_1$ is therefore given by 
 $$ t^* = s_B^{-1} ((c-\theta)/b)) = T\!R-T\!D-D-\tau_i \log((c-b)/\theta). $$ 
 From Lemma \ref{lem:LONG_conditions} and from $t^*\!<\!s^*$ we obtain that both units instantaneously turn OFF at time $t^*\!-\!2D$. Thus the synaptic variable $s_A(t)$ and $s_B(t)$ exponentially decay following the same dynamics on the slow time scale starting from time $t^*\!+\!2D$. This leads to $s_A(t)\!=\!s_B(t)\!=\!e^{-(t\!-\!t^*\!-\!2D)/\tau_i}$, for $t \in [T\!R,T\!R\!+\!T\!D]$. Moreover, since the A unit is OFF in $I_2$ and due to the $2\!T\!R$-periodicity we have $s_A(t)\!=\!e^{-(2T\!R\!+\!t\!-\!t^*\!-\!2D)/\tau_i}$, for $t \in [0,T\!D]$. These properties yield $s_A^{2+}\!=\!s_A(T\!R\!+\!T\!D)\!=\!e^{(D\!-\!2T\!D)/\tau_i}(c\!-\!\theta)/b$ and $s_A^{1+}\!=\!s_A(T\!D)\!=\!e^{(D\!-\!2T\!D\!-\!T\!R)/\tau_i}(c\!-\!\theta)/b$. To complete the conditions $C_5$ on $I_1$ we need to guarantee that the B unit turns ON at some time $s^*\!\geq\!t^* \in [0,T\!D]$. These are equivalent to $d\!-\!bs_A(t^*)\!\!<\!\theta$ and $a\!-\!bs_A^{1+}\!+d\!\geq\!\theta$, which can respectively be rewritten as $d\!-\!be^{2(D\!-\!T\!R)/\tau_i}\!<\!\theta$ and $a\!-\!Lc\!+\!d\!\geq\!(1\!-\!L)\theta$, where $L\!=\!e^{(D\!-\!2T\!D\!-\!T\!R)/\tau_i}$. Condition $C_3$ on $I_2$ requires $c\!-\!bs_A^{2-}\!<\!\theta$ and $c\!-\!bs_A^{2+}\!\geq\!\theta$. This first of these conditions is not necessary, since it is implied by the already existing condition $C_3^-\!<\!\theta$ (since $s_A^{2+}\!\geq\!N^-$). The second is equivalent to $(c\!-\!\theta)(1\!-\!K)\!\geq\!0$, with $K\!=\!e^{(D\!-\!2T\!D)/\tau_i}$, which occurs if and only if $D\!\leq\!2T\!D$ (since $c\!\geq\!\theta$). To complete condition $C_3$ we need to guarantee that the A unit stays OFF in $I_2$, ie that $a\!-\!bs_B^{2+}\!+d\!<\!\theta$, which is equivalent to $a\!-\!Kc\!+\!d\!<\!(1\!-\!K)\theta$. Finally, the last condition derives from $w^1\!=\!1$ ($w^2\!=\!0$ is authomatically guaranteed since the A is OFF at time $T\!R\!+\!T\!D$), ie $C_{10}\!=\!a\!-\!bN^+\!\geq\!\theta$. This guarantees also $a\!-\!bs_A^{1+}\!=\!a\!-\!bM_L^+\!\geq\!\theta$, since $M_L^+\!\leq\!N^+$. Thus in summary the list of conditions for this state is:
 \begin{equation}
  \begin{aligned}[t]
   C_3^-\!&<\!\theta \\
   C_{10}\!&\geq\!\theta \\
   D\!&\leq\!2T\!D \\
   d\!&<\!\theta\!-\!be^{2(D\!-\!T\!R)/\tau_i} \\
   a\!-\!Lc\!+\!d\!&\geq\!(1\!-\!L)\theta \\
   a\!-\!Kc\!+\!d\!&<\!(1\!-\!K)\theta
  \end{aligned}
 \label{ScASL_2}
 \end{equation}
 \item Case $ZcIL$. This state satisfies conditions $C_5$ during both intervals $I_1$ and $I_2$. Both units turn OFF at time $T\!R\!+\!T\!D$. Lemma \ref{lem:LONG_conditions} implies that both synaptic variables exponentially decay starting from time $T\!R\!+\!T\!D\!+\!D$ and due to the $2\!T\!R$-periodicity we must have $s_A(t)\!=\!s_B(t)\!=\!e^{-(T\!R\!+\!t\!-\!T\!D\!-\!D)/\tau_i}$, for $t \in [0,T\!D]$. From the A unit turns ON before the B unit in interval $I_1$, precisely at time $t^*$, and both units turn OFF at time $t^*\!+\!D$ for lemma \ref{lem:LONG_conditions}. Thus the delayed synaptic variables exponentially decay from time $t^*\!+\!2\!D$ and we have $s_A(t)\!=\!s_B(t)\!=\!e^{-(t\!-\!t^*\!-\!2D)/\tau_i}$, for $t \in [T\!R,T\!R\!+\!T\!D]$. Thus both variables evolve equally (on the slow time scale) respectively on $I_1$ and on $I_2$. Although condition $C_5$ on both intervals could lead to potentially 4 cases, we only have one case to consider, the A (B) unit turns ON before the B (A) unit in interval $I_1$ ($I_2$). Analogously to the case $ScASL_2$, condition $C_5$ on $I_1$ requires $C_3^-\!=\!c\!-\!bN^-\!<\!\theta$ and $c\!-\!bN^+\!\geq\!\theta$, and $t^*$ is given by
 $$ t^* = s_B^{-1} ((c-\theta)/b)) = T\!R-T\!D-D-\tau_i \log((c-b)/\theta). $$ 
 As in case $ScASL_2$ condition $C_3$ on $I_2$ requires $D\!\leq\!2T\!D$. To complete the conditions $C_5$ we require $a\!-\!bs_B^{2+}\!+\!d\!\geq\!\theta$, which is equivalent to $a\!-\!Kc\!+\!d\!\geq\!\theta(1\!-\!K)$. Lastly, we need to guarantee $w^1\!=\!1$ and $w^2\!=\!0$, which are equivalent respectively to $C_{10}\!\geq\!\theta$ and $a\!-\!Kc\!<\!\theta(1\!-\!K)$ (ie $a\!-\!bs_B^{2+}\!<\!\theta$). Thus in summary the list of conditions for this states are:
 \begin{equation}
  \begin{aligned}[t]
   C_3^-\!&<\!\theta \\
   C_{10}\!&\geq\!\theta \\
   D\!&\leq\!2T\!D \\
   a\!-\!Kc\!+\!d\!&\geq\!(1\!-\!K)\theta \\
   a\!-\!Kc\!&<\!(1\!-\!K)\theta
  \end{aligned}
 \label{ZcIL}
 \end{equation}
\end{itemize}

Lastly, we show that the following three states may exist only if $\tau\!=\!0$ (degenerate cases). These states complete all the existing states after application of conditions \ref{remainig_thm}. This finally concludes the existence conditions for all $2T\!R$-periodic states in the system. 

 \begin{equation}
	ScASL_2 \!=\! \begin{bmatrix} [cc]
		0011 & 0000 \\
		0011 & 0010
	\end{bmatrix} 
	\mbox{ for } t^* \geq s^* ,
	ZcIL_2 \!=\! \begin{bmatrix} [cc]
		0011 & 0011 \\
		0011 & 0011
	\end{bmatrix}
	\mbox{ and }
	ZcSL \!=\! \begin{bmatrix} [cc]
		0011 & 0000 \\
		0011 & 0000
	\end{bmatrix}.
	\nonumber
 \end{equation}
 We show that these three states cannot exist unless $\tau\!=\!0$ (degenerate for the model). Firstly we note that each case satisfies condition $C_5$ on interval $I_1$, so that A turns ON at time $t^*$ and B turns ON at time $s^*$, for some $t^*, s^* \in [0,T\!D]$. Next we divide the proof for the three cases above:
 \begin{enumerate}
  \item $ZcSL$ - This state satisfies condition $C_5$ on interval $I_1$. That A turns ON at time $t^*$ and B turns ON at time $s^*$, for some $t^*, s^* \in I_1$. Since both units turn OFF instantaneously and at the same time in $\mathbb{R}\!-\!I$, both synaptic variables evolve equally (on the slow time scale) in $I_1$. Therefore we have that $t^*\!\geq\!s^*$. Let us rename $t_1\!=\!t^*$. On the fast time scale $r$ the variable $s_A(t)$ converges to $1$ at time $t_1$ following 
  \begin{eqnarray*}
  s_A(r)'&=&(1-s_A(r)) \\
  s_A(0)&=&s_A(t_1)
  \end{eqnarray*}
  where $'$ is the derivative with respect to the fast time scale $r$. The analytic solution is given by $s_A(r)\!=\!1\!-\!(1\!-\!s_A(t_1))e^{-r}$. Therefore, this equation describes the (fast) evolution of the delayed synaptic variable $s_A(t\!-\!D)$ at time $t\!=\!t_1\!+\!D$. At this time the B unit instantaneously turn OFF, since $a\!-\!bs_A(t\!-\!D) \!\rightarrow\! a\!-\!b\!<\!\theta$ for hypothesis \ref{U2}. We can use the equation for $s_A(r)$ and derive the precise time when $u_B$ turns OFF. Since $a\!-\!bs_A(t_1)\!\geq\!\theta$ and $a\!-\!b\!<\!\theta$  there $\exists s^* \in [s_A(t_1),1]$ for which $a\!-\!bs^*\!=\!\theta$. Given the evolution of $s_A$, the time when B unit turns OFF is precisely $r^*\!=\!r^*(t_1)\!=\!\log((1\!-\!s_A(t_1))/(1\!-\!s^*))$. The latter equality highlights the dependence on $t_1$. By adding the delay and returning to the normal time scale the B unit turns OFF at time $t_1\!+\!D\!+\!\delta(t_1)$, where $\delta(t_1)\!=\!\tau r^*$. Since the dynamics of delayed synaptic variable $s_B(t\!-\!D)$ is dictated by the B unit activity, it starts to exponentially decay at time $t_1\!+\!2D\!+\!\delta(t_1)$. Thus it evolves according to $s_B(t\!-\!D)\!=\!e^{-(t\!-t_1\!-\!2D\!-\!\delta(t_1))/\tau_i}$, for $t \in I_3\!=\![2T\!R,2T\!R\!+\!T\!D]$. A necessary condition for this state to exist is that it satisfies $C_5$ is that A turns ON within $I_3$. This occurs if and only if $c\!-\!bs_B(T\!R\!-\!D)\!<\!\theta$ and $c\!-\!bs_B(T\!R\!+\!T\!D\!-\!D)\!\geq\!\theta$. This is equivalent to $\exists t_2 \in I_3^0$ (the open set) such that $c\!-\!bs_B(t_2\!-\!D)\!=\!\theta$. From the analytic solution of $s_B(t_2\!-\!D)$ we can solve this equation and obtain $t_2\!=\!t_1\!+\delta(t_1)\!+\!Q$, where $Q\!=\!2D\!-\!\log((c\!-\!\theta)/b)$ is a constant. By repeating this process across subsequent the periodic intervals $I_k\!=\![2(k\!-\!1)T\!R,2(k\!-\!1)T\!R\!+\!T\!D]$ we obtain that the $k$ turning ON time for the A unit is given by the map 
  \begin{equation}
  t_{k+1}\!=\!t_k\!+\!\delta(t_k)\!+\!Q.
  \label{map_ZcSL}
  \end{equation}
  Since we are interested in the limit $\tau\!\rightarrow\!0$ and on $T\!R$-periodic solution it must be that $2T\!R\!=\!Q$. However, assuming true this condition and $\tau\!>\!0$ arbitrarily small, this map shows that the A unit turns ON with after a small delay $\delta$ across subsequent intervals $I_k$ (ie the map has no fixed point). Therefore, $ZcIL$ cannot exist.
  \item $ScASL_2$ for $s^*\!<\!t^*$ - The proof is analogous to the case above ($ZcSL$) after swapping the A and B units. Briefly, if the B unit turns ON at time $t_1 \in [0,T\!D]$ the A unit turns OFF at time $t_1\!+\!D\!+\!\delta(t_1)$, where $\delta(t_1)\!\sim\!\tau$. This means that $s_A$ evolves according to $s_A(t\!-\!D)\!=\!e^{-(t\!-t_1\!-\!2D\!-\!\delta(t_1))/\tau_i}$, for $t \in I_3\!=\![2T\!R,2T\!R\!+\!T\!D]$. The $k$ turning ON time for the B unit is given by the map \ref{map_ZcSL}, where $Q\!=\!2D\!-\!\log((d\!-\!\theta)/b)$. As in the previous case, $ScASL_2$ cannot exist because this map has no fixed point unless $\tau\!=\!0$.
  \item $ScASL_2$ - This state satisfies condition $C_5$ in both intervals $I_1$ and $I_2$. Let us call $t_1$ and $s_1$ the turning ON times for A and B in $I_1$ respectively, and $t_2$ and $s_2$ the turning ON times for B and A in $I_2$ respectively. On the slow time scale both (delayed) synaptic variables $s_A$ and $s_B$ evolve equally in $I_1$ and $I_2$, because both units turn OFF instantaneously and at the same time in $\mathbb{R}\!-\!I$ for Lemma \ref{lem:LONG_conditions}. Since $d\!\leq\!c$ the total input to A is greater than the one to B for $t \in I_1$, ie $c\!-\!bs_B(t\!-\!D)\!\geq\!d\!-\!bs_A(t\!-\!D)$, which leads to $t_1\!\leq\!s_1$. Analogous considerations lead to $t_2\!\leq\!s_2$. Moreover, it turns out that $t_1\!=\!t_2\!-\!T\!R$. Indeed, WLOG suppose that $t_2\!-\!T\!R\!>\!t_1$. Since $c\!-\!bs_B(t_1\!-\!D)\!=\!\theta$ and due to the monotonic decay of the delayed synaptic variables in $I_1$ we must have $c\!-\!bs_B(t_2\!-\!T\!R\!-\!D)\!\geq\!\theta$. Moreover, since $c\!-\!bs_A(t_2\!-\!D)\!=\!\theta$ we have that $s_B(t_2\!-\!T\!R\!-\!D)\!<\!s_A(t_2\!-\!D)$. Since A turns ON at time $t_1$ the B unit turns OFF at time $t_1\!+\!D\!+\!\delta(t_1)$, where $\delta(t_1)\!\sim\!\tau$. Thus $s_A(t\!-\!D)\!=\!e^{-(t\!-t_1\!-\!2D\!-\!\delta(t_1))/\tau_i}$, for $t \in I_2$. Similarly, since B turns ON at time $t_2$ and for the $2T\!R$-periodicity we have that $s_B(t\!-\!D)\!=\!e^{-(2T\!R\!+t\!-t_2\!-\!2D\!-\!\delta(t_1))/\tau_i}$, for $t \in I_1$. On the slow time scale ($\tau\!\rightarrow\!0$) these identities evaluated at time $t_2$ imply $s_B(t_2\!-\!T\!R\!-\!D)\!=\!e^{(-\!T\!R\!+\!2D)/\tau_i}$ and $s_A(t_2\!-\!D)\!=\!e^{(t_1\!-\!t_2\!+\!2D)/\tau_i}$. Due to the hypothesis $t_2\!-\!T\!R\!>\!t_1$ the latter lead to $s_B(t_2\!-\!T\!R\!-\!D)\!\geq\!s_A(t_2\!-\!D)$, which is absurd. Therefore we have that $t_1\!=\!t_2\!-\!T\!R$. This in turn leads to $s_B(t\!-\!D)\!=\!e^{-(t\!-t_1\!-T\!R\!-\!2D\!-\!\delta(t_1))/\tau_i}$, for $t \in I_3\!=\![2T\!R,2T\!R\!+\!T\!D]$. Due to the $2T\!R$ periodicity the second turning ON time for A (after $t_1)$ must be at a time $t_3\!=\!t_1\!+\!2T\!R \in I_3$ such that $c\!-\!bs_B(t_3\!-\!D)\!=\!\theta$. From the analytic solution of $s_B(t_2\!-\!D)$ we obtain $t_3\!=\!t_1\!+\delta(t_1)\!+\!Q$, where $Q\!=\!T\!R\!+\!2D\!-\!\log((c\!-\!\theta)/b)$ is a constant. Thus the $k$ turning ON time for the A unit is given by the map \ref{map_ZcSL}. Due to the dependence on $\tau$, this map has no fixed point unless $\tau\!=\!0$, thus proving that $ScASL_2$ cannot exist. 
\end{enumerate}

\subsection{$2T\!R$-periodic states for $D\!<\!T\!D$ and $T\!D\!+\!D\!<\!T\!R$ and $a\!+\!d\!-\!b\!<\!\theta$} \label{appendix4}
\begin{theorem}
 Let us now consider $2T\!R$-periodic states for $D\!<\!T\!D$, $T\!D\!+\!D\!<\!T\!R$ and $a\!+\!d\!-\!b\!<\!\theta$, and define $L_1\!=\![0,D]$ and $L_2\!=\![T\!R,T\!R\!+\!D]$. The synaptic quantities defining the entries of the matrix form in $L_1$ and $L_2$ are given by
 \begin{equation}
  s_A^{2 \pm}\!=\!s_B^{1 \pm}\!=\!N^{\pm}, \quad
  s_A^{1 \pm}\!=\!\begin{cases} 
	R^{\pm} & \mbox{if } z_A^2=1 \\ 
	M^{\pm}  & \mbox{otherwise} 
  \end{cases}
  \quad \mbox{and} \quad
  s_B^{2 \pm}\!=\!\begin{cases} 
	R^{\pm} & \mbox{if } z_B^1=1 \\ 
	M^{\pm}  & \mbox{otherwise} 
  \end{cases}
 \end{equation}
 Where $R^-\!=\!e^{-(T\!R\!-\!2D)/\tau_i}$ and $R^+\!=\!e^{-(T\!R\!-\!D)/\tau_i}$. Quantities $M^\pm$ and $N^\pm$ were defined in equations \ref{syn_values}. 
\end{theorem}
 \begin{proof}
 Since A (B) is ON in $[0,T\!D]$ ($[0,T\!R\!+\!T\!D]$) and turn OFF instantaneously at time $T\!D$ ($T\!R\!+\!T\!D$) due to property \ref{U3}. The synaptic variable $s_A$ ($s_B$) thus exponentially decays on the slow time scale starting from time $T\!D$ ($T\!R\!+\!T\!D$) and ending at time $T\!R$ ($2T\!R$). Due to this and to the $2T\!R$-periodicity the delayed synaptic variable $s_A(t\!-\!D)$ ($s_B(t\!-\!D)$) evaluated at times $T\!R$ and $T\!R\!+\!T\!D$ ($0$ and $T\!D$) are equal to $N^\pm$, which proves the first identity of the theorem. If $z_A^2\!=\!1$ the A unit is ON in $L_2$ and turns OFF instantaneously at time $T\!R\!+\!D$ for both MAIN or CONNECT states. Thus the synaptic variable $s_A$ slowly decays starting from time $T\!R\!+\!D$ until the A unit turns ON at time $2T\!R$. This implies $s_A(t\!-\!D)\!=\!e^{-(t\!-\!T\!R\!-\!2D)}$, for $t\!\in\![2T\!R,2T\!R\!+\!D]$. The $2T\!R$-periodicity leads to $s_A^{1-}\!=\!s_A(2T\!R\!-\!D)\!=\!e^{-(T\!R\!-\!2D)/\tau_i}$ and $s_A^{1+}\!=\!s_A(2T\!R)\!=\!e^{-(T\!R\!-\!D)/\tau_i}$, proving the second identity of the theorem. The proof of the third identity is analogous to the previous one. 
 \end{proof}

\subsection{Case $T\!D\!+\!D\!\geq\!T\!R$} \label{appendix_TD_D_gr_TR}
The condition $T\!D\!+\!D\!<\!T\!R$ enabled us to obtain a complete classification of network states via the application of Lemma \ref{lem:syn_decay}. Each of these states (except for $AScI$) has the same existence conditions given in the two tables also for $T\!D\!+\!D\!\geq\!T\!R$ with two adjustments. More precisely, if $T\!D\!+\!D\!\geq\!T\!R$ and $2D\!<\!T\!R$ we must replace the quantity $N^-$ with unity in the existing condition $C_7^-$ in Table \ref{tab:main_case2_1}. If $2D\!\geq\!T\!R$ (which implies $T\!D\!+\!D\!\geq\!T\!R$) we must replace the quantity $R^-$ with unity in the existing $R_6^-$ and $R_7^-$ in Table \ref{tab:MAIN_CONNECT_case2}). This is valid for all states except for $AScI$, for which we additionally need to impose that the turning ON time $t^*$ for the B unit in $[0,D]$, or equivalently the turning ON time $T\!R\!+\!t^*$ for the A unit in $[T\!R,T\!R\!+\!D]$, satisfies 
$$ t^*+D < T\!R .$$
Where $t^*$ is given by the solution of $a-be^{-(T\!R\!-\!D-\!t^*)/\tau_i}\!+\!d\!=\!\theta$. We now proof that state $AScI$ cannot exist if $t^*\!+\!D\!\geq\!T\!R$ and $D\!<\!T\!D$, where $t^* \in [0,D]$ ($t^*\!+\!T\!R \in [T\!R,T\!R\!+\!D]$) is the turning ON time for the B (A) unit in the interval $I_1$ ($I_2$). We need to show that the B (A) unit cannot be OFF for $t\!<\!t^*$ ($t\!<\!t^*+\!T\!R$) and ON for $D\!\geq\!t\!>\!t^*$ ($D+\!T\!R\!\geq\!t\!>\!t^*+\!T\!R$). By absurd suppose the contrary. We now determine the dynamics of the the delayed synaptic variable $s_B(t\!-\!D)$ during the interval $[T\!R,T\!R\!+\!D]$. The B unit turns ON at time $t^*\!\geq\!T\!R\!-\!D$ and is ON in $I_2$ (due to properly \ref{U3}). These properties and the $2T\!R$-periodicity of $AScI$ imply that $s_B(t\!-\!D)$ evolves according to 
$$ s_B(t\!-\!D)=e^{-(T\!R+t-T\!D-D))/\tau_i}, \quad \forall t\!\in\![T\!R,T\!R\!+\!t^*).$$
Evaluating this equation at time $t_1\!=\!T\!R$ leads to $s_B(t_1\!-\!D)\!=\!e^{-(2T\!R-T\!D-D))/\tau_i}$. Secondly we have that 
$$ s_B(t\!-\!D)=e^{-(t-2D))/\tau_i}, \quad \forall t\!\in\!(2D,T\!R\!+\!D].$$
Evaluating this equation at time $t_2\!=\!T\!R\!+\!D$ leads to $s_B(t_2\!-\!D)\!=\!e^{-(T\!R\!-\!D))/\tau_i}$. This implies:  
$$ s_B(t_1\!-\!D) \leq s_B(t_2\!-\!D).$$
However by hypothesis A is OFF at time $t_1\!<\!t^*\!+\!T\!R$ and ON at time $t_2\!>\!t^*\!+\!T\!R$, i.e. t$a+d-bs_B(t_1\!-\!D) \!<\! \theta$ and $a+d-bs_B(t_2\!-\!D) \!\geq\! \theta$, which is absurd. 


\end{document}